\documentclass[12pt,a4paper,leqno]{amsart}

\usepackage[applemac]{inputenc}
\usepackage{amsmath}
\usepackage{amsthm}
\usepackage{amscd}
\usepackage{amsfonts}
\usepackage{amssymb}
\usepackage{mathrsfs}
\usepackage{geometry}
\usepackage{enumerate}
\usepackage{xcolor}
\usepackage{eucal}
\usepackage{amsmath, amstext, amsbsy, amssymb, amscd, mathrsfs}
\usepackage{amsmath}
\usepackage{amsxtra}
\usepackage{amscd}
\usepackage{amsthm}
\usepackage{amsfonts}
\usepackage{amssymb}
\usepackage{eucal}
\usepackage{color}
\usepackage[all]{xy}
\usepackage{bbding}
\usepackage{pmgraph}
\usepackage{stmaryrd}
\usepackage{graphicx}
\usepackage{bm}
\usepackage[enableskew,vcentermath]{youngtab}
\usepackage{tikz}
\usepackage{graphicx}
\usepackage{multicol}
\newgeometry{top=2.2cm,bottom=1.78cm,left=2cm,right=2cm}
\newtheorem{thm}{Theorem}[section]

\newtheorem{cor}[thm]{Corollary}
\newtheorem{lem}[thm]{Lemma}
\newtheorem{prop}[thm]{Proposition}

\numberwithin{equation}{thm}

\theoremstyle{definition}
\newtheorem{defn}[thm]{Definition}
\newtheorem{rem}[thm]{Remark}
\newtheorem{rems}[thm]{Remarks}

\newtheorem{subs}[thm]{}

\newcommand{\D}{\mathcal{D}}
\newcommand{\fsG}{\mathscr{G}}

\newcommand{\US}{\boldsymbol{\mathfrak{A}}}

\newcommand{\fcY}{\mathcal{Y}}

\def\f#1{\mathfrak{#1}}
\newcommand{\fS}{\f{S}}
\newcommand{\C}{\mathcal{C}}
\def\MN(#1){ M_{#1}(\mathbb{N})}
\def\MNR(#1,#2){ M_{#1}(\mathbb{N})_{#2}}
\def\MNS(#1){ M_{#1}(\mathbb{N})^{\pm}}

\newcommand{\ZZ}{\mathbb{Z}}
\newcommand{\QQ}{\mathbb{Q}}
\newcommand{\NN}{\mathbb{N}}

\newcommand{\ZG}{{{\mathbb{Z}}}_2}
\newcommand{\Ng}{{{\mathbb{N}}}_2}
\def\MZ(#1){ M_{#1}(\ZG)}
\def\NZ(#1){ (\NN|\Ng)^{#1}}
\def\NZST(#1,#2,#3){ (\NN|\Ng)^{#1}_{#2|#3}}
\def\NZS(#1,#2){ {\NN}^{#1}_{#2}}
\def\MNZ(#1,#2){ M_{#1}(\NN | \Ng)_{#2}}
\def\MNZN(#1){M_{#1}(\NN | \Ng)}
\def\CMNZ(#1,#2,#3){\Lambda(#1,#2|#3)}
\def\CMN(#1,#2){\Lambda(#1,#2)}
\def\CMNP(#1,#2){\Lambda_{#1,#2}}
\def\MNZNS(#1){M_{#1}^{\pm}(\NN | \Ng)}
\def\MNZNp(#1){M_{#1}^+(\NN | \Ng)}
\def\MNZNz(#1){M_{#1}^0(\NN | \Ng)}
\def\MNZNm(#1){M_{#1}^-(\NN | \Ng)}
\def\SE#1{{#1}^{\bar{0}}}
\def\SO#1{{#1}^{\bar{1}}}
\def\SEE#1{{#1}^{\bar{0}}}
\def\SOE#1{{#1}^{\bar{1}}}
\def\SUP#1{\SE{#1}|\SO{#1}}
\def\SS(#1,#2){{#1}^{\ol{#2}}}
\def\SSE(#1,#2){{#1}^{\ol{#2}}}

\def\ESE(#1, #2, #3){ \SEE{#1}_{{#2},{#3}} }
\def\ESO(#1, #2, #3){ \SOE{#1}_{{#2},{#3}} }
\def\bs#1{\boldsymbol{#1}}

\def\Qqs(#1,#2){ \mathcal{Q}(#1,#2) }

\def\lcase#1{\MakeLowercase{#1}}

\newcommand{\HCR}{\mathcal{H}^c_{r, R}}
\newcommand{\Heck}{\mathcal{H}_{r,R}}

\newcommand{\QqnrR}{\mathcal{Q}_{\lcase{q}}(\lcase{n},\lcase{r}; R)}

\newcommand{\SQqnrR}{{\mathcal{Q}^s_{\lcase{q}}(\lcase{n},\lcase{r}; R)}}
\newcommand{\SQvnR}{\boldsymbol{\mathcal{Q}}_{\up}{(\lcase{n})}}
\newcommand{\SQvnQ}{\boldsymbol{\mathcal{Q}}_{\up}{(\lcase{n})}}

\newcommand{\bsSQvnr}{\boldsymbol{\mathcal{Q}}^s_{\up}(\lcase{n},\lcase{r})}
\newcommand{\bsSQvn}{\boldsymbol{\mathcal{Q}}^s_{\up}(\lcase{n})}

\newcommand{\USnv}{{\US_{v}(n)}}
\def\ol#1{\overline{#1}}

\newcommand{\ep}{\epsilon}
\newcommand{\Qv}{\mathbb Q({v})}

\def\normc{\mathfrak{o}}

\def\Qvs(#1){\mathcal{Q}_{\lcase{v}}{(\lcase{#1})}}
\def\Uvqn{{\boldsymbol U}_{\!{v}}(\mathfrak{\lcase{q}}_{n})}

\def\Uvq(#1){U_{\lcase{v}}(\mathfrak{\lcase{q}}_{\lcase{#1}})}

\def\USN(#1){{\US[#1]}_{v}}

\def\SABJR(#1,#2,#3,#4){({#1}|{#2})[\bs{#3}, #4]}
\def\SABJRS(#1,#2,#3,#4){({#1}|{#2})[#3, #4]}
\def\SABJS(#1,#2,#3){({#1}|{#2})[#3]}
\def\SAJRS(#1,#2,#3){{#1}[#2, #3]}
\def\SAJS(#1,#2){{#1}[#2]}
\def\SABJ(#1,#2,#3){({#1}|{#2})[\bs{#3}]}
\def\SAJR(#1,#2,#3){{#1}[\bs{#2}, #3]}
\def\SAJ(#1,#2){{#1}[\bs{#2}]}
\def\ABJR(#1,#2,#3,#4){({#1}|{#2})(\bs{#3}, #4)}
\def\ABJRS(#1,#2,#3,#4){({#1}|{#2})(#3, #4)}
\def\ABJS(#1,#2,#3){({#1}|{#2})(#3)}
\def\AJRS(#1,#2,#3){{#1}(#2, #3)}
\def\AJS(#1,#2){{#1}(#2)}
\def\ABJ(#1,#2,#3){({#1}|{#2})(\bs{#3})}
\def\AJR(#1,#2,#3){{#1}(\bs{#2}, #3)}
\def\AJ(#1,#2){{#1}(\bs{#2})}

\def\snorm#1{|{#1}|}
\def\STDUE(#1,#2){({#1}+E_{{#2},{#2}+1}-E_{{#2}+1,{#2}+1}|0)}
\def\STDUO(#1,#2){({#1}-E_{{#2}+1,{#2}+1}|E_{{#2},{#2}+1})}
\def\STDLE(#1,#2){({#1}-E_{{#2},{#2}}+E_{{#2}+1,{#2}}|0)}
\def\STDLO(#1,#2){({#1}-E_{{#2},{#2}}|E_{{#2}+1,{#2}})}
\def\STDDE(#1){(D_{#1}|0)}
\def\STDDO(#1,#2){({#1}-E_{{#2},{#2}}|E_{{#2},{#2}})}

\newcommand{\tspan}{\mathrm{span}}

\newcommand{\Tr}{\mathscr{D}\!spt}

\newcommand{\End}{\mathrm{End}}

\newcommand{\ro}{\mathrm{ro}}
\newcommand{\co}{\mathrm{co}}

\def\STEPX#1#2{{[\![{#1}]\!]}_{#2}}

\def\STEP#1{ {[\![{#1}]\!]}_{{q}} }
\def\STEPP#1{{[\![{#1}]\!]}_{{q}^2}}
\def\STEPPD#1{{[\![{#1}]\!]}_{{q},{q}^2}}
\def\STEPPDR#1{{[\![{#1}]\!]}_{{q}^2,{q}}}

\newcommand{\where}{\ \bs{|} \ }

\def\Hom{\mathrm{Hom}}

\def\wt{\mathrm{wt}}

\newcommand{\spaceintv}{}
\def\intd(#1,#2,#3){\left[\begin{matrix}{#1};{#2}\\{#3}\end{matrix}\right]}
\def\intds(#1,#2){\left[\begin{matrix}{#1}\\{#2}\end{matrix}\right]}
\def\intdss(#1,#2){\intd({#1},{0},{#2})}
\def\diag{{\rm{diag}}}
\def\ker{\rm{ker}}

\def\parity#1{\wp({#1})}

\def\TAIJ(#1,#2){ T^\lhd_{({#1}, {#2})} }
\def\TDIJ(#1,#2){ T^\rhd_{({#1}, {#2})} }
\def\rmText#1{}
\def\rmForm#1{}

\def\AK(#1,#2){ {\overleftarrow{\bf r}}^{#2}_{#1} }
\def\BK(#1,#2){ {\overrightarrow{\bf r}}^{#2}_{#1} }
\def\Adlk(#1,#2,#3){{}^{\ }_{#1}\!A^{#2}_{#3}}

\newcommand{\ft}{{\boldsymbol{\mathrm{t}}}}

\def\sdpNHk{{}_\text{\sc sdp}{\rm {H}}}
\def\sdpNHk{{}_\text{\sc sdp}{\rm {H}}\overline{\textsc{k}}}
\def\sdpNHe{{}_\text{\sc sdp}{\rm {H}}\overline{\textsc{e}}}
\def\sdpNHf{{}_\text{\sc sdp}{\rm {H}}\overline{\textsc{f}}}

\def\sdpCHk{{\rm {H}}\overline{\textsc{k}}}
\def\sdpCHe{{\rm {H}}\overline{\textsc{e}}}
\def\sdpCHf{{\rm {H}}\overline{\textsc{f}}}
\def\sdpCHHf{{\rm {HH}}\overline{\textsc{f}}}

\def\fkf{{\mathfrak f}}
\def\up{{\upsilon}}
\def\bsU{\boldsymbol U}
\def\sfK{\mathsf K}
\def\sfE{\mathsf E}
\def\sfF{\mathsf F}

\def\RZ{\hlt{{\mathsf{E}}}}

\def\genE{\hlt{\mathsf{E}}}
\def\genF{\hlt{\mathsf{F}}}
\def\genK{\hlt{\mathsf{K}}}

\def\RZ{{{\mathsf{E}}}}

\def\genE{{\mathsf{E}}}
\def\genF{{\mathsf{F}}}
\def\genK{{\mathsf{K}}}

\def\Ad{{A^{\!\star}}}
\def\Bd{{B^{\!\star}}}
\def\Cd{{C^{\!\star}}}

\def\Md{{M^{\!\star}}}

\def\Dd{D^{\star}}

\def\Ed{E^{\star}}
\def\Fd{F^{\star}}

\def\qSchvsZ{\mathcal Q^s_\up(n,r)}
\def\qSchvsQ{\boldsymbol{\mathcal Q}^s_\up(n,r)}

\def\llcm{A_{\text{\normalsize$\urcorner$}}}

\def\fkM{\mathfrak M}
\def\la{{\lambda}}
\def\Od{O^{\star}}
\def\Norm{\partial}
\def\OG{g^-}
\def\Og{g}

\def\preeq{\approx}
\def\ndelta{\dot{\delta}}
\def\sH{{\mathcal H}}
\def\sZ{{\mathcal Z}}
\def\bsal{{\alpha}}

\def\scc{\textsc{c}}
\def\scp{{\textsc{p}}}

\def\Ahkzp{A^{\ol{0},+}_{h,k}}
\def\Ahkzm{A^{\ol{0},-}_{h,k}}
\def\Ahkop{A^{\ol{1},+}_{h,k}}
\def\Ahkom{A^{\ol{1},-}_{h,k}}
\def\Ahhop{A^{\ol{1},+}_{h,h}}
\def\Ahhiop{A^{\ol{1},+}_{h,h+1}}
\def\Ahhom{A^{\ol{1},-}_{h,h}}
\def\Ahhiom{A^{\ol{1},-}_{h,h+1}}
\def\Ahkomm{A^{\ol{1},\approx}_{h,k}}
\def\Ahhomm{A^{\ol{1},\approx}_{h,h}}
\def\Ahhiomm{A^{\ol{1},\approx}_{h,h+1}}
\allowdisplaybreaks[1]
\def\lc{\mathsf{lc}}
\def\bfm{\mathbf m}

\def\ttb{\mathsf{B}}
\def\ttm{\mathsf{M}}
\def\scm{\textsc{m}}
\def\scrR{\mathscr R}
\def\sfM{\mathsf M}
\def\wt{\texttt{wt}}

\title[The quantum queer supergroup]
{Constructing the quantum queer supergroup using Hecke-Clifford superalgebras}

\author{Jie Du, Haixia Gu, Zhenhua Li and Jinkui Wan}

\address{Jie Du, School of Mathematics, University of New South Wales, UNSW Sydney 2052, Australia}
\email{j.du@unsw.edu.au}
\address{Haixia Gu, School of Science, Huzhou University, Huzhou 313000, China}
\email{ghx@zjhu.edu.cn}
\address{Zhenhua Li, School of Mathematical Sciences, Xiamen University, Xiamen 361005, China}
\email{zhen-hua.li@qq.com}
\address{Jinkui Wan,  School of Mathematics and Statistics, Beijing Institute of Technology, Beijing 100081, China}
\email{wjk302@hotmail.com}
\date{\today}
\begin{document}
\maketitle

\begin{abstract}
In \cite{DGLW}, we use certain special elements and their commutation relations in the Hecke-Clifford algebras  $\HCR$ to derive some fundamental multiplication formulas associated with the natural bases in queer $q$-Schur superalgebras $\mathcal Q_q(n,r;R)$ introduced in \cite{DW2}. Here a natural basis element is defined by a special element $T_\Ad$ in $\HCR$ associated with a pair of certain $n\times n$ matrices $\Ad=(A^{\bar0}|A^{\bar1})$ over $\NN$ with entries sum to $r$. The definition of $T_\Ad$ consists of  an element $c_\Ad$ in the Clifford superalgebra and an element $T_A$ in the Hecke algebra, where $A=A^{\bar0}+A^{\bar1}$. Note that all $T_A$ can be used to define the natural basis for the corresponding $q$-Schur algebra ${\mathcal S}_q(n,r)$.

This paper is a continuation of \cite{DGLW}. We start with standardized queer $\up$-Schur superalgebras $\mathcal
Q^s_\up(n,r)$, for $R=\ZZ[\up,\up^{-1}]$ and $q=\up^2$, and their natural bases. With the $\up$-Schur algebra ${\mathcal S}_\up(n,r)$ at the background, the first key ingredient is a standardisation of the natural basis for $\mathcal
Q^s_\up(n,r)$ and their associated standard multiplication formulas.  By introducing some long elements of finite sums, we then extend the formulas to these long elements which allow us to explicitly define $\QQ(\up)$-superalgebra homomorphisms $\bs\xi_{n,r}$ from the quantum queer supergroup $\Uvqn$ to queer $q$-Schur superalgebras $\boldsymbol{\mathcal Q}^s_\up(n,r)$, for all $r\geq1$. Finally, taking limits of long elements yields certain infinitely long elements as formal infinite series which eventually lead to a new construction for $\Uvqn$.

\end{abstract}

\maketitle


\tableofcontents

\section{Introduction}\label{sec_introduction}

\begin{subs}
{\bf Finite algebra approach to quantum groups.}
Since the introduction of quantum groups as quantized enveloping algebras
 of semisimple Lie algebras in late eighties of the last century, their realizations and related
applications have achieved outstanding progress. First, C. M. Ringel developed a realization
for the positive part of such a quantum group via Ringel--Hall algebras
associated with the representation category of finite quiver (or path) algebras.
Then G. Lusztig's geometric approach to quantum groups and canonical bases advanced the theory to a new level.

On the other hand, almost at the same time,
Beilinson, Lusztig and MacPherson \cite{BLM} initiated a new method
of using a sequence of geometrically defined convolution algebras associated
with some (finite) partial flag varieties (i.e., finite $q$-Schur algebras) to approach quantum linear groups $\mathbf U=\bsU_{\!\up}(\mathfrak{gl}_n)$.
In their work, certain structure constants relative to the orbital basis satisfy a stabilisation property. This property results in the definition of an infinite dimensional idempotented algebra---the modified quantum group $\dot{\mathbf U}$.
Then, by a process of taking limit, it
 results in a new realization for the entire quantum group $\bsU_{\!\up}(\mathfrak{gl}_n)$
via its regular representation defined explicitly by
some multiplication formulas of a basis by generators.

The convolution algebra approach was soon generalized to quantum affine $\mathfrak{gl}_n$ in \cite{GV, L},
which motivated, with a modification of the original BLM approach in \cite{DF1}, an algebraic approach to a new realization of the quantum loop algebra of $\mathfrak{gl}_n$; see \cite{DDF, DF2}.
Recently,  Bao and Wang \cite{BW} investigated $i$-quantum groups $\boldsymbol U^\jmath$ and $\boldsymbol U^\imath$ (or more precisely, certain quantum symmetric pairs of type AIII) to give a reformulation of Kazhdan-Lusztig theory that provides a
perfect solution to the problem of character formulas of $\mathfrak{osp}$ Lie superalgebras. Motivated from the Bao-Wang's work,
the finite convolution algebra approach has also been generalized to
the type B/C geometry and to the modified $i$-quantum groups and their canonical bases
in \cite{BKLW}. Thus, the new Schur--Weyl duality in \cite{BW} further motivates
BLM type realizations for the above mentioned twin $i$-quantum groups; see \cite{DWu1,DWu2}.
\end{subs}

\begin{subs}{\bf A super construction $\bs\sH_r\rightsquigarrow\bsU_\up(\mathfrak{gl}_{m|n})$.}
Though the geometric (or convolution algebra) approach seems not available for  quantum supergroups,
 the idea of using finite dimensional superalgebras to approach quantum supergroups
continues to shed lights on the development of the algebraic approach. For example, the work \cite{DGZ} shows that
the desired fundamental multiplication formulas can be rooted out from the associated Hecke algebras (see \cite{DG, DGZ}) via the quantum Schur--Weyl duality.  In other words, these multiplication formulas can be derived by certain special elements and commutation relations in the associated Hecke algebras. Thus,  the new construction for $\mathbf U_\up(\mathfrak{gl}_n)$ or the quantum supergroup $\mathbf U_\up(\mathfrak{gl}_{m|n})$ can be built on some structures in the Hecke algebras. This tells that the structure of these quantum groups is actually hidden in the structure of Hecke algebras.

More precisely,
the Hecke algebra $\sH=\sH_{r,R}$ of the symmetric group $\fS_r$ over a commutative ring $R$ has the following fundamental structure: it is generated by $T_{s_1},\ldots, T_{s_{r-1}}$ and has basis $\{T_w\mid w\in\fS_r\}$ with respect to which we have the following matrix representation of the regular module $_{\sH}\sH$:
$$T_{s_i}T_w=\begin{cases}T_{s_iw},&\text{ if }s_iw>w;\\
(q-1)T_w+qT_{s_iw},&\text{ if }s_iw<w.\end{cases}$$
Building on the aforementioned special elements and commutation relations, one may construct a new basis  $\{A({\bs{j}})\}_{A,{\bs{j}}}$ for the quantum linear supergroups $\bsU=\bsU_{\!\up}(\mathfrak{gl}_{m|n})$, containing its generators $K_j,E_h,F_h$,
 such that
 the structure constants relative to the basis appearing in
 $$K_jA({\bs{j}}),\;\; E_hA({\bs{j}}),\;\;F_hA({\bs{j}})$$
 can be explicitly computed; see \cite[Lem.~5.3]{BLM} for the quantum linear group ($n=0$) case and \cite[Cor.~4.4]{DGZ} and \cite[Prop.~5.5-6, Th.~8.4]{DG} in general. In other words, this construction, symbolised as $\bs\sH_{r}\rightsquigarrow\bsU_\up(\mathfrak{gl}_{m|n})$ with $\bs\sH_{r}=\sH_{r,\QQ(\up)}$, shows that the structure of $\bsU_\up(\mathfrak{gl}_{m|n})$ is hidden in that of $\bs\sH_{r}$. It is natural to ask if there exists a similar hidden structural connection $\bs\sH_{r}^c\rightsquigarrow\bsU_{\!\up}(\mathfrak{q}_{n})$ for the quantum queer supergroup $\Uvqn$ and Hecke-Clifford superalgebras $\bs\sH_{r}^c:=\sH^c_{r,\QQ(\up)}$.

\end{subs}

\begin{subs}{\bf A roadmap for $\bs\sH_{r}^c\rightsquigarrow\bsU_{\!\up}(\mathfrak{q}_{n})$.}
 In this paper, we give an affirmative answer to the question. By displaying a construction $\bs\sH_{r}^c\rightsquigarrow\bsU_{\!\up}(\mathfrak{q}_{n})$ that presents the quantum queer supergroup $\bsU_{\!\up}(\mathfrak{q}_{n})$ via Hecke-Clifford algebras $\bs\sH_{r}^c$, we show
 that the structure of the quantum queer supergroup $\bsU_{\!\up}(\mathfrak q_n)$ is hidden in the Hecke--Clifford superalgebra $\boldsymbol\sH_{r}^c$.

  We now describe a roadmap for the construction $\bs\sH_{r}^c\rightsquigarrow\bsU_{\!\up}(\mathfrak{q}_{n})$ in the following steps\footnote{Note that the geometric approach initiated in \cite{BLM} is a shortcut to the construction $\bs{\sH}_r\rightsquigarrow\bsU(\mathfrak{gl}_{m|n})$, where steps (1)--(3) are combined in terms of convolution algebras and convolution product.}:
\begin{enumerate}
\item Constructing some special elements and
  commutation formulas  in the Hecke--Clifford superalgebra $\HCR$.
\item Defining queer $q$-Schur superalgebras $\mathcal Q_q(n,r)$ and its natural basis $\{\phi_\Ad\}_\Ad$.
\item Using (1) to compute some fundamental multiplication formulas (MFs) in $\mathcal Q_q(n,r)$.
\item Standardizing everything: $\mathcal Q^s_\up(n,r)$, standard basis $[\Ad]$, and new MFs.
\item Introducing long elements $\{\Ad(\bs j,r)\}_{\Ad,\bs j}$ and deriving their multiplication formulas in $\boldsymbol{\mathcal Q}^s_\up(n,r)$ to establish homomorphisms $\bs{\xi}_{n,r}:\bsU_{\!\up}(\mathfrak q_n)\rightarrow \boldsymbol{\mathcal Q}^s_\up(n,r)$ for $r\geq 1$.
\item Establishing triangular relations to determine the image of $\bs\xi_n:\bsU_{\!\up}(\mathfrak q_n)\rightarrow \prod_{r\geq 1}\boldsymbol{\mathcal Q}^s_\up(n,r)$.
\item Completing the construction by proving the injectivity of $\bs\xi_n$.
\end{enumerate}
Note that steps (1)--(3) have been completed in \cite{DW2} and \cite{DGLW}. Building on \cite{DGLW}, we complete the remaining steps in this paper.

The following theorem is the main result of the paper which gives a new basis for the quantum queer supergroup and their structure constants for generators.
Let $\ep_1,\ep_2,\ldots,\ep_n$ denote the standard basis for $\ZZ^n$ and let
$$\bsal_i=\ep_i-\ep_{i+1}\text{ and }\bsal^+_i=\ep_i+\ep_{i+1}.$$
Then $\bsal_i$, for $i\in[1,n-1]$ are known as the ``simple roots''.

\end{subs}
\begin{subs} \label{mthm}{\bf Main Theorem.} \it
The quantum queer supergroup ${\bsU}_{\!\up}(\mathfrak{q}_{n})$, generated by $\sfE_{h},\sfF_{h}, \sfK_i,\sfK_i^{-1}, \sfK_{\ol{n}}$, for $h,i\in[1,n],h\neq n$, has a basis\vspace{-1ex}
$$\{\Ad({\bs{j}})\mid \Ad=(a_{i,j}^{\ol{0}}|a_{i,j}^{\ol{1}})\in M_n({\mathbb N}|\Ng)^{\pm}, {\bs{j}}=(j_k)\in\mathbb Z^{n}\}\vspace{-1ex}$$
that satisfies the following multiplication rules:%
$$\begin{aligned}
(1)\quad&\sfK_i^{\pm1}\cdot \Ad(\bs{j})
	= {v}^{ \pm  \sum_{u=1}^n a_{h,u} } \Ad(\bs{j} \pm\ep_i),\\
(2)\quad &\sfE_h \cdot \Ad(\bs{j})=\\
&\sum_{k<h}
	{\up}^{f_{h,k}^{\ol{0}}}   [ \SEE{a}_{h,k} + 1] (\Ahkzp | \SO{A})(\bs{j}+\bsal_h)+\sum_{h+1<k}
	{\up}^{f_{h,k}^{\ol{0}}}   [ \SEE{a}_{h,k} + 1] (\Ahkzp | \SO{A})(\bs{j})\\
&\quad+ \frac{{\up}^{f_{h,h}^{\ol{0}}-j_h}}{\up-\up^{-1}}\left\{(\SE{A} - E_{h+1, h} | \SO{A}) (\bs{j}+\bsal_h)
-(\SE{A} - E_{h+1, h} | \SO{A}) (\bs{j}-\bsal_h^+)\right\}\\
&\quad+{\up}^{f_{h,h+1}^{\ol{0}}+j_{h+1}}   [ \SEE{a}_{h,h+1} + 1]  (\SE{A} + E_{h,h+1} | \SO{A})(\bs{j})\\
+\sum_{k<h}&{\up}^{f_{h,k}^{\ol{1}}}(\SE{A} | \SO{A} - E_{h+1, k}  + E_{h,k})(\bs{j}+\bsal_h)+{\up}^{f_{h,h}^{\ol{1}}}(\SE{A} | \SO{A} - E_{h+1, h}  + E_{h,h})(\bs{j}-\ep_{h+1})\\
+&{\up}^{f_{h,h+1}^{\ol{1}}}(\SE{A}  | \SO{A} - E_{h+1, h+1}  + E_{h,h+1})(\bs{j}-\ep_{h+1})+\sum_{k>h+1}{\up}^{f_{h,k}^{\ol{1}}}(\SE{A}  | \SO{A} - E_{h+1, k}  + E_{h,k})(\bs{j})\\
-\sum_{k<h}&{\up}^{f^{\bar1}_{h,k}}(\up-\up^{-1}) \left[a_{h,k}+1\atop 2\right](\SE{A}+ 2E_{h,k} | \SO{A}  -E_{h,k} - E_{h+1,k})(\bs{j}+\bsal_h)\\
-&\frac{{\up}^{f^{\bar1}_{h,h}-2j_h}}{(\up-\up^{-1})}\big\{\frac{\up^{-1}}{[2]}(\SE{A}| \SO{A}  -E_{h,h} - E_{h+1,h})(\bs{j}+2\ep_h-\ep_{h+1})\\
&\hspace{0.2cm}+\frac{\up}{[2]}(\SE{A}| \SO{A}  -E_{h,h} - E_{h+1,h})(\bs{j}-2\ep_h-\ep_{h+1})-(\SE{A}| \SO{A}  -E_{h,h} - E_{h+1,h})(\bs{j}-\ep_{h+1})\big\}\\
-&{\up}^{f^{\bar1}_{h,h+1} }(\up-\up^{-1})\left[{a_{h,h+1}}+1\atop 2\right](\SE{A}+ 2E_{h,h+1} | \SO{A}  -E_{h,h+1} - E_{h+1,h+1})(\bs{j}-\ep_{h+1})\\
-&\sum_{k>h+1}{\up}^{f^{\bar1}_{h,k}}(\up-\up^{-1})\left[{a_{h,k}}+1\atop 2\right](\SE{A}+ 2E_{h,k} | \SO{A}  -E_{h,k} - E_{h+1,k})(\bs{j}).
\end{aligned}
$$
$$
\begin{aligned}
(3)\quad&\sfF_h \cdot \Ad(\bs{j})= \\
&\sum_{k < h}\up^{g_{h,k}^{\bar 0}}
	[ \SEE{a}_{h+1, k} +1]
	\ABJS( \SE{A} - E_{h,k} + E_{h+1, k}, \SO{A}, \bs{j}) +
	{\up}^{g_{h,h}^{\bar 0}+j_h}
	[\SEE{a}_{h+1, h} +1]
	\ABJS( \SE{A} + E_{h+1, h}, \SO{A}, \bs{j}) \\
  +&\frac{{\up}^{g_{h,h+1}^{\bar 0}-j_{h+1}}} {\up-\up^{-1} }\{
	\ABJS( \SE{A} - E_{h,h+1}, \SO{A}, \bs{j}-\bsal_{h})
	-\ABJS( \SE{A} - E_{h,h+1}, \SO{A}, \bs{j}-\bsal^+_{h})  \}
	\\
	+&\sum_{k > h+1}
	{\up}^{g_{h,k}^{\bar 0}}
	[\SEE{a}_{h+1, k} +1]
	\ABJS( \SE{A} - E_{h,k} + E_{h+1, k}, \SO{A}, \bs{j}-\bsal_h)\\
	+\sum_{k<h}&
	{\up}^{g_{h,k}^{\ol{1}}}	\ABJRS(\SE{A}, \SO{A} - E_{h,k} + E_{h+1,k}, \bs{j})+
	{\up}^{g_{h,h}^{\ol{1}}}
	\ABJS(\SE{A}, \SO{A} - E_{h,h} + E_{h+1,h},  {\bs{j}} +\ep_h) \\
+&{{\up}^{g_{h,h+1}^{\ol{1}}} }
	\ABJS(\SE{A}, \SO{A} - E_{h,h+1} + E_{h+1,h+1}, \bs{j}-\ep_h)  +
	\sum_{k>h+1}
	{\up}^{g_{h,k}^{\ol{1}}}
	\ABJS(\SE{A}, \SO{A} - E_{h,k} + E_{h+1,k}, {\bs{j}} -\bsal_{h}) \\
-\sum_{k < h }&
	{\up}^{g_{h,k}^{\ol{1}}}(\up-\up^{-1})\left[{{a}_{h+1, k} +1}\atop 2\right]
	( \SE{A} + 2E_{h+1,k}| \SO{A}  - E_{h,k} - E_{h+1,k})(\bs{j}) \\
-&{\up}^{g_{h,h}^{\ol{1}}}
	(\up-\up^{-1})\left[{{a}_{h+1, h} +1}\atop 2\right]
	( \SE{A}  + 2E_{h+1,h}| \SO{A}  - E_{h,h} - E_{h+1,h})( \bs{j} + \ep_{h}) \\
-&\frac{{\up}^{g_{h,h+1}^{\ol{1}}-2j_{h+1}}}{(\up-\up^{-1})}\big\{\frac{\up^{-1}}{[2]}
	( \SE{A}| \SO{A}  - E_{h,h+1} - E_{h+1,h+1})(\bs{j} -\ep_{h}+ 2 \ep_{h+1}) \\
	&\hspace{0.2cm}+\frac{\up}{[2]}
	( \SE{A}| \SO{A}  - E_{h,h+1} - E_{h+1,h+1})(\bs{j}-\ep_{h}-2 \ep_{h+1}) -( \SE{A}| \SO{A}  - E_{h,h+1} - E_{h+1,h+1})
(\bs{j}-\ep_{h})\big\}\\
-&\sum_{k>h+1}{\up}^{g_{h,k}^{\ol{1}}}
	(\up-\up^{-1})\left[{{a}_{h+1, k} +1}\atop 2\right]( \SE{A} + 2E_{h+1,k}| \SO{A}  - E_{h,k} - E_{h+1,k})(\bs{j} -\bsal_{h} ).\\
(4)\quad&\sfK_{\ol{n}}\cdot\Ad(\bs{j})=\\
&\sum_{k<n}{(-1)}^{ {{\SOE{\widetilde{a}}}_{n,k}}+\parity{\Ad}}{\up}^{d_n(\Ad,k)}(\SE{A} -E_{n,k}|\SO{A}+ E_{n,k})(\bs{j}+\ep_n)\\
&\qquad+{(-1)}^{ {{\SOE{\widetilde{a}}}_{n,n}}+\parity{\Ad} }{\up}^{d_n(\Ad,n)+j_n}(\SE{A}|\SO{A}+ E_{n,n})(\bs{j})\\
+&\sum_{k<n}{(-1)}^{ {{\SOE{\widetilde{a}}}_{n,k}}+\parity{\Ad}}{\up}^{d_n(\Ad,k)}[a_{n,k}]_{\up^2}(\SE{A} + E_{n,k}| \SO{A} - E_{n,k})(\bs{j}+\ep_n)\\
&\qquad+{(-1)}^{ {{\SOE{\widetilde{a}}}_{n,n}}+\parity{\Ad} }\frac{{\up}^{d_n(\Ad,n)-j_n}}{\up^2-\up^{-2}}\{(\SE{A} | \SO{A} - E_{n,n})(\bs{j}+2\ep_n)-(\SE{A} | \SO{A} - E_{n,n})(\bs{j}-2\ep_n)\}.
\end{aligned}
$$
Here undefined notations follow \eqref{f_hk} and Propositions \ref{mulformzerocor}--\ref{mulformdiag}.
\end{subs}

\begin{subs}{\bf Remarks on the history of the project.}
This project was initiated almost ten years ago.  The main obstacles occurred in deriving fundamental multiplication formulas for odd generators. After some unsuccessful early attempts in 2014, we realised that the computation for the multiplication formulas in the odd case is almost impossible. Thus, we suspected the existence of such a construction! A couple of years later,  we tried to tackle the problem in two different pre-projects. First, establish the construction in the classical (i.e., ${v}=1$) case and second, develop a direct construction of the regular representation of $\bsU_\up(\mathfrak q_n)$, using $\up$-differential operators. Both projects were successful.

In \cite{GLL}, a new realisation of the enveloping superalgebra {$\mathcal U(\mathfrak{q}_n)$} using Sergeev superalgebras is obtained by following closely the aforementioned roadmap. On the other hand, motivated from the examples given in {\color{black}\cite[\S5A.7-6]{J}}, a direct construction of the regular representation, using $\up$-differential operators, was first tested for quantum $\mathfrak{gl}_{m|n}$ in \cite{DZ} and then generalised to the queer case in \cite{DLZ}. The two works convinced us such a construction $\bs\sH_r^c\rightsquigarrow\bsU_\up(\mathfrak q_n)$, using Hecke--Clifford algebras $\bs\sH_r^c$, must exist. Thus, we completed a preliminary version (see unpublished manuscript \cite{PV}) in 2022. In this work, standardization was not considered and a complicated order relation was used in the triangular relations! With another two-year efforts, we finally discovered how to normalize the natural basis to obtain the standard basis and to replace the ugly order with the order defined as an easy extension of the well-known order in \cite{BLM}.

It should be pointed out that, unlike the classical case in \cite{GLL}, we still could not write down completely all  multiplication formulas
 involving odd generators  because of the complexity caused by the Clifford part. This
 shows a drastic difference from the case of  $\mathfrak{gl}_{m|n}$.
However, by introducing the SDP condition
 (see \cite[Def.~4.2]{DGLW}), we derive some fundamental multiplication formulas
 under the SDP condition; see \cite[Th.~6.2--4]{DGLW}.\footnote{We will standardise these formula in this paper.}
 Fortunately, by an inductive approach (cf. Remark \ref{induction_N}),
 this is enough to guarantee that the generators and relations for the quantum queer supergroup $\bsU_{\!\up}(\mathfrak{q}_n)$ all hold in queer $\up$-Schur superalgebras.
 In this way, together with an application of the PBW type basis established in \cite{DW1}, we successfully completed Steps (4)--(8) in the roadmap.
Note that the combinatorics developed in \cite{DGZ} and \cite{DW1,DW2} played decisive roles.
\end{subs}

\begin{subs}{\bf Further applications.}
In forthcoming papers, we will apply
this new construction via queer $q$-Schur algebras to establish the integral Schur--Weyl--Olshanski duality and
to develop the modular representation theory of quantum queer supergroups at roots of unity, generalizing the work \cite{BK1, BK2} to the quantum case. We will also use a bar involution, whose definition relies heavily on our theory, to establish the canonical basis theory for polynomial representations (cf. \cite{GJKK,GJKKK}) and for the modified supergroup $\dot\bsU_{\!\up}(\mathfrak{q}_n)$.
\end{subs}

\begin{subs}{\bf Organization.}
This  paper is organized as follows.
In Section \ref{sec_qqschur_natural basis},
we recall some basics on the natural bases for queer $q$-Schur superalgebras $\mathcal Q_q(n,r;R)$ and their standardized (or signed) version
$\SQqnrR$. We further show the two superalgebras are isomorphism if $\sqrt{-1}\in R$.
Section \ref{sec_qqschur} starts with a Drinfeld--Jimbo type presentation for the quantum queer supergroup $\bsU_{\!\up}(\mathfrak{q}_n)$ and Olshanski's epimorphism from $\bsU_{\!\up}(\mathfrak{q}_n)$ to $\qSchvsQ:=\mathcal Q^s_q(n,r;R)$ for $R=\QQ(\up),q=\up^2$.
The integral version $\qSchvsZ:=\mathcal Q_{\up^2}^s(n,r;\sZ)$ ($\sZ:=\ZZ[\up,\up^{-1}]$) of $\qSchvsQ$ admits a standard basis. This basis is a normalization of the natural basis for $\qSchvsZ$, a key ingredient which has no geometric interpretation.
In Section \ref{SMFQS},
we derive the multiplication formulas of the standard basis by certain homogeneous generators for the superalgebra
$\qSchvsQ$.
 Similar to BLM's original method, we introduce
in Section \ref{sec_spanningsets} some ``long elements'' which span the algebra $\qSchvsQ$ over $\QQ(\up)$ uniformly on $r$ and derive the corresponding multiplication formulas.
Note that the ``structure constants'' in these multiplication formulas are independent of $r$. We then check all the defining relations of $\Uvqn$ in $\qSchvsQ$ in Section \ref{Quantum relations}.
 In this way, we obtain superalgebra homomorphisms $\bs\xi_{n,r}$ from $\Uvqn$ to $\qSchvsQ$ (Theorem \ref{qqschur_reltion-r}).

The aim of the remaining three sections is to prove the main Theorem \ref{mthm}. The homomorphisms $\bs\xi_{n,r}$ induces a homomorphism $\bs{\xi}_n$ from $\Uvqn$ to the direct product $\bsSQvn$ of  $\bsSQvnr$. The determination of the image of $\bs{\xi}_n$ requires a triangular transition matrix between a monomial basis and the standard basis for $\qSchvsQ$. This is done in Theorem \ref{triang-QqS}.
A key ingredient for the triangular relations is the selection of an order relation. For example, in a preliminary version of the work \cite{PV}, a very complicated order was used. The new order relation $\prec^\star$ in Definition \ref{preorder-2} is an easy extension of the order relation used for quantum $\mathfrak{gl}_n$ \cite{BLM}. In Section 8, by taking limits of long elements, we obtain infinitely long elements (as formal infinite  series) which span a subspace $ \USnv $ of $\bsSQvn$. We then use the extended triangular relations to prove that $\text{im}(\bs{\xi}_n)=\USnv$ (Theorem  \ref{triangular_relation_q}).
Finally, in Section \ref{sec_generators},
we use the triangular decomposition of $\Uvqn$ together with a certain weight space decomposition to prove that $\bs{\xi}_n$ is an isomorphism.
Thus, the new superalgebra {$\USnv$} becomes a new realization of {$\Uvqn$}, proving the main Theorem \ref{mthm}.
\end{subs}

\begin{subs}\label{sec_notations}{\bf Some standard notations.}
Denote by $\mathbb{N}$ (resp., $\ZZ$) the set of non-negative integers (resp., integers) and let $\mathbb{N}_{2}=\{0,1\}$, $\mathbb{Z}_{2}=\{\bar0,\bar1\}$. Let {${\up}$} be an indeterminate, $\sZ=\mathbb Z[\up,\up^{-1}]$, and $\QQ(\up)$ the fraction field of $\sZ$.

For any positive integers $m\in {\ZZ}$, let
\begin{equation}\label{stepd}
\begin{aligned}
&[m]=\frac{\up^m-\up^{-m}}{\up-\up^{-1}},\quad [m]!=[m][m-1]\cdots[1],\quad \left[m\atop s\right]=\frac{[m]!}{[s]![m-s]!}\\
&{\STEPX{m}{}} = 1 + {q} + \cdots + {q}^{m-1}
	 = \frac{{q}^{m} - 1}{{q} - 1}\;\;(q=\up^2).
\end{aligned}
\end{equation}

For {$i,j \in \ZZ$} with {$i \le j$}, let {$[i, j] = \{ i, i+1, \cdots, j \}$} and $[i,j)=[i,j-1]$, for $i<j$.
\end{subs}

\spaceintv
\section{Preliminaries}\label{sec_qqschur_natural basis}
We start with the definition of a Hecke-Clifford superalgebra.
Let $R$ be   a commutative ring such that {$2 \in R^\times$} (the group of units in $R$).
Assume $r \in \NN$ is a positive integer, and let
 {$\fS_{r}$} be the symmetric group on {$r$} letters with Coxeter generator $s_i=(i,i+1)$, for all $i\in[1,r)$.

 Let $q\in R^\times$.
The {\it Hecke-Clifford superalgebra}  {$\HCR$}  is a  superalgebra  over {${R}$}  generated by
even generators {$T_{i} = T_{s_{i}}$} for all  {$1 \le i < r$},  and odd generators {$c_1, \cdots, c_r$},
subject to the relations
\begin{equation}\label{hc_generators}
\begin{aligned}
&	c_i^2 = -1,  \quad c_i c_j = - c_j c_i  ,  \qquad (1 \le i, j \le r-1);  \\
&	(T_i - {q} )(T_i + 1) = 0, \qquad T_i T_j = T_j T_i , \qquad
	 (q\in R,\;|i-j| >1),\\
&	T_i T_{i+1} T_i = T_{i+1} T_i T_{i+1}, \qquad (i\neq r-1); \\
&T_i c_j = c_j T_i , \qquad (j\neq i, i+1); \\
&T_i c_i = c_{i+1} T_i, \qquad
 T_i c_{i+1} = c_i T_i - ({q}-1)(c_i - c_{i+1}).
\end{aligned}
\end{equation}
The subalgebra  {$\Heck$} generated by even generators is the Hecke algebra associated with {$\fS_{r}$} and has
basis $\{T_w\mid w\in\fS_r\}$, where $T_w=T_{i_1}\cdots T_{i_l}$ if $w=s_{i_1}\cdots s_{i_l}$ is reduced.

For any homogeneous element {$h \in \HCR$}, let {$\parity{h}\in \ZG$} be its parity, more precisely,  {$\parity{h} = \bar 0$} if  {$h$} is even, or {$\parity{h} = \ol{1}$} if {$h$} is odd. From Section \ref{sec_qqschur} onwards, we will only consider the generic version $\sH_r^c=\sH_{r,\sZ}^c$ and
$\boldsymbol\sH_r^c=\sH_{r,\mathbb Q(\up)}^c$, where $q=\up^2$.

We now introduce some special elements in $\HCR$ building on which the structure constants of the generators of $\bsU_{\!\up}(\mathfrak{q}_n)$ relative to a new basis can be computed near the end of the paper.

First, the Young (or parabolic) subgroups of $\fS_r$ are defined by compositions of $r$.
Recall the set of  compositions of $r$ into $n$ parts:
\begin{equation}\label{Lanr}
\CMN(n,r)=\Big\{\lambda=(\lambda_1,\lambda_2,\ldots,\lambda_n)\in \NN^n \where  |\lambda|:=\sum_i\lambda_i=r\Big \}.
\end{equation}
Associated with {${\lambda} = ({\lambda}_1, \cdots , {\lambda}_{n}) \in \CMN(n,r)$}, define its {\it partial sum sequence} $ (\widetilde{\lambda}_1,\ldots,\widetilde{\lambda}_n)$, where
\begin{equation}\label{latilde}
\widetilde{\lambda}_i = \sum_{k=1}^i \lambda_{k} \;\;(1\leq i\leq n), \text{ and set }\widetilde{\lambda}_0=0.
\end{equation}
Then the Young subgroup $\fS_\la$ is the subgroup
$$\fS_\la:=\langle s_i\mid i\in[1,r)\setminus\{\widetilde \la_j\mid j\in[1,n)\}\rangle.$$

Let $x_\la=\sum_{w\in\fS_\lambda}T_w$. These special elements are used to define  the right $\HCR$-module $M=\bigoplus_{\lambda \in \CMN(n,r)} x_{\lambda} \HCR$. An
$\HCR$-module homomorphisms $\phi:M\to M$ is defined to satisfy $\phi(mh)=\phi(m)h$, for all $m\in M$ and $h\in\HCR$ and all such homomorphisms define,
following \cite{DW2},  the {\it queer $q$-Schur superalgebra} over {$R$}
\begin{equation}\label{QqnrR}
	\QqnrR := \End_{\HCR}\Big(\bigoplus_{\lambda \in \CMN(n,r)} x_{\lambda} \HCR\Big).
\end{equation}

Second,
we introduce some special $c$-elements $c_{q,i,j}, c_\nu^\alpha$ in $\HCR$ to describe bases for $x_\la\HCR\cap\HCR x_\mu$ and, hence, for $\QqnrR$.
By definition, for  $1\leq i<j\leq r$, we set
\begin{equation}\notag
c_{q,i,j}=q^{j-i}c_i+q^{j-i-1}c_{i+1}+\cdots+qc_{j-1}+c_j,\quad
c'_{q,i,j}=c_i+q c_{i+1}+\cdots+q^{j-i}c_j
\end{equation}
Putting $c_{q,r}=c_{q,1,r}$ and $c_{q,r}'=c_{q,1,r}'$, we have the first commutation relations: $
x_{(r)}c_{q,r}=c'_{q,r}x_{(r)}$. 
Thus, the set $\{x_{(r)}, x_{(r)}c_{q,r}\}$ forms a basis for $x_{(r)}\HCR\cap\HCR x_{(r)}$.

For {$\lambda= (\lambda_1, \cdots , \lambda_m)\in {\NN}^m$} and {$\alpha \in \Ng^m $} with $m\geq 1$,  the following elements introduced in \cite[\S4]{DW2}
\begin{equation}\label{eq_c_a}
\begin{aligned}
	&c^{\alpha}_{\lambda} = {(c_{{q}, 1, \widetilde{\lambda}_1})}^{\alpha_1}
				{(c_{ {q}, \widetilde{\lambda}_1+1 , \widetilde{\lambda}_2} )}^{\alpha_2}
				\cdots
				{(c_{{q}, \widetilde{\lambda}_{m-1}+1,  \widetilde{\lambda}_m} )}^{\alpha_m}, \\
	&(c^{\alpha}_{\lambda})' = {(c'_{{q}, 1, \widetilde{\lambda}_1})}^{\alpha_1}
				{(c'_{{q}, \widetilde{\lambda}_1 + 1,  \widetilde{\lambda}_2} )}^{\alpha_2}
				\cdots
				{(c'_{{q}, \widetilde{\lambda}_{m-1} + 1,  \widetilde{\lambda}_m })}^{\alpha_m} .
\end{aligned}
\end{equation}
Here we use the convention $x^0=1$.
By \cite[Lem.~4.1, Cor.~4.3]{DW2},
we have commutation relations
\begin{align}\label{xc}
	x_{\lambda} c_{\lambda}^{\alpha} = (c_{\lambda}^{\alpha})'  x_{\lambda},
\end{align}
for all $\alpha\leq\lambda$ (meaning $\alpha_i\leq\lambda_i $ for any $i$),
and the set $\{x_{\la}c_\la^{\alpha}\mid \alpha\in\NN^n_2, \alpha\leq\la\}$ forms a basis for
$x_{\la}\sH_{\la,R}^c\cap\sH_{\la,R}^c x_{\la}$, where $\sH_{\la,R}^c$ is the parabolic subalgebra associated with $\la$.

Finally, we describe a basis for $\QqnrR$.

Denote by $\MN(n)$ the set of $n\times n$-matrices with entries being non-negative integers. For any {$A=(a_{i,j}) \in \MN(n)$}, set $|A|=\sum_{i,j}a_{i,j}$ and
\begin{equation}\label{roco}
\ro(A)=\Big(\sum_{j=1}^na_{1,j},\sum_{j=1}^na_{2,j},\ldots,\sum_{j=1}^na_{n,j}\Big),\quad
\co(A)=\Big(\sum_{i=1}^na_{i,1},\sum_{i=1}^na_{i,2},\ldots,\sum_{i=1}^na_{i,n}\Big).
\end{equation}

Let {$\Ng = \{0, 1\} \subset \NN$}. Recall the following matrix sets introduced in \cite{DW2}:
\begin{equation}\label{labelsets}
\aligned
&\MNR(n,r) =\{A=(a_{i,j})\in \MN(n)\mid r=|A|\},\\
	& \MNZN(n) = \{ \Ad=(\SUP{A}) \where \SE{A} \in \MN(n), \SO{A} \in M_n(\mathbb{N}_2) \}, \\
	& \MNZ(n,r) = \{ \Ad=(\SUP{A}) \in \MNZN(n) \where  \SE{A} + \SO{A} \in \MNR(n,r) \}.
	\endaligned
\end{equation}
Define the parity map $\wp$ $:\MNZN(n)\longrightarrow \mathbb{Z}_2$ as follows: for any $\Ad\in\MNZN(n)$,
\begin{equation}\label{parity}
\parity\Ad=|\SOE{A}|\; \mbox{mod }2=\sum_{i,j}a^{\bar1}_{i,j}\; \mbox{mod } 2 .
\end{equation}

For $\Ad =(A^{\bar0}|A^{\bar1}) \in\MNZ(n,r)$, let $A=A^{\bar0}+A^{\bar1}$, which is called the {\it base} of $\Ad$. Suppose
$A^{\bar0}=(\SEE{a}_{i,j})$, $A^{\bar1}=(\SOE{a}_{i,j}) $, and $A=(a_{i,j})$ with $a_{i,j}=a^{\bar0}_{i,j}+a^{\bar1}_{i,j}$.
By concatenating columns,  we set
\begin{align}\label{nuA}
\nu_A:=(a_{1,1},\ldots,a_{n,1},\ldots,a_{1,n},\ldots,a_{n,n}),
\end{align}
and its associated partial sum sequence (see \eqref{latilde})
\begin{align}\label{tnuA}
\widetilde \nu_A=(\widetilde a_{1,1},\ldots,\widetilde a_{i,j},\ldots,\widetilde a_{n,n}),\;\text{ where }\;\widetilde{a}_{i,j} := \sum_{p=1}^{j-1}\sum_{u=1}^{n} a_{u,p} + \sum_{u=1}^{i} a_{u,j}.
\end{align}

For $\nu=\nu_A$ and $\nu^{\bar1}=\nu_{A^{\bar1}}$, define
\begin{equation}
\begin{aligned}\label{eq cA}
	c_{\Ad} &=  c_{\nu}^{\nu^{\bar1}}, \qquad c'_{\Ad} =  (c_{\nu}^{\nu^{\bar1}})', \qquad\\
	T_{\Ad}& =  x_{\ro(A)} T_{d_{A}} c_{\Ad}
			\sum _{{\sigma} \in \D_{{\nu}_{A}} \cap \fS_{\co(A)}} T_{\sigma},
\end{aligned}
\end{equation}
where $d_A$ is the distinguished (or shortest) representative of the double coset $\fS_\la d_A\fS_\mu$ ($\la={\ro(A)},\mu={\co(A)}$) defined by matrix $A$; (see \cite[\S4.2]{DDPW}). Note that $\parity{T_\Ad}=\parity{c_\Ad}=\parity{\Ad}$.

By \cite[Prop.~5.2]{DW2}, for any $\lambda,\mu\in\Lambda(n,r)$,
the intersection $x_\lambda\HCR\cap \HCR x_\mu$ is a free {$R$-module} with basis $\{T_\Ad\mid \Ad  \in\MNZ(n,r), \lambda=\ro(A), \mu=\co(A)\}$. Thus, we have the following.

\begin{prop}[{\cite[Th.~5.3]{DW2}}]\label{DW-basis}
The queer $q$-Schur superalgebra $\QqnrR$ is a free $R$-module with {\sf natural basis} $\{ \phi_{\Ad} \where  \Ad \in \MNZ(n,r)\}$, where {$\phi_{\Ad} $} is defined by
\begin{equation*}
\phi_{\Ad}: \bigoplus_{\lambda \in \CMN(n,r)} x_{\lambda} \HCR \longrightarrow \bigoplus_{\lambda \in \CMN(n,r)} x_{\lambda} \HCR,\;\;
 x_{\mu}h \longmapsto \delta_{\mu, \co(A)} T_{\Ad}h,\;\forall  \mu \in \CMN(n,r), h \in \HCR,
\end{equation*}
 and $\phi_\Ad$ has parity $\parity{\Ad}$.
\end{prop}
We observe that, for $\Ad,\Bd\in \MNZ(n,r)$,
\begin{equation}\label{co=ro}
\phi_\Bd\phi_\Ad \ne  0 \implies \co(B) = \ro(A).
\end{equation}
Note that this natural basis for $\QqnrR$ is an natural extension of the double coset basis of $q$-Schur algebras or superalgebras.

Fundamental multiplication formulas are referred to writing the product $\phi_\Bd\phi_\Ad$ as a linear combination of natural basis elements for certain choices of $\Bd$:
\begin{equation}\label{fMFs}
\phi_\Bd\cdot\phi_\Ad= \sum_{\Md} \gamma_{\Bd,\Ad}^{\Md} \phi_\Md.
\end{equation}
 According to the generators for the queer Lie superalgebra $\mathfrak q_n:=\mathfrak q(n,\QQ)$, $\Bd$ is one of the following six matrices
$$(\mu+E_{j,j}|O),(\mu+E_{h, h+1}|O), (\mu+E_{h+1,h}|O);\quad (\mu|E_{j,j}),
(\mu|E_{h, h+1}),
(\mu|E_{h+1,h}),$$
where $\mu\in\Lambda(n,r-1)$ stands for the diagonal matrix diag$(\mu)$ and {$E_{i,j}$} denotes the $n \times n$ matrix unit with entry $1$ at {$(i,j)$} position and {$0$} elsewhere. Here
we also use $O$ to denote the zero matrix in {$\MN(n)$} or  {$M_n(\Ng)$}, and set {${\Od} = ({O}|{O}) \in \MNZN(n)$}.

 In the even case, where $\Bd\in \{(\mu+E_{j,j}|O),(\mu+E_{h, h+1}|O), (\mu+E_{h+1,h}|O)\}$, \eqref{fMFs} is given in \cite[Th.~5.3]{DGLW}; while in the odd case, where $\Bd\in\{(\mu|E_{j,j}),(\mu|E_{h, h+1}),(\mu|E_{h+1,h})\}$, \eqref{fMFs} is given in \cite[Th.~6.2--4]{DGLW}.

  The multiplication formulas in the odd case only display ``the head terms'' defined by a so-called SDP condition. The SDP condition is about commuting a Clifford generator with an element in $\sH_{r,R}$ associated with the distinguished double coset representative.

Recall the following bijection (see e.g., \cite[Section 4.2]{DDPW})
\begin{equation}\label{jmath}
\aligned
	\mathfrak{j}: \MNR(n,r) &\longrightarrow \{ (\lambda,  d, \mu) \where \lambda,\mu \in \CMN(n,r), d \in \D_{\lambda, \mu} \},\quad
	A \longmapsto (\ro(A), d_A, \co(A)).
	\endaligned
\end{equation}
where $\D_{\lambda, \mu} $ is the set of distinguish representatives of the $(\mathfrak{S}_\lambda,\mathfrak{S}_\mu)$ double coset in $\mathfrak{S}_r$. The representative $d_A$ of minimal length in the double coset can be defined as follows.

For $A=(a_{ij})$, recall from \eqref{tnuA} the column entry sum $\widetilde a_{i,j}$ up to (and including) $a_{i,j}$. We now define
 $\widetilde a^r_{ij}$ to be the entry sum along row 1 and row 2, etc., up to (and including) $a_{i,j}$. Then\footnote{The row entry partial sum $\widetilde a^r_{ij}$ is denoted as $\widehat a_{i,j}$ in \cite{DGLW}.} $d_A({\widetilde a}_{h-1,k}+ p )=\widetilde a^r_{h,k-1} + p $,
for all $h,k\in[1,n] $ with $a_{h,k}>0$ and $p\in[1,a_{h,k}]$. In other words, as a permutation, we have
$$d_A=\displaystyle\prod_{h,k\atop a_{h,k}>0}\left(\begin{smallmatrix}
\widetilde a_{h-1,k}+1 &\widetilde a_{h-1,k}+2&\cdots&\widetilde a_{h-1,k}+a_{h,k}\\
\widetilde a^r_{h,k-1} +1
& \widetilde a^r_{h,k-1}+2
&\cdots&\widetilde a^r_{h,k-1}+a_{h,k}
\end{smallmatrix}\right).$$
We say that  {$A$}
satisfies the {\it semi-direct product} (SDP) {\it condition} at {$(h, k)$} with $a_{h,k}>0$ (or at a non-zero $(h,k)$-entry) if, for each {$p \in [1,a_{h,k}]$},
\begin{displaymath}
 c_{\widetilde a^r_{h,k-1}+p} {T_{d_A}}={T_{d_A}} c_{\widetilde{a}_{h-1,k} + p}
	 \qquad(\mbox{in }\HCR).
\end{displaymath}
 If $A$ satisfies the SDP condition at $(h,k)$, for all $k\in[1,n]$ with $a_{h,k}>0$, then we say that $A$ satisfies the SDP condition on the $h$-th row.
\begin{thm}[{\cite[Th.~4.6, Cor.~4.7]{DGLW}}]\label{CdA1}
Let $A\in \MN(n)$ and $h,k\in[1,n]$.
Then $A$ satisfies the SDP condition at $(h,k)$ if and only if $a_{h,k}>0$ and $a_{i,j}=0$, for $i>h$ and $j<k$ (equivalently, the lower left corner matrix $\llcm^{h,k}$ at $(h,k)$ is 0). In particular, $A$ satisfies the SDP condition on the $n$-th row.
\end{thm}

For $\Ad=(A^{\ol{0}}|A^{\ol{1}})\in  \MNZ(n,r)$, $i\in\Ng$, and {$1 \le h \le n-1, 1 \le k \le n$}, let
\begin{align}\label{A01hk}
	& A^{\ol{i},+}_{h,k} = A^{\ol{i}} + E_{h,k} - E_{h+1, k}, \qquad
	  A^{\ol{i},-}_{h,k} = A^{\ol{i}} - E_{h,k} + E_{h+1, k} \in \MN(n).
\end{align}
If $A=A^{\ol{0}}+A^{\ol{1}}$, then $(A^{\ol{0},+}_{h,k}|A^{\ol{1}})$ and $(A^{\ol{0}}|A^{\ol{1},+}_{h,k})$, resp., $(A^{\ol{0},-}_{h,k}|A^{\ol{1}})$ and $(A^{\ol{0}}|A^{\ol{1},-}_{h,k})$ has base
\begin{equation}\label{Ahk}
A^+_{h,k}= A+ E_{h,k} - E_{h+1, k}\;\;\text{ resp. } A^-_{h,k}= A- E_{h,k} +E_{h+1, k}.
\end{equation}
Each of these matrices is obtained by moving 1  up or down a row in column $k$. We call them a {\it 1-up matrix} or {\it 1-down matrix}.
\begin{rem}\label{Md}
If we only consider the fundamental multiplication formulas \eqref{fMFs} as given in \cite[Ths.~5.3, 6.3--4]{DGLW} for $\Bd\in\{(\mu+E_{h, h+1}|O), (\mu+E_{h+1,h}|O),
(\mu|E_{h, h+1}),
(\mu|E_{h+1,h})\}$ (See Lemmas \ref{phiupper-even-norm} and \ref{standard-phiupper1} below), then the matrices $\Md$, appearing in \eqref{fMFs}, or in the leading term if $\Bd$ is odd, take the {\color{black}form}
$(\SE{A} + E_{h,k}  - E_{h+1, k} | \SO{A} ),(\SE{A}  | \SO{A}+ E_{h,k}  - E_{h+1, k} ), (\SE{A} + 2E_{h,k}  | \SO{A} - E_{h, k} - E_{h+1, k} )$ etc.
Note that the base matrices of $\Md$ are all of the {\color{black}form} $A^+_{h,k}$ or $A^-_{h,k}$.
\end{rem}

We end this preliminary section with the definition of the superalgebras $\SQqnrR$ of superhomomorphisms.
More precisely,
given ${R}$-free supermodule\footnote{By an $R$-free supermodule, we mean an $R$-free module with a basis consisting of homogeneous elements.}  $V,W$,
we used    $\Hom_{{R}}(V,W)$ to denote the set of  all ${R}$-linear  maps from $V$ to $W$.
We also made $\Hom_{{R}}(V,W)$  into an $R$-free supermodule by declaring
${\Hom_{{R}}(V,W)}_{\ol{i}}$
to be  the set of  homogeneous map of parity $\bar i\in \ZG$,
that is,
the linear maps $\theta:V\rightarrow W$ with $\theta(V_{\ol{j}})\subset W_{\overline{i}+\ol{j}}$ for {$\bar j\in\ZG$}.

If $R$ is also a superalgbera, we may define $\Hom^s_{{R}}(V,W)$ to be the set of all $R$-superhomomorphisms $f:V\to W$ such that
$$f(rv)=(-1)^{\parity{f}\parity{r}}rf(v),\;\;\text{ for all }r\in R, v\in V.$$
Note that it is typical in a category of super objects to write the expressions
which only make sense for homogeneous elements,
and the expected meaning for arbitrary elements
 is obtained by extending linearly from the homogeneous case.

 Thus, following \cite[Definition 1.1]{CW},
 we define, for right $\HCR$-modules $M,N$,
 $\Hom^s_{\HCR}(M,N)$ to be the $R$-supermodule generated by
 all $\HCR$-superhomomorphisms $\Phi:M\to N$ satisfying
\begin{equation}\label{eq:twist}
 \Phi(mh)=(-1)^{\parity{\Phi}\parity{h}}\Phi(m)h\;\; (m\in M, h\in\HCR).
\end{equation}
In particular, for $M=\bigoplus_{\lambda \in \CMN(n,r)} x_{\lambda} \HCR$,
let $\End^s_{\HCR}(M):=\Hom^s_{\HCR}(M,M)$.
 One checks easily that, if
 $\Phi\in\End^s_{\HCR}(M)_{\bar i}$ and $\Psi\in\End^s_{\HCR}(M)_{\bar j}$, then $\Psi\Phi\in\End^s_{\HCR}(M)_{\bar i+\bar j}$. Thus, we obtain a new (associative) superalgebra
\begin{equation}\label{SQqnrR}
	\SQqnrR
	:= \End^s_{\HCR}\Big(\bigoplus_{\lambda \in \CMN(n,r)} x_{\lambda} \HCR\Big)
	=\bigoplus_{\lambda, \mu \in \CMN(n,r)} \Hom^s_{\HCR}(x_{\lambda}\HCR,x_{\mu}\HCR).
\end{equation}
This is called the {\it standardised queer $q$-Schur superalgebra}.

The following result shows that the standardised queer $q$-Schur superalgebra $\SQqnrR$ is just a sign modification of the original queer $q$-Schur superalgebra $\QqnrR$.

\begin{prop}\label{prop_PhiAPhiB}
The standardised queer $q$-Schur superalgebra $\SQqnrR$ is $R$-free with basis $\{ \Phi_{\Ad} \where  \Ad \in \MNZ(n,r)\}$, where {$\Phi_{\Ad}: \oplus_{\mu \in \CMN(n, r)} x_{\mu}\HCR  \rightarrow \oplus_{\mu \in \CMN(n, r)} x_{\mu}\HCR $} is defined by
\begin{equation}\label{PhiA}
\Phi_\Ad(x_{\mu} h)= {(-1)}^{\parity{\Ad} \cdot \parity{ h}} \delta_{\mu, \co(A)} \Phi_{\Ad}(x_{\co(A)})  h=
 {(-1)}^{\parity{\Ad} \cdot \parity{ h}} \delta_{\mu, \co(A)} T_\Ad h,
\end{equation}
for all $\mu\in\CMN(n,r)$ and $h\in\HCR$.

Moreover,
for any  {$\Ad, \Bd \in \MNZ(n, r)$} with  {$\ro(A) = \co(B)$},
if $\phi_\Bd \phi_\Ad= \sum_{\Md} { \gamma^{\Md}_{\Bd,\Ad}} \phi_\Md$ and $\Phi_\Bd \Phi_\Ad= \sum_{\Md}{  \scc_{\Bd,\Ad,\Md} }\Phi_\Md$ with structure respective constants
${ \gamma^{\Md}_{\Bd,\Ad},\scc_{\Bd,\Ad,\Md}}\in R$, then
\begin{equation}\label{Phi_structure}
{ \scc_{\Bd,\Ad,\Md}
 =  {(-1)}^{\parity{\Ad} \cdot  \parity{ \Bd } } \gamma^{\Md}_{\Bd,\Ad} =\begin{cases}-\gamma^{\Md}_{\Bd,\Ad},&\text{ if both $\Ad,\Bd$ are odd};\\
\;\;\, \gamma^{\Md}_{\Bd,\Ad},&\text{ otherwise.}\end{cases}
}
\end{equation}
\end{prop}
\begin{proof} The freeness follows easily from the linear isomorphism
\begin{equation}\label{dagger}
(\;\;)^\dagger:\QqnrR\to\SQqnrR
\end{equation} as defined in \cite[(1.1)]{CW} or, more precisely, for a homogeneous $\phi\in\QqnrR$ of parity $\parity{\phi}$ and $M=\bigoplus_{\lambda \in \CMN(n,r)} x_{\lambda} \HCR$, $\phi^\dagger:M\to M$ is defined by \begin{align}\label{phi-dag}
\phi^{\dag}(m)=(-1)^{\parity{\phi}\cdot\parity{m}}\phi(m)
\end{align}
for homogeneous $m\in M$. It is straightforward to check  $\phi^{\dag}\in\SQqnrR $ and $\dagger\dagger=1$. The basis assertion is clear by setting  $\Phi_\Ad=\phi_\Ad^{\dag}$ for
all {$\Ad \in \MNZ(n, r)$}.

Finally,
 if we write $T_\Ad=x_{\ro(A)}h'$,
where $h'={T_{d_A}} c_{\Ad}  \sum _{{\sigma} \in \D_{{\nu}_{_A}} \cap \fS_{\co(A)}} T_{\sigma}$ (see \eqref{eq cA}),
then $h'$ is homogeneous with parity $\wp(h')=\wp(\Ad)$
and the structure constants  $ { \gamma^{\Md}_{\Bd,\Ad}}$
appearing in $\phi_\Bd\phi_\Ad = \sum_{\Md}  { \gamma^{\Md}_{\Bd,\Ad}} \phi_\Md$
are determined by writing $T_\Bd h'\in x_{\ro(B)}\HCR\cap \HCR x_{\co(A)}$
as a linear combination of $T_\Md$'s. In other words,
\begin{equation}\label{eq_gB}
z(\Bd,\Ad):=\phi_\Bd\phi_\Ad(x_{\co(A)})=T_\Bd h'= \sum_{\Md \in \MNZ(n,r)} { \gamma^{\Md}_{\Bd,\Ad}}  T_\Md,
\end{equation}
and, for each {$\Md$}, we have {$\parity{\Md} = \parity{\Ad} + \parity{\Bd}$}.

Thus,
for any homogeneous element {$h \in \HCR$}
and  {$ \mu \in \CMN(n, r)$}, by \eqref{PhiA},
we have
\begin{align*}
\Phi_\Bd \Phi_\Ad (x_{\mu} h)
& = \delta_{\mu, \co(A)} {(-1)}^{\parity{\Ad} \cdot \parity{ h }}  {(-1)}^{\parity{\Bd} \cdot ( \parity{ \Ad } + \parity{h})}   \phi_\Bd (T_{\Ad} h ) \\
& = \delta_{\mu, \co(A)} {(-1)}^{\parity{\Ad} \cdot \parity{ h }}  {(-1)}^{\parity{\Bd} \cdot ( \parity{ \Ad } + \parity{h})}    \sum_{\Md}{  \gamma^{\Md}_{\Bd,\Ad}}  T_\Md \cdot h  \\
& = \delta_{\mu, \co(A)}  {(-1)}^{\parity{\Bd} \cdot  \parity{ \Ad } }   \sum_{\Md} {(-1)}^{\parity{\Md}\cdot \parity{ h }}   {  \gamma^{\Md}_{\Bd,\Ad}}  \phi_\Md(x_\mu) \cdot h \\
& = \delta_{\mu, \co(A)} {(-1)}^{\parity{\Ad} \cdot  \parity{ \Bd } }   \sum_{\Md}  {  \gamma^{\Md}_{\Bd,\Ad}} \Phi_\Md(x_{\mu} h ).
\end{align*}
Hence, ${ \scc_{\Bd,\Ad,\Md} ={(-1)}^{\parity{\Ad} \cdot  \parity{ \Bd } }  \gamma^{\Md}_{\Bd,\Ad} }$,
as desired.
\end{proof}
 We have seen above that the  standardised version $\SQqnrR$ is linearly isomorphic to $\QqnrR$ as free $R$-modules. The next result shows that if, in addition, $\sqrt{-1}\in R$, then $\SQqnrR$ is isomorphic to $\QqnrR$ as $R$-superalgebras.

\begin{cor}\label{super_al_iso}
If {$\sqrt{-1} \in R$},
then the $R$-linear isomorphism $(\;\;)^\dagger$ in \eqref{dagger} induces  an
isomorphism of $R$-superalgebras
$${ \mathfrak{\tau}}:  \QqnrR  \longrightarrow \SQqnrR,\;\;
{\phi}_{\Ad}  \longmapsto  {(\sqrt{-1})}^{\parity{\phi_\Ad}}  {\Phi}_{\Ad},\;\;
\forall \Ad\in\MNZ(n,r).$$
\end{cor}

\begin{proof} Clearly, ${ \mathfrak{\tau}}$ is an $R$-linear isomorphism.
It remains to check {$ { \mathfrak{\tau}}({\phi}_{\Ad} {\phi}_{\Bd}) =   { \mathfrak{\tau}}( {\phi}_{\Ad})  { \mathfrak{\tau}}( {\phi}_{\Bd}) $}
for all $\Ad,\Bd\in\MNZ(n,r)$. By writing $\phi_\Bd \phi_\Ad
= \sum_{\Md} {  \gamma^{\Md}_{\Bd,\Ad}}  \phi_\Md$, as in Lemma \ref{prop_PhiAPhiB},
this can be easily proved by the following facts: \\
\qquad(a)\;\;${(\sqrt{-1})}^{\parity{\phi_\Ad}}  {(\sqrt{-1})}^{\parity{\phi_\Bd}}
	 {(-1)}^{ \parity{\phi_\Ad} \parity{\phi_\Bd} }
 = {(\sqrt{-1})}^{\parity{{\phi}_{\Ad} {\phi}_{\Bd}}}$;\\
 \qquad(b)\;\; {  {$\parity{{\phi}_{\Bd}{\phi}_{\Ad}}  = \parity{{\phi}_{\Md}} $} for each $\Md$ with $\gamma^{\Md}_{\Bd,\Ad}\neq 0$}.
\end{proof}

\spaceintv
\section{The quantum queer supergroup and standard basis for $\qSchvsZ$}\label{sec_qqschur}

We first introduced the quantum queer supergroups $\Uvqn$ and their associated Schur--Weyl duality (see \cite{Ol}). Then we investigate the standard basis, as a normalization of the natural basis, for the generic superalgebra $\qSchvsZ$ over $\sZ:=\ZZ[\up,\up^{-1}]$.

The following definition of $\Uvqn$, introduced by Olshanski \cite{Ol},  is taken from
\cite[Prop.~5.2]{DW1} (cf. \cite[\S2]{GJKK}).
Recall  the standard basis {$\{ \ep_i \where  i \in [1,n]\}$} for ${\ZZ}^n$ and define the ``dot product'' $\ep_i\centerdot \ep_j=\delta_{i,j}$, linearly extended to ${\ZZ}^n$.

\begin{defn}\label{defqn}
The quantum queer supergroup $\Uvqn$ ($n\geq2$) is the (Hopf) superalgebra
over $\Qv$  generated by
even generators  {${\genK}_{i}$}, {${\genK}_{i}^{-1}$},  {${\genE}_{j}$},  {${\genF}_{j}$},
and odd generators  {${\genK}_{\ol{i}}$},  {${\genE}_{\ol{j}}$}, {${\genF}_{\ol{j}}$},
for   {$ 1 \le i \le n$}, {$ 1 \le j \le n-1$}, subject to the following relations:
\begin{align*}
({\rm QQ1})\quad
&	{\genK}_{i} {\genK}_{i}^{-1} = {\genK}_{i}^{-1} {\genK}_{i} = 1,  \qquad
	{\genK}_{i} {\genK}_{j} = {\genK}_{j} {\genK}_{i} , \qquad
	{\genK}_{i} {\genK}_{\ol{j}} = {\genK}_{\ol{j}} {\genK}_{i}, \\
&	{\genK}_{\ol{i}} {\genK}_{\ol{j}} + {\genK}_{\ol{j}} {\genK}_{\ol{i}}
	= 2 {\delta}_{i,j} \frac{{\genK}_{i}^2 - {\genK}_{i}^{-2}}{{v}^2 - {v}^{-2}}; \\
({\rm QQ2})\quad
& 	{\genK}_{i} {\genE}_{j} = {v}^{\ep_i\centerdot\alpha_j} {\genE}_{j} {\genK}_{i}, \qquad
	{\genK}_{i} {\genE}_{\ol{j}} = {v}^{\ep_i\centerdot\alpha_j} {\genE}_{\ol{j}} {\genK}_{i}, \\
& 	{\genK}_{i} {\genF}_{j} = {v}^{-(\ep_i\centerdot\alpha_j)} {\genF}_{j} {\genK}_{i}, \qquad
	{\genK}_{i} {\genF}_{\ol{j}} = {v}^{-(\ep_i\centerdot \alpha_j)} {\genF}_{\ol{j}} {\genK}_{i}; \;
	(\mbox{here }\alpha_j=\ep_j-\ep_{j+1})\\
({\rm QQ3})\quad
& {\genK}_{\ol{i}} {\genE}_{i} - {v} {\genE}_{i} {\genK}_{\ol{i}} = {\genE}_{\ol{i}} {\genK}_{i}^{-1}, \qquad
	{v} {\genK}_{\ol{i}} {\genE}_{i-1} -  {\genE}_{i-1} {\genK}_{\ol{i}} = - {\genK}_{i}^{-1} {\genE}_{\ol{i-1}}, \\
& {\genK}_{\ol{i}} {\genF}_{i} - {v} {\genF}_{i} {\genK}_{\ol{i}} = - {\genF}_{\ol{i}} {\genK}_{i}, \qquad
	{v} {\genK}_{\ol{i}} {\genF}_{i-1} -  {\genF}_{i-1} {\genK}_{\ol{i}} = {\genK}_{i} {\genF}_{\ol{i-1}},\\
& {\genK}_{\ol{i}} {\genE}_{\ol{i}} + {v} {\genE}_{\ol{i}} {\genK}_{\ol{i}} = {\genE}_{i} {\genK}_{i}^{-1}, \qquad
	{v} {\genK}_{\ol{i}} {\genE}_{\ol{i-1}} +  {\genE}_{\ol{i-1}} {\genK}_{\ol{i}} =   {\genK}_{i}^{-1} {\genE}_{i-1}, \\
& {\genK}_{\ol{i}} {\genF}_{\ol{i}} + {v} {\genF}_{\ol{i}} {\genK}_{\ol{i}} =   {\genF}_{i} {\genK}_{i}, \qquad
	{v} {\genK}_{\ol{i}} {\genF}_{\ol{i-1}} +  {\genF}_{\ol{i-1}} {\genK}_{\ol{i}} = {\genK}_{i} {\genF}_{i-1}, \\
&
 {\genK}_{\ol{i}} {\genE}_{j} - {\genE}_{j} {\genK}_{\ol{i}} =  {\genK}_{\ol{i}} {\genF}_{j} - {\genF}_{j} {\genK}_{\ol{i}}
 = {\genK}_{\ol{i}} {\genE}_{\ol{j}} + {\genE}_{\ol{j}} {\genK}_{\ol{i}} =  {\genK}_{\ol{i}} {\genF}_{\ol{j}} + {\genF}_{\ol{j}} {\genK}_{\ol{i}}
	= 0, \mbox{ for } j \ne i, i-1; \\
({\rm QQ4}) \quad
& {\genE}_{i} {\genF}_{j} - {\genF}_{j} {\genE}_{i}
	= \delta_{i,j}  \frac{{\genK}_{i} {\genK}_{i+1}^{-1} - {\genK}_{i}^{-1}{\genK}_{i+1}}{{v} - {v}^{-1}}, \\
&
{\genE}_{\ol{i}} {\genF}_{\ol{j}} + {\genF}_{\ol{j}} {\genE}_{\ol{i}}
	= \delta_{i,j}  ( \frac{{\genK}_{i} {\genK}_{i+1} - {\genK}_{i}^{-1} {\genK}_{i+1}^{-1}}{{v} - {v}^{-1}}
	 + ({v} - {v}^{-1}) {\genK}_{\ol{i}} {\genK}_{\ol{i+1}} )  ,\\
& {\genE}_{i} {\genF}_{\ol{j}} - {\genF}_{\ol{j}} {\genE}_{i}
	= \delta_{i,j}  ( {\genK}_{i+1}^{-1} {\genK}_{\ol{i}}  - {\genK}_{\ol{i+1}} {\genK}_{i}^{-1} ) , \qquad
 {\genE}_{\ol{i}} {\genF}_{j} - {\genF}_{j} {\genE}_{\ol{i}}
	= \delta_{i,j}  ( {\genK}_{i+1} {\genK}_{\ol{i}}  - {\genK}_{\ol{i+1}} {\genK}_{i} ) ;\\
({\rm QQ5}) \quad
&{\genE}_{\ol{i}}^2 = -\frac{ {v} - {v}^{-1} }{{v} + {v}^{-1}} {\genE}_{i}^2, \quad
	{\genF}_{\ol{i}}^2 = \frac{ {v} - {v}^{-1} }{{v} + {v}^{-1}} {\genF}_{i}^2, \\
&
{\genE}_{i} {\genE}_{\ol{j}} - {\genE}_{\ol{j}} {\genE}_{i}
	=  {\genF}_{i} {\genF}_{\ol{j}} - {\genF}_{\ol{j}} {\genF}_{i}
	= 0 ,  \quad \mbox{ for } |i - j| \ne 1,
\\
&
{\genE}_{i} {\genE}_{j} - {\genE}_{j} {\genE}_{i} = {\genF}_{i} {\genF}_{j} - {\genF}_{j} {\genF}_{i}
= {\genE}_{\ol{i}}{\genE}_{\ol{j}}  + {\genE}_{\ol{j}}  {\genE}_{\ol{i}}= {\genF}_{\ol{i}} {\genF}_{\ol{j}}  + {\genF}_{\ol{j}} {\genF}_{\ol{i}}
	= 0 \quad \mbox{ for }\quad |i-j|  > 1, 	
\\
& {\genE}_{i} {\genE}_{i+1} - {v} {\genE}_{i+1} {\genE}_{i}
	= {\genE}_{\ol{i}} {\genE}_{\ol{i+1}} + {v} {\genE}_{\ol{i+1}} {\genE}_{\ol{i}}, \qquad
  {\genE}_{i} {\genE}_{\ol{i+1}} - {v} {\genE}_{\ol{i+1}} {\genE}_{i}
	= {\genE}_{\ol{i}} {\genE}_{i+1} - {v} {\genE}_{i+1} {\genE}_{\ol{i}}, \\
& {\genF}_{i} {\genF}_{i+1} - {v} {\genF}_{i+1} {\genF}_{i}
	= - ({\genF}_{\ol{i}} {\genF}_{\ol{i+1}} + {v} {\genF}_{\ol{i+1}} {\genF}_{\ol{i}}), \qquad
  {\genF}_{i} {\genF}_{\ol{i+1}}  - {v} {\genF}_{\ol{i+1}}  {\genF}_{i}
	=  {\genF}_{\ol{i}} {\genF}_{i+1} - {v} {\genF}_{i+1} {\genF}_{\ol{i}} ;\\
({\rm QQ6}) \quad
& {\genE}_{i}^2 {\genE}_{j} - ( {v} + {v}^{-1} ) {\genE}_{i} {\genE}_{j} {\genE}_{i} + {\genE}_{j}  {\genE}_{i}^2 = 0, \qquad
	{\genF}_{i}^2 {\genF}_{j} - ( {v} + {v}^{-1} ) {\genF}_{i} {\genF}_{j} {\genF}_{i} + {\genF}_{j}  {\genF}_{i}^2 = 0, \\
& {\genE}_{i}^2 {\genE}_{\ol{j}} - ( {v} + {v}^{-1} ) {\genE}_{i} {\genE}_{\ol{j}} {\genE}_{i} + {\genE}_{\ol{j}}  {\genE}_{i}^2 = 0, \qquad
	{\genF}_{i}^2 {\genF}_{\ol{j}} - ( {v} + {v}^{-1} ) {\genF}_{i} {\genF}_{\ol{j}} {\genF}_{i} + {\genF}_{\ol{j}}  {\genF}_{i}^2 = 0,  \\
& \qquad \mbox{ where } \quad |i-j| = 1.
\end{align*}
\end{defn}

\begin{rem}\label{induction_N}
Observe that the relations in (QQ3) imply
\begin{align*}
{\genE}_{\ol{j}}
&= - {v} {\genK}_{j+1} {\genK}_{\ol{j+1}} {\genE}_{j} + {v}^{-1} {\genE}_{j} {\genK}_{j+1}  {\genK}_{\ol{j+1}} , \\
{\genF}_{\ol{j}}
&= {v}  {\genK}_{j+1}^{-1} {\genK}_{\ol{j+1}} {\genF}_{j} - {v}^{-1} {\genF}_{j} {\genK}_{j+1}^{-1}  {\genK}_{\ol{j+1}},
\end{align*}
and the third relation  in  (QQ4) implies
\begin{align*}
{\genK}_{\ol{j}}
&=
	{\genE}_{j} {\genF}_{\ol{j}}  {\genK}_{j+1}  - {\genF}_{\ol{j}} {\genE}_{j} {\genK}_{j+1}  +  {\genK}_{j}^{-1} {\genK}_{\ol{j+1}} {\genK}_{j+1}\\
&= {\genE}_{j}  {\genK}_{\ol{j+1}} {\genF}_{j}
	 - {v}^{-1} {\genE}_{j}  {\genF}_{j}  {\genK}_{\ol{j+1}}
		- {v}  {\genK}_{\ol{j+1}} {\genF}_{j}  {\genE}_{j}
		+  {\genF}_{j}  {\genK}_{\ol{j+1}} {\genE}_{j}
		+  {\genK}_{j}^{-1} {\genK}_{\ol{j+1}} {\genK}_{j+1} .
\end{align*}
Hence, given $\genK_{\bar n}$ and by induction on {$j$} in descending order  from {$n-1$} to $1$,
one can obtain all other odd generators in $\Uvqn$. In other words,
$\Uvqn$ can be generated by $\{ {\genE}_j, {\genF}_j, {\genK}_i^{\pm1}, {\genK}_{\ol{n}}\where  1\le j \le n-1, \  1 \le i \le n \}$; see \cite{GJKKK} (cf., \cite[Rem.~2.2]{DLZ}).
This fact shows that it suffices to use the even generators together with ${\genK}_{\ol{n}}$ to describe the multiplication formulas for the regular representation.
\end{rem}
The supergroup $\Uvqn$ and Hecke-Clifford algebra $\bs\sH^c_r$ enjoys a beautiful duality. This involves the queer $q$-Schur superalgebras introduced in the previous section.

{\it From now on, we assume the ground ring $R=\sZ(:=\ZZ[\up,\up^{-1}])$ or $R=\QQ(\up)$ and define the following standardised queer $\up$-Schur superalgebras:
\begin{equation}\label{v-Sch}
\qSchvsZ:=\mathcal Q_q^s(n,r;\mathcal{Z}),\qquad\qSchvsQ:=\mathcal Q_q^s(n,r;\QQ(\up))\;\;(q=\up^2).
\end{equation}}
\quad Following \cite{Ol},
there is  a Schur-Weyl type duality
between $\Uvqn$ and $\boldsymbol\sH_{r}^c$ 
 on the tensor space ${\bs V}(n|n)^{\otimes r}$, where
 ${\bs V}(n|n):=V(n|n)_{\mathbb Q(\up)}$ is the superspace of both even and odd dimension $n$.
Meanwhile by \cite[Cor.~6.3]{DW2}, the tensor space $\bs V(n|n)^{\otimes r}$
can be identified with $\bigoplus_{\mu\in\Lambda(n,r)}x_{\mu}\boldsymbol{\sH}_{r}^c$ as $\boldsymbol\sH_{r}^c$-supermdoules.
 The following result is taken from \cite[Th.~5.3]{Ol} (with $\mathbb C$ replaced by $\QQ$) and  \cite[Th.~9.2]{DW1}.

\begin{prop}\label{DW2}
There is a superalgebra epimorphisms\footnote{The homomorphism considered in \cite[Prop.~6.4]{DW2} is a composition of $\Phi_r$, base change to $\mathbb C(\up)$, with the isomorphism $\mathcal Q^s_\up(n,r;\mathbb C(\up))\cong \mathcal Q_\up(n,r;\mathbb C(\up))$ given by Corollary \ref{super_al_iso}.} ${\Phi}_r:\Uvqn \to\text{\rm End}^s_{\bs{\sH}^c_{r}}\big(\bs V(n|n)^{\otimes r}\big)$ and
${\Psi}_r:\bs{\sH}_r^c \to\text{\rm End}^s_{\Uvqn}\big(\bs V(n|n)^{\otimes r}\big)$. Moreover,  ${\ker}({\Phi}_r)$ is generated by the elements
\begin{equation}\label{SvRelations}
{\genK}_1\cdots {\genK}_n-\up^r,\;\; ({\genK}_i-1)({\genK}_i-\up)\cdots({\genK}_i-\up^{r}), \;\;{\genK}_{\bar i} ({\genK}_i-\up)({\genK}_i-\up^2)\cdots({\genK}_i-\up^{r}),
\end{equation}
for all $1\leq i\leq n$.
\end{prop}

Note that $\qSchvsQ\cong \text{\rm End}^s_{\bs{\sH}^c_{r}}\big(\bs V(n|n)^{\otimes r}\big)$. We identify the two in the sequel.
A more structure-preserving $\QQ(\up)$-superalgebra homomorphism $\bs\xi_{n,r}$ from $\Uvqn$ to $\qSchvsQ$ will be constructed in \S6. For this purpose, we need a process of standardisation of the natural basis and fundamental multiplication formulas in $\qSchvsZ$.

We first standardise the defining basis for the generic Hecke-Clifford superalgebra $\sH_r^c$ over $\sZ$ and $\bs\sH_r^c$ over $\QQ(\up)$.
For any $i\in [1,r-1]$, set $\mathcal{T}_i=\up^{-1}T_i$ as usual, then the new generators $\mathcal{T}_i \in \sH_r^c$ satisfy the following relations (cf. \eqref{hc_generators}):
\begin{equation}\label{Hecke}
	\aligned &(\mathcal{T}_i - \up )(\mathcal{T}_i + \up^{-1}) = 0, \quad \\
	&\mathcal{T}_i\mathcal{T}_j =\mathcal{T}_j\mathcal{T}_i\,\;
	 (|i-j| >1),\quad\\
	&\mathcal{T}_i\mathcal{T}_{i+1}\mathcal{T}_i =\mathcal{T}_{i+1}\mathcal{T}_{i} \mathcal{T}_{i+1}\; (i\neq r),\endaligned
	\quad
	\mathcal{T}_i c_j=\begin{cases}c_j \mathcal{T}_i\,&\text{ if }j\neq i, i+1;\\
c_{i+1} \mathcal{T}_i, &\text{ if }j=i;\\
 c_i \mathcal{T}_i - (\up-\up^{-1})(c_i - c_{i+1}),&\text{ if }j=i+1.
 \end{cases}
\end{equation}

We now standardise the natural basis $\{\Phi_\Ad\}_\Ad$ defined in Proposition \ref{prop_PhiAPhiB}. For any subset $D\subseteq \mathfrak{S}_r$, let
$T_D=\sum_{w\in D}T_w$.

Recall from the definition of $T_\Ad$ in \eqref{eq cA}. It is obtained from the even part
$$T_A:=T_{\mathfrak{S}_{\ro(A)} d_A \mathfrak{S}_{\ro(A)}}=x_{\ro(A)}T_{d_A}\Sigma_A$$
by inserting the $c$-element $c_\Ad$ between $d_A$ and $\Sigma_A$, where $\Sigma_A=T _{ \D_{{\nu}_{A}} \cap \fS_{\co(A)}} $ denotes the ``tail term''.
Motivated by the standard basis element $\xi_A$ for $q$-Schur algebra $S_q(n,r)$ in \cite[(9.3.1)]{DDPW}, we may standardise the even part $T_A$,  by setting
\begin{equation}\label{mcTA-TA}
\mathcal{T}_A=\up^{-\ell(d_A^+)}T_A,
\end{equation}
 where
 $d^+_A$ is the longest element in ${\mathfrak{S}_{\ro(A)} d_A \mathfrak{S}_{\ro(A)}}$.

We need to define the standardisation of $c_\Ad$. Recall
the $c$-elements $c_{{q},i,j} $  and $c'_{{q},i,j}$ in \cite[Section 4]{DW2}:
\begin{equation*}
\begin{aligned}
	c_{{q},i,j} &= {q}^{j-i}c_i  + \cdots + {q} c_{j-1} + c_j=\up^{j-i}({\up}^{j-i}c_i  + \cdots + {\up^{-(j-i)+2}} c_{j-1} + {\up}^{-(j -i)}c_j),\\
c'_{{q},i,j} &= c_i + {q} c_{i+1} + \cdots + {q}^{j-i}c_j=\up^{j-i}({\up}^{-(j-i)}c_i  + \cdots + {\up^{(j-i)-2}} c_{j-1} + {\up}^{(j-i)}c_j).
\end{aligned}
\end{equation*}
Since $q=\up^2$, we may introduce the standard version $\normc_{{\up},i,j}$ and $\normc'_{{\up},i,j}$ such that
\begin{equation*}
	c_{{q},i,j}=\up^{j-i}\normc_{{\up},i,j}, \qquad c'_{{q},i,j}=\up^{j-i}\normc'_{{\up},i,j}.
\end{equation*}
In other words, we have
\begin{equation}\label{normc}
\begin{aligned}
\normc_{{\up},i,j}&={\up}^{j-i}c_i  + \cdots + {\up^{-(j-i)+2}} c_{j-1} + {\up}^{-(j -i)}c_j,\\
\normc'_{{\up},i,j}&={\up}^{-(j-i)}c_i  + \cdots + {\up^{(j-i)-2}} c_{j-1} + {\up}^{(j-i)}c_j,
\end{aligned}
\end{equation}
Thus, for any $\Ad=(A^{\bar 0}|A^{\bar 1})=(a^{\bar 0}_{i,j}|a^{\bar 1}_{i,j})\in \MNZ(n,r)$, substituting every $c_{q,i,j}$ in $c_{\Ad}$ by $\up^{(j-i)}\normc_{\up,i,j}$, we obtain $c_{\Ad}=\up^{\sum_{i,j}a^{\bar1}_{i,j}a^{\bar0}_{i,j}}\normc_{\Ad}.$
Putting $A^{\ol{0}}* A^{\ol{1}}:=\sum_{i,j}a^{\bar0}_{i,j}a^{\bar1}_{i,j}$ (a ``dot product''), we obtain the standard form of $c_\Ad$:
\begin{equation}\label{CA-MCA}
\normc_{\Ad}=\up^{-A^{\ol{0}}* A^{\ol{1}}}c_\Ad.
\end{equation}

Combining the two standardisations \eqref{mcTA-TA} and \eqref{CA-MCA} gives the following standard form of $T_\Ad$:
  $$\mathcal{T}_{\Ad}=\up^{-\ell(d^+_A)-{A^{\ol{0}}* A^{\ol{1}}}}T_\Ad.$$
In particular, if $A=(\mu|O)$ for some $\mu\in\Lambda(n,r)$, then $\mathcal T_{(\mu|O)}=\up^{-\ell(w_{0,\mu})}x_\mu$, where $w_{0,\mu}$ is the longest element in $\fS_\mu$.

\begin{defn}\label{def[A]} For $\Ad=(A^{\bar 0}|A^{\bar 1})=(a^{\bar 0}_{i,j}|a^{\bar 1}_{i,j})\in \MNZ(n,r)$, define
$[\Ad]\in \SQqnrR$ such that
\begin{equation}
[\Ad](\mathcal{T}_{(\mu|O)}h)=\delta_{\co(A),\mu}(-1)^{\wp(\Ad)\wp(h)}\mathcal{T}_{\Ad}h, \forall h\in \HCR.
\end{equation}
\end{defn}

\begin{prop}\label{relation-[A]-PhiA}
 For $\Ad=(a^{\bar 0}_{i,j}|a^{\bar 1}_{i,j})\in \MNZ(n,r)$ with base $A=(a_{i,j})$, let
 \begin{equation}\label{norm-fl}
\Norm(\Ad):=\sum_{i\geq k,j<l}a_{i,j}a_{k,l}+\sum_{i,j}a^{\bar 1}_{i,j}a^{\bar0}_{i,j}.
\end{equation}
Then
 $[\Ad]=\up^{-\Norm(\Ad)}\Phi_\Ad $.\underline{}
\end{prop}
See \cite[(13.2.5)]{DDPW} for a geometric interpretation of the number
$\sum_{i\geq k,j<l}a_{i,j}a_{k,l}$.

\begin{proof} By definition (see Proposition \ref{prop_PhiAPhiB}), it is clear (cf. \cite[(9.3.1)]{DDPW}) that
$$[\Ad]=\up^{-\ell(d^+_A)+\ell(w_{0,\co(A)})-A^{\ol{0}}* A^{\ol{1}}}\Phi_\Ad.$$
On the other hand, \cite[(13.2.5), Lem.~13.10]{DDPW} implies
 \begin{equation}\label{hatell}
 \widehat\ell(A):=\ell(d^+_A)-\ell(w_{0,\co(A)})=\sum_{i\geq k,j<l}a_{i,j}a_{k,l}=\partial(\Ad)-A^{\ol{0}}* A^{\ol{1}},
 \end{equation}
proving the assertion.
\end{proof}

By Propositions \ref{relation-[A]-PhiA} and \ref{prop_PhiAPhiB}, we have immediately the following.
\begin{cor}\label{sdbasis}The set
$\{ [\Ad] \where  \Ad \in \MNZ(n,r)\}$ forms a basis for the $\sZ$-superalgebra $\mathcal Q_\up^s(n,r)$.
\end{cor}
This basis is called the {\it standard (integral) basis} for $\mathcal Q_\up^s(n,r)$. In the next section, we will turn the fundamental multiplication formulas into standard multiplication formulas, especially for those discussed in Remark \ref{Md}. The following result will simplify  the computation for the number $\partial(\Md)-\partial(\Ad)$, where $\Md$ is the matrix mentioned in Remark \ref{Md} with base  $A^+_{h,k}$ or $A^-_{h,k}$.

For any $A\in M_n(\NN)$, let
\begin{equation}\label{partialrow}
\BK(h,k):=\BK(h,k)(A)=\sum_{j>k}a_{h,j},\qquad
\AK(h,k):=\AK(h,k)(A)=\sum_{j<k}a_{h,j}.
\end{equation}
\begin{lem}\label{hatell1}
Maintain the notation above. We have
$$
\widehat\ell(A^+_{h,k})-\widehat\ell(A)=\AK(h,k)-\BK(h+1,k),\quad
\widehat\ell(A^-_{h,k})-\widehat\ell(A)=-\AK(h,k)+\BK(h+1,k).$$
\end{lem}
\begin{proof}Let $A'=A^+_{h,k}=A+E_{h,k}-E_{h+1,k}=(a_{i,j}')$. Then
$$\aligned
\widehat\ell(A^+_{h,k})-\widehat\ell(A)&=\sum_{i\geq k,j<l}a'_{i,j}a'_{k,l}-\sum_{i\geq k,j<l}a_{i,j}a_{k,l}\\
&=\sum_{i\geq h,j<k}a_{i,j}-\sum_{i\geq h+1,j<k}a_{i,j}+\sum_{i\leq h,j>k}a_{i,j}-\sum_{i\leq h+1,j>k}a_{i,j}=\AK(h,k)-\BK(h+1,k),
\endaligned$$
as desired. The proof for the other case is similar.
\end{proof}

\section{Standard multiplication formulas in $\mathcal Q^s_\up(n,r)$ }\label{SMFQS}
In this section, we standardise the fundamental multiplication formulas for $\Phi_\Ad$ arising from \cite[\S\S5-7]{DGLW}  via Proposition \ref{prop_PhiAPhiB}.  In other words, we derive multiplication formulas in $\qSchvsZ$ relative to the standard basis $\{[\Ad]\}_{\Ad}$ introduced in the previous section by modifying those in \cite{DGLW}.

The modification process is in two steps. First, we use Proposition \ref{prop_PhiAPhiB} to turn a formula of the form $\phi_\Bd\cdot\phi_\Ad= \sum_{\Md} {  \gamma^{\Md}_{\Bd,\Ad}} \phi_\Md$ in the natural basis $\{\phi_\Ad\}_{\Ad}$ to a formula of the form $\Phi_\Bd\cdot\Phi_\Ad= \sum_{\Md}{  \scc_{\Bd,\Ad,\Md} }\Phi_\Md$ in the superhomomorphism basis $\{\Phi_\Ad\}_{\Ad}$, the natural basis for $\qSchvsZ$. Second, compute the number
$\partial(\Md)-\partial(\Ad)-\partial(\Bd)$ in
the structure constants ${ \scp_{\Bd,\Ad,\Md}:=\up^{\partial(\Md)-\partial(\Ad)-\partial(\Bd)}\scc_{\Bd,\Ad,\Md}}$ relative to the standard basis $\{[\Ad]\}_{\Ad}$. (In other words, the first step determines the sign and the second step determines the power of $\up$.)
The standard multiplication formulas have the form
\begin{equation}\label{SMF}
[\Bd]\cdot[\Ad]=\sum_{\Md\in\MNZ(n,r)}{ \scp_{\Bd,\Ad,\Md}}[\Md].
\end{equation}

{\it As usual, for notational simplicity, we make the following convention: if a term ${ \scp_{\Bd,\Ad,\Md}}[\Md]$ with $\Md\notin \MNZ(n,r)$ appears in a multiplication formula in $\qSchvsZ$, then we regard the entire term as 0.} 

Keeping the notations in \cite{DGLW}, we also regard any $\lambda \in \Lambda(n,r)$
 as a diagonal matrix
 when writing $\lambda+B:={\diag}(\lambda_1,\ldots,\lambda_n)+B$, for $B\in\MN(n)$.

We first observe that,  for any $a\in\mathbb{Z}_{\geq0}$ and $q=\up^2$,
\begin{equation}\label{QN2}
[\![a+1]\!]_{q,q^2}:=[\![a+1]\!]_{q}-[\![a+1]\!]_{q^2}=-(\up-\up^{-1})\up^{2a+1}\left[a+1\atop 2\right].
\end{equation}
This polynomial occurs in some fundamental multiplication formulas in \cite{DGLW}.

For $\lambda,\mu\in\Lambda(n,r)$, let $\Bd$ in \eqref{SMF} be one of the following even matrices
 \begin{equation}\label{EVEN}
\Dd_\mu:=(\mu|O),\;\;
 \Ed_{h,\lambda}:=(\lambda + E_{h, h+1}-E_{h+1, h+1}|O),\;\; \Fd_{h,\lambda}:=(\lambda-E_{h,h} + E_{h+1,h} |O).
 \end{equation}
Following the two step modification process, we obtain three standard multiplication formulas in this even case.
For notational simplicity, we also introduce the following, using \eqref{partialrow},
\begin{equation}\label{f_h/g_h}
  f_h(A,k)=\BK(h,k)({ A})-\BK(h+1,k)({ A});\qquad g_h(A,k)=-\AK(h,k)({ A})+\AK(h+1,k)({ A}).
\end{equation}

\begin{lem}\label{phiupper-even-norm}
Let {$h \in [1,n-1]$} and {$\Ad =  ({\SEE{a}_{i,j}} | {\SOE{a}_{i,j}})   \in \MNZ(n,r)$}.
Then, for $\lambda,\mu\in\Lambda(n,r)$, the following multiplication formulas hold in $\qSchvsZ$:
\begin{align*}
(1)\;\;\;\;[\Dd_\mu] [\Ad] =& \delta_{\mu,\ro(A)}[\Ad],\qquad
	 [\Ad] [\Dd_\mu]  =\delta_{\mu,\co(A)}[\Ad].\\ 
(2)\;\; {[\Ed_{h,\lambda}][\Ad]} = &\delta_{\lambda,\ro(A)}\sum_{k=1}^n \up^{f_h(A,k)}\Big\{\up^{\SOE{a}_{h+1,k}}[\SEE{a}_{h,k} + 1]
[\SE{A} + E_{h,k}  - E_{h+1, k} | \SO{A} ]\\
	&+\up^{-a^{\bar0}_{h+1,k}}[\SE{A}  | \SO{A} + E_{h,k} - E_{h+1, k}  ]\\
&-(\up-\up^{-1})\up^{-a^{\bar0}_{h+1,k}}\left[{a_{h, k} +1}\atop 2\right][\SE{A} + 2E_{h,k} | \SO{A}  -E_{h,k} - E_{h+1,k} ]\Big\}.\\ 
(3)\;\;{[\Fd_{h,\lambda}]  [\Ad] }=&\delta_{\lambda,\ro(A)}\sum_{k=1}^n \up^{g_h(A,k)}\Big\{\up^{-a^{\bar1}_{h,k}}[\SEE{a}_{h+1, k} +1]
	[\SE{A} - E_{h,k} + E_{h+1, k} | \SO{A} ]\\
	 &+{\up}^{ a^{\bar0}_{h, k} }
	[\SE{A}  |\SO{A} - E_{h,k} + E_{h+1,k}]\\
&-(\up-\up^{-1}){\up}^{ a^{\bar0}_{h, k} }\left[{a_{h+1,k}+1}\atop 2\right]
	[\SE{A} + 2E_{h+1,k} | \SO{A}  - E_{h,k} - E_{h+1,k}]
\Big\}.
\end{align*}
\end{lem}

\begin{proof} In all three cases, the even parity of $\Bd$ implies that the sign in \eqref{Phi_structure} can be dropped.
Thus, the formula in (1) follows easily from \cite[Th.~5.3(1)]{DGLW}.

For (2),
by \cite[Th.~5.3(2)]{DGLW} and Proposition \ref{prop_PhiAPhiB}, we have for $\la=\ro(A)$
\begin{equation}\label{DGLW-thm-5.3-2}
\begin{aligned}
 \Phi_{\Ed_{h,\lambda}}  \Phi_\Ad &=\sum_{k=1}^n \Big\{
	{q}^{\BK(h,k) + \SOE{a}_{h+1,k}}  \STEP{ \SEE{a}_{h,k} + 1}
		\Phi_{(\SE{A} + E_{h,k}  - E_{h+1, k} | \SO{A} )}\\
	&\quad+{q}^{\BK(h,k)}  \Phi_{(\SE{A}  | \SO{A} + E_{h,k} - E_{h+1, k}  )}  +
	{q}^{\BK(h,k)-1} \STEPPD{a_{h,k}+1}
	\Phi_{(\SE{A} + 2E_{h,k} | \SO{A}  -E_{h,k} - E_{h+1,k} )}\Big
\}.
\end{aligned}
\end{equation}
Thus, in this case, the matrices $\Md$ in \eqref{SMF} have the form
$$\Md=\begin{cases}
(\SE{A} + E_{h,k}  - E_{h+1, k} | \SO{A} ),&\text{ if }a^{\bar0}_{h+1,k}>0;\\
(\SE{A}  | \SO{A} + E_{h,k} - E_{h+1, k}  ),&\text{ if }a^{\bar1}_{h,k}=0,a^{\bar1}_{h+1,k}=1;\\
(\SE{A} + 2E_{h,k} | \SO{A}  -E_{h,k} - E_{h+1,k} ),&\text{ if }a^{\bar1}_{h,k}=1,a^{\bar1}_{h+1,k}=1.
\end{cases}$$
To compute the number $\partial(\Md)-\partial(\Ad)-\partial(\Bd)$, by \eqref{norm-fl} and Lemma \ref{hatell1},
\begin{equation}\label{DGLW-thm-5.3-2-pf-1}
\begin{aligned}
&\Norm(\Ed_{h,\lambda})=\lambda_h=\sum_{j=1}^n a_{h,j}=:{\bf r}_h,\\
&\Norm(\SE{A} + E_{h,k}  - E_{h+1, k} | \SO{A} )-\Norm(\Ad)=\widehat\ell(A^+_{h,k})-\widehat\ell(A)+A^{\ol{0},+}_{h,k}*A^{\ol{1}}-A^{\ol{0}}*A^{\ol{1}}
\\
&\qquad=\AK(h,k)-\BK(h+1,k)+a^{\bar1}_{h,k}-a^{\bar1}_{h+1,k}.
\end{aligned}
\end{equation}

Observing that $(\SE{A}  | \SO{A}+ E_{h,k}  - E_{h+1, k} ) \mbox{ and }(\SE{A} + 2E_{h,k} | \SO{A}-E_{h,k}  - E_{h+1, k} )$  have the same base matrix $(\SE{A} +\SO{A}+ E_{h,k}  - E_{h+1, k} )$ as $(\SE{A} + E_{h,k}  - E_{h+1, k} | \SO{A} )$, applying \eqref{DGLW-thm-5.3-2-pf-1} leads to
 \begin{equation}\label{DGLW-thm-5.3-2-pf-2}
\begin{aligned}
&\Norm(\SE{A}  | \SO{A}+ E_{h,k}  - E_{h+1, k} )=\Norm(\Ad)+\AK(h,k)-\BK(h+1,k)+a^{\bar0}_{h,k}-a^{\bar0}_{h+1,k},\\
&\Norm(\SE{A} + 2E_{h,k} | \SO{A}-E_{h,k}  - E_{h+1, k} )=\Norm(\Ad)+\AK(h,k)-\BK(h+1,k)-a^{\bar0}_{h,k}-a^{\bar0}_{h+1,k}.
\end{aligned}
\end{equation}

Continuing the computations on the coefficients,  for  $\Md=(\SE{A} + E_{h,k}  - E_{h+1, k} | \SO{A}) $, the coefficient in \eqref{DGLW-thm-5.3-2} has the form  ($q=\up^2$):
$$
{q}^{\BK(h,k) + \SOE{a}_{h+1,k}}  \STEP{ \SEE{a}_{h,k} + 1}=\up^{2\BK(h,k) + 2\SOE{a}_{h+1,k}+\SEE{a}_{h,k} }[\SEE{a}_{h,k} + 1].
$$
By \eqref{DGLW-thm-5.3-2-pf-1}, we obtain the power of $\up$
\begin{equation}
\begin{aligned}
&2\BK(h,k) + 2\SOE{a}_{h+1,k}+\SEE{a}_{h,k}+\Norm(\Md) -(\Norm(\Ed_{h,\lambda})+\Norm(\Ad))\\
&=2\BK(h,k) + 2\SOE{a}_{h+1,k}+\SEE{a}_{h,k}+ \Norm(\Ad)+\AK(h,k)-\BK(h+1,k)+a^{\bar1}_{h,k}-a^{\bar1}_{h+1,k}-({\bf r}_h+\Norm(\Ad))\\
&=\BK(h,k)-\BK(h+1,k)+\SOE{a}_{h+1,k}=f_h(A,k)+\SOE{a}_{h+1,k}.
\end{aligned}
\end{equation}

For  $\Md=(\SE{A}  | \SO{A}+ E_{h,k}  - E_{h+1, k} )\in \MNZ(n,r)$ with $a^{\bar1}_{h,k}=0$ and $a^{\bar1}_{h+1,k}=1$.
By \eqref{DGLW-thm-5.3-2-pf-1} and \eqref{DGLW-thm-5.3-2-pf-2}, the power of $\up$ becomes
\begin{equation}
\begin{aligned}
&2\BK(h,k)+\Norm(\Md)-(\Norm(\Ed_{h,\lambda})+\Norm(\Ad))\\
&=2\BK(h,k)+\Norm(\Ad)+\AK(h,k)-\BK(h+1,k)+a^{\bar0}_{h,k}-a^{\bar0}_{h+1,k}-(\Norm(\Ed_{h,\lambda})+\Norm(\Ad))\\
&=\BK(h,k)-\BK(h+1,k)-a^{\bar0}_{h+1,k}=f_h(A,k)-a^{\bar0}_{h+1,k}.
\end{aligned}
\end{equation}

Finally, for $\Md=(\SE{A} + 2E_{h,k} | \SO{A}-E_{h,k}  - E_{h+1, k} )\in \MNZ(n,r)$ with $a^{\bar1}_{h,k}=a^{\bar1}_{h+1,k}=1$,
We have, by \eqref{QN2},
$${q}^{\BK(h,k)-1} \STEPPD{a_{h, k} +1}=-\up^{2(\BK(h,k) -1)+2a_{h, k} +1}(\up-\up^{-1})\left[{a_{h, k} +1}\atop 2\right].$$
The power of $\up$ has the form, by \eqref{DGLW-thm-5.3-2-pf-1} and \eqref{DGLW-thm-5.3-2-pf-2},
\begin{equation}
\begin{aligned}
&2(\BK(h,k) -1)+2a_{h, k} +1+\Norm(\SE{A} + 2E_{h,k} | \SO{A}-E_{h,k}  - E_{h+1, k} )-(\Norm(\Ed_{h,\lambda})+\Norm(\Ad))\\
&=2(\BK(h,k) -1)+2a_{h, k} +1+\Norm(\Ad)+\AK(h,k)-\BK(h+1,k)-a^{\bar0}_{h,k}-a^{\bar0}_{h+1,k}-{\bf r}_h-\Norm(\Ad)\\
&={\AK(h,k)-\AK(h+1,k)-a^{\bar0}_{h+1,k}}=f_h(A,k)-a^{\bar0}_{h+1,k}.
\end{aligned}
\end{equation}
Substituting all three cases into \eqref{DGLW-thm-5.3-2} yields (2).
Formula (3) can be proved similarly.\end{proof}

We now derive the standard multiplication formulas \eqref{SMF} with $\Bd$ being one of the following three matrices with odd parity:
$(\la|E_{h,h}),\qquad(\mu|E_{h, h+1}),\qquad(\nu|E_{h+1,h}).$

Recall the preorder  {$\preceq$} over $\MN(n)$ from \cite{BLM} and its associated equivalence relation $\preeq$: for {$A=(a_{i,j}), B=(b_{i,j}) \in \MN(n)$},
\begin{equation}\label{preorder}
\aligned
	B \preceq A &\iff
\begin{cases}
	\sum_{i \le s, j\ge t } b_{i,j} \le \sum_{i \le s, j\ge t } a_{i,j}, & \mbox{ for all } s<t; \\
	\sum_{i \ge s, j\le t } b_{i,j} \le \sum_{i \ge s, j\le t } a_{i,j}, & \mbox{ for all } s>t.
\end{cases}\\
B \preeq A &\iff B \preceq A,\;\; A \preceq B.
\endaligned
\end{equation} Then,
$B\preeq A$  if and only if  {$B^{\pm}= A^{\pm}$}, where $X^\pm$ is the matrix obtained from $X$ with all diagonal entries replaced by 0s. We also write {$B \prec A $} for {$B \preceq A$} but {$B^{\pm} \ne A^{\pm}$}.

Similar to \eqref{EVEN}, we introduce the following matrices:
 \begin{equation}\label{ODD}
\Dd_{\ol{i},\lambda}:=(\lambda-E_{i,i}|E_{i,i}),\;\;
 \Ed_{\ol{h},\lambda}:=(\lambda-E_{h+1, h+1}|E_{h,h+1}),\;\; \Fd_{\ol{h},\lambda}:=(\lambda-E_{h,h}  |E_{h+1,h}).
 \end{equation}
We also set (cf. \eqref{f_h/g_h}), for  $A$ as above and $h,k\in[1,n]$,
\begin{equation}\label{d_hA}
d_{h}(A,k):=\BK(h,k)(A)-\AK(h,k)(A).
\end{equation}
The following is the standardisation of \cite[Th.~6.2]{DGLW}.

\begin{lem}\label{normlized-phidiag1}
For {$\Ad =  (A^{\bar0}|A^{\bar 1})=(a^{\bar0}_{i,j}|a^{\bar 1}_{i,j})   \in \MNZ(n,r)$} with base $A=A^{\bar0}+A^{\bar 1},$ { $ \lambda=\ro(A)$} and  $h\in[1,n]$.
\begin{enumerate}
\item	If {$A$}  satisfies the SDP condition on the $h$-th row,
	 then we have in $\qSchvsZ$:
\begin{equation*}
\begin{aligned}
&{{ [{\Dd_{\bar h,\lambda}}]} [\Ad]}=\\
	&\;\;
\sum_{ k=1}^n
{(-1)}^{\wp(\Ad)+{\widetilde{a}}^{\ol{1}}_{h,k} } {\up}^{d_{h}(A,k) }
\Big\{
	[{\SE{A} -E_{h,k}|\SO{A}+ E_{h,k}}]
	+[a_{h,k}]_{\up^2}
	[{\SE{A} +  E_{h,k}| \SO{A} - E_{h,k}}]
\Big\}\\
&\;\;=:\sdpNHk(h,\Ad).
\end{aligned}
\end{equation*}
\item In general, we have
$${{ [{\Dd_{\bar h,\lambda}}]} [\Ad]}=\begin{cases}
\sdpNHk(h,\Ad)+ \sum_{\Bd,
 				 {B} \prec { {A} }
 			} { \scp_{{ {\Dd_{\bar h,\lambda}},\Ad},\Bd}}  [{\Bd}],&\text{ if }h<n;\\
			\sdpNHk(n,\Ad),&\text{ if }h=n,
			\end{cases}$$
\end{enumerate}
where $\Bd\in \MNZ(n, r)$ and ${ \scp_{{ {\Dd_{\bar h,\lambda}}},\Ad,\Bd}} \in\sZ$.
\end{lem}

\begin{proof} We first prove (1). Since $\wp({ {\Dd_{\bar h,\lambda}}})=1$, by
\cite[Th.~6.2]{DGLW}(1) and Proposition \ref{prop_PhiAPhiB},
\begin{equation}\label{DGLW-thm-6.2}
\begin{aligned}
&{{ [{\Dd_{\bar h,\lambda}}]}[\Ad]}=\up^{-\Norm(\Dd_{\bar h})-\Norm(\Ad)}\Phi_{\Dd_{\bar h}} \Phi_\Ad\\
	&=(-1)^{\wp(\Ad)}\up^{-\Norm(\Dd_{\bar h})-\Norm(\Ad)}
\sum_{ k=1}^n
{(-1)}^{ {\SOE{\widetilde{a}}}_{h,k} } {q}^{\BK(h,k) }
\Big\{
	 \Phi_{(\SE{A} -E_{h,k}|\SO{A}+ E_{h,k})}
	+ {\STEPP{a_{h,k}}}
	\Phi_{(\SE{A} +  E_{h,k}| \SO{A} - E_{h,k})}
\Big\}.\\
\end{aligned}
\end{equation}
We want to write this as a linear combination of $[\Md]$ as in \eqref{SMF}. First, we see from
 \eqref{norm-fl}
\begin{equation}\label{DGLW-thm-6.2-pf1}
\begin{aligned}
&\Norm({ {\Dd_{\bar h,\lambda}}})=\lambda_h-1=\ro(A)_h-1=\AK(h,k)+\BK(h,k)+a_{h,k}-1,\\
&\Norm(\SE{A} -E_{h,k}|\SO{A}+ E_{h,k})=\Norm(\Ad)+a^{\bar0}_{h,k}-1,\\
&\Norm(\SE{A} +E_{h,k}|\SO{A}-E_{h,k})=\Norm(\Ad)-a^{\bar0}_{h,k}.
\end{aligned}
\end{equation}

For $\Md=(\SE{A} -E_{h,k}|\SO{A}+ E_{h,k})$ (forcing $a^{\bar1}_{h,k}=0$),  \eqref{DGLW-thm-6.2-pf1} implies
\begin{equation}\label{DGLW-thm-6.2-pf2}
\begin{aligned}
&2\BK(h,k)+\Norm(\SE{A} -E_{h,k}|\SO{A}+ E_{h,k})-(\Norm({ {\Dd_{\bar h,\lambda}}})+\Norm(\Ad))\\
&=2\BK(h,k)+\Norm(\Ad)+a^{\bar0}_{h,k}-1-(\AK(h,k)+\BK(h,k)+a_{h,k}-1)-\Norm(\Ad)\\
&=\BK(h,k)-\AK(h,k)=d_{h}(A,k).
\end{aligned}
\end{equation}

For $\Md=(\SE{A} +  E_{h,k}| \SO{A} - E_{h,k})$ (forcing $a^{\bar1}_{h,k}=1$), since ${\STEPP{a_{h,k}}}=\up^{2(a_{h,k}-1)}[a_{h,k}]_{\up^2}$, we have by \eqref{DGLW-thm-6.2-pf1},
\begin{equation}\label{DGLW-thm-6.2-pf3}
\begin{aligned}
&2\BK(h,k)+2(a_{h,k}-1)+\Norm(\SE{A} +E_{h,k}|\SO{A}-E_{h,k})-\Norm({ {\Dd_{\bar h,\lambda}}})+\Norm(\Ad))\\
&=2\BK(h,k)+2(a_{h,k}-1)+\Norm(\Ad)-a^{\bar0}_{h,k}-(\AK(h,k)+\BK(h,k)+a_{h,k}-1)-\Norm(\Ad)\\
&=\BK(h,k)-\AK(h,k)=d_{h}(A,k).
\end{aligned}
\end{equation}
Substituting \eqref{DGLW-thm-6.2-pf2} and \eqref{DGLW-thm-6.2-pf3} into \eqref{DGLW-thm-6.2} yields (1).

{ We now prove (2). The case  when $h=n$  follows from the last assertion of Theorem \ref{CdA1} and (1). The general case when $h<n$ follows from \cite[Th.~6.2(2)]{DGLW} by setting
${ \scp_{{ {\Dd_{\bar h,\lambda}}},\Ad,\Bd}} =(-1)^{\wp(\Ad)}\up^{\Norm(\Bd)-\Norm({ {\Dd_{\bar h,\lambda}}})-\Norm(\Ad)}{f}^{{ {\Dd_{\bar h,\lambda}}},\Ad}_{\Bd}$.}
\end{proof}

The next result is the standardisation of  \cite[Ths.~6.3--4]{DGLW}. In addition to the notation $g_h(A,k)=-\AK(h,k)+\AK(h+1,k)$ defined in \eqref{f_h/g_h},  we also introduce the new notation
\begin{equation}\label{g-hk}
\OG_h(A,k)=-\AK(h,k)-\AK(h+1,k),\;\;\text{ where}\;\; \AK(j,k)=\AK(j,k)(A).
\end{equation}

\begin{lem}\label{standard-phiupper1}
	For {$h \in [1,n-1]$}  and {$\Ad =(A^{\bar0}|A^{\bar 1})=  ({\SEE{a}_{i,j}} | {\SOE{a}_{i,j}})   \in \MNZ(n,r)$} with base $A=A^{\bar0}+A^{\bar 1}$ {\color{black} and $\lambda=\ro(A)$}.
\begin{enumerate}
\item	Suppose that, for every $k\in[1,n]$ such that $a_{h+1,k}>0$ and $A$ satisfies the SDP condition at $(h,k)$ if $a_{h,k}>0$ and  $\llcm^{h,k}=0$ if $a_{h,k}=0$.
	Then
	we have in $\qSchvsZ$
\begin{equation}\label{oddUP}
\aligned
	{{ [{\Ed_{\bar h,\lambda}}]}[\Ad]}
	=& (-1)^{\wp(\Ad)}\sum_{k=1}^n \up^{f_h(A,k)}\Big \{
	{(-1)}^{{\SOE{\widetilde{a}}}_{h-1,k}} \up^{a^{\bar1}_{h+1,k}} [{\SE{A} - E_{h+1, k}| \SO{A}  + E_{h,k} }] \\
	& +
	{(-1)}^{{\SOE{\widetilde{a}}}_{h,k} + 1}
	\up^{-a^{\bar0}_{h+1,k}} [\SEE{a}_{h,k} +1]
		[{\SE{A} + E_{h,k}| \SO{A} - E_{h+1, k} }]  \\
	&  +
	{(-1)}^{{\SOE{\widetilde{a}}}_{h-1,k} } \up^{\SOE{a}_{h+1,k}}(\up-\up^{-1}) \left[{a_{h,k}+1}\atop 2\right]
		[{\SE{A} + 2 E_{h,k}  -E_{h+1, k} |\SO{A} - E_{h,k}}]
\Big\}\\
=&:\sdpNHe(h,\Ad)\\
\endaligned
\end{equation}
\hspace{-.5cm} In general, we have
$${ [{\Ed_{\bar h,\lambda}}]} [\Ad]=\sdpNHe(h,\Ad) + \sum_{
 			\substack{\Bd \in \MNZ(n, r) \\
 				 {\exists k,\, B} \prec { {A^+_{h.k}} }
 			}
 			}  { \scp_{{ {\Ed_{\bar h,\lambda}}},\Ad,\Bd}} [{\Bd}]\;\;\;({ \scp_{{ {\Ed_{\bar h,\lambda}}},\Ad,\Bd}} \in\mathcal{Z}).$$
			
\item	  If $A$ satisfies the SDP condition on the $h$-th row, 
then we have in $\qSchvsZ$
\begin{equation}
 \begin{aligned}
 {{ [{\Fd_{\bar h,\lambda}}]} [\Ad]}=&(-1)^{\wp(\Ad)}\sum_{k=1}^n{(-1)}^{{\SOE{\widetilde{a}}}_{h,k}} \up^{\OG_h(A,k)}\Big\{
	\up^{-a^{\bar1}_{h,k}}	[{\SE{A} - E_{h,k}| \SO{A}+ E_{h+1, k}}] \\
	& \qquad -
	{\up}^{-a^{\bar1}_{h,k}}(\up-\up^{-1}) \left[{a_{h+1, k} +1}\atop 2\right]
	[{\SE{A} - E_{h,k} + 2E_{h+1, k} | \SO{A}  -E_{h+1, k}}] \\
	& \qquad + {\up}^{a^{\bar0}_{h,k}}
	[\SEE{a}_{h+1,k}+1]
	[{\SE{A} + E_{h+1,k} | \SO{A}  - E_{h,k} }]
\Big\}=:\sdpNHf(h,\Ad).
 \end{aligned}
 \end{equation}
\hspace{-.5cm} In general, we have, for some $ { \scp_{\Fd_{\bar h,\lambda},\Ad,\Bd}} \in \mathcal{Z}$,
$${{ [{\Fd_{\bar h,\lambda}}]}[{\Ad}]}=\sdpNHf(h,\Ad)+(\up-\up^{-1})\text{\rm HH}\overline{\textsc f}(h,\Ad)+\sum_{\Bd\in \MNZ(n,r)\atop \exists k, B \prec{A}^-_{h,k}}  { \scp_{\Fd_{\bar h,\lambda},\Ad,\Bd}} [{\Bd}],$$
where, with notation \eqref{A01hk} and $\Ahkomm:=\SO{A}  -E_{h,k} - E_{h+1,k}$,
\begin{equation*}
\begin{aligned}
&\text{\rm HH}\overline{\textsc f}(h,\Ad)=(-1)^{\wp(\Ad)}\sum_{k=1}^n\sum_{l=1}^{k-1}(-1)^{\tilde{a}^{\bar1}_{h,l}}~\cdot \up^{2(\overrightarrow{\bf r}^{l}_{h+1}{-\overrightarrow{\bf r}_{h+1}^{k-1}})+\OG_h(A,k)+a_{h+1,l}+a^{\bar0}_{h,k}} ~\Big\{\up^{-a_{h,k}}[\SEE{a}_{h+1, k} +1] \\
&\quad\qquad\cdot([{\Ahkzm -E_{h+1,l}| \SO{A}+E_{h+1,l}}]-{[{a_{h+1,l}}]_{\up^2}}[{\Ahkzm +E_{h+1,l}| \SO{A}-E_{h+1,l}}])\\
&\quad\quad+([{{\SE{A} -E_{h+1,l} |\Ahkom+E_{h+1,l}}}]-{[{a_{h+1,l}}]_{\up^2}}[{{\SE{A} +E_{h+1,l} |\Ahkom-E_{h+1,l}}}])\\	
&\quad\quad-(\up-\up^{-1})\left[{a_{h+1, k} +1}\atop 2\right]\big(
	[{\SE{A} + 2E_{h+1,k}-E_{h+1,l} |\Ahkomm+E_{h+1,l}}]\\
&\hspace{6cm}-{[{a_{h+1,l}}]_{\up^2}}[{\SE{A} + 2E_{h+1,k}+E_{h+1,l} |\Ahkomm-E_{h+1,l}}]\big)
\Big\}.
\end{aligned}
\end{equation*}
	
\end{enumerate}
\end{lem}

\begin{proof} With the SDP condition, the proof of (1) is similar to that of Lemma \ref{phiupper-even-norm}(2) and is omitted.

We now prove (2). First, assume that $A$ satisfies the SDP condition at $(h,k)$, for all $k\in[1,n]$ with $a_{h,k}>0$. By \cite[Th.~6.4]{DGLW}(1) together with Proposition \ref{prop_PhiAPhiB} and noting $\wp({ {\Fd_{\bar h,\lambda}}})=\bar1$, we have 
 \begin{equation}\label{philower1-pf-2}
 \begin{aligned}
 {{ [{\Fd_{\bar h,\lambda}}]} [\Ad]}=&\up^{-\Norm(\Fd_{\bar h})-\Norm(\Ad)}\Phi_{\Fd_{\bar h}}\Phi_\Ad\\
 =&(-1)^{\wp(\Ad)}\up^{-\Norm(\Fd_{\bar h})-\Norm(\Ad)}\sum_{{1\leq k\leq n}\atop {a_{h,k}>0}} \Big\{
	{(-1)}^{{\SOE{\widetilde{a}}}_{h,k} } 	\Phi_{(\SE{A} - E_{h,k}| \SO{A}+ E_{h+1, k})} \\
	& \qquad +
	{(-1)}^{{\SOE{\widetilde{a}}}_{h,k}+1}
	{q}^{-1}  \STEPPDR{a_{h+1, k} +1}
	\Phi_{(\SE{A} - E_{h,k} + 2E_{h+1, k} | \SO{A}  -E_{h+1, k})} \\
	& \qquad +
	{(-1)}^{{\SOE{\widetilde{a}}}_{h,k}} {q}^{a_{h,k}-1}
	\STEP{ \SEE{a}_{h+1,k}+1}
	\Phi_{(\SE{A} + E_{h+1,k} | \SO{A}  - E_{h,k} )}
\Big\}.
 \end{aligned}
 \end{equation}
Let $\lambda=\ro(A)$. Then, by \eqref{norm-fl}, $\Norm({ {\Fd_{\bar h,\lambda}}})=\lambda_{h+1}$. Observe that every $\Phi_\Md$ occurring
on the right hand side has the same base matrix $A^-_{h,k}$. By \eqref{hatell} and Lemma \ref{hatell1}, $\partial(\Md)$ is given as follows:
\begin{equation}\label{philower1-pf-1}
\begin{aligned}
\Norm(\SE{A} - E_{h,k}| \SO{A}+ E_{h+1, k})&=\Norm(\Ad)-\AK(h,k)+\BK(h+1,k)-a^{\bar 1}_{h,k}+a_{h+1,k},\\
\Norm(\SE{A} - E_{h,k} + 2E_{h+1, k} | \SO{A}  -E_{h+1, k})&=\Norm(\Ad)-\AK(h,k)+\BK(h+1,k)-a^{\bar0}_{h+1,k}-a^{\bar1}_{h,k},\\
\Norm(\SE{A} + E_{h+1,k} | \SO{A}  - E_{h,k} )&=\Norm(\Ad)-\AK(h,k)+\BK(h+1,k)-a^{\bar0}_{h,k}+a^{\bar1}_{h+1,k}.
\end{aligned}
\end{equation}
Thus, we may determine the powers of $\up$ in the coefficients of $[\Md]$ as follows:
\begin{equation}\label{philower1-pf-3}
\begin{aligned}
&\Norm(\SE{A} - E_{h,k}| \SO{A}+ E_{h+1, k})-\Norm(\Fd_{\bar h,\lambda})-\Norm(\Ad)=-\AK(h,k)-\AK(h+1,k)-a^{\bar1}_{h,k}=\OG_h(A,k)-a^{\bar1}_{h,k},\\
&\Norm(\SE{A} - E_{h,k} + 2E_{h+1, k} | \SO{A}  -E_{h+1, k})-2+2a_{h+1,k}+1-\Norm({ {\Fd_{\bar h,\lambda}}})-\Norm(\Ad)=\OG_h(A,k)-a^{\bar1}_{h,k},\\
&\Norm(\SE{A} + E_{h+1,k} | \SO{A}  - E_{h,k} )+2(a_{h,k}-1)+\SEE{a}_{h+1,k}-\Norm({ {\Fd_{\bar h,\lambda}}})-\Norm(\Ad)=\OG_h(A,k)+a^{\bar0}_{h,k}.
\end{aligned}
\end{equation}
Here, in the third case, $\SEE{a}_{h+1,k}$ comes from  $\STEP{ \SEE{a}_{h+1,k}+1}=\up^{\SEE{a}_{h+1,k}}[\SEE{a}_{h+1,k}+1]$.
By \eqref{QN2}, we have $\STEPPDR{a_{h+1, k} +1}=(\up-\up^{-1})\up^{2a_{h+1, k} +1}\left[a_{h+1, k} +1\atop 2\right].$
Substituting gives the first formula in (2).

 For the general case, the term $\text{\rm HH}\overline{\textsc f}$ in \cite[Th.~6.4]{DGLW}(2) can be similarly standardized to obtain
 HH$\ol{\textsc{f}}(h,\Ad)$ in (2), noting that the base matrices of the labelling matrices $\Md$ in $\text{\rm HH}\overline{\textsc f}$ are all the same as  $A^-_{h,k}$.
\end{proof}

We now standardise some special cases discussed in \cite[\S7]{DGLW}, noting the sign change for a product of two odd elements in each case. Recall the convention that we regard a term $[\Md]$ as $0$ if $\Md\not\in \MNZ(n,r)$.
\begin{lem}[cf. {\cite[Prop.~7.1]{DGLW}}]\label{phiupper2}
	For any given {$\mu \in \CMN(n,r-1)$}, {$h \in [1,n-1]$}, the following formulas hold in {$\qSchvsZ$}:
	\begin{align*}
{\rm {(1)}}& \quad
[\mu | E_{h, h+1}]\cdot
[{\mu}+  E_{h+1, h}|O]
= 			[{\mu}-E_{h+1, h+1}   + E_{h+1, h}| E_{h, h+1}]
	+	 \up^{\mu_{h+1}} [ {\mu}   | E_{h,h}] \\
	& \qquad\qquad\qquad\qquad\qquad\qquad
	- 	({\up}-\up^{-1}) [{{\mu}_{h}+1}]	[{\mu}-E_{h+1, h+1}   + E_{h, h}  | E_{h+1,h+1}], \\
{\rm {(2)}}& \quad
[{\mu}| E_{h, h+1}]\cdot
[{\mu} | E_{h+1, h}]
=
 {\up}^{ {\mu}_{h+1} }[{\mu}_{h} + 1]
		[{\mu}  +  E_{h,h}|{O}] + [{\mu}-E_{h+1, h+1}   | E_{h, h+1} + E_{h+1, h}]\\
	& \qquad\qquad\qquad\qquad\qquad\qquad
	-	({\up}-\up^{-1}) [{\mu}-E_{h+1, h+1}    | E_{h,h} + E_{h+1,h+1}] \\
{\rm {(3)}}& \quad
[{\mu}| E_{h+1, h}]\cdot
[{\mu}+  E_{h, h+1}|O]
= [{\mu} - E_{h,h} + E_{h, h+1}| E_{h+1, h}]
+\up^{-\mu_h}[{\mu} | E_{h+1, h+1}], \\
{\rm (4)}& \quad
 [{\mu}| E_{h+1, h}]\cdot
[{\mu} | E_{h, h+1}]
= -[{\mu} - E_{h,h} |   E_{h, h+1} +  E_{h+1, h}]
	+ \up^{-\mu_h}[{\mu}_{h+1} + 1 ][{\mu} + E_{h+1, h+1} | O].
\end{align*}
\end{lem}

Using the standardisation of $\phi_{A^\star}$, the relations in \cite[Lem.7.2]{DGLW} can be modified as follows.
\begin{lem}\label{relation_eei}Let $h\in[1,n-1]$ and
 {${\lambda}  \in \CMN(n, r - 1)$}. Then the following relations hold in {$\qSchvsZ$}:
\begin{align}\notag
 (1)&\quad[ {\lambda}  | E_{h, h+1} ]\cdot
		[ {\lambda}  |  E_{h+1, h+2} ]+{\up}  [{\lambda}   + \alpha_h | E_{h+1, h+2}]\cdot
		[ {\lambda} -\alpha_{h+1} | E_{h, h+1} ]=\\
		\notag
&\hspace{1.5cm}[  {\lambda} + E_{h, h+1} | O]\cdot
		[{\lambda}  +  E_{h+1, h+2} | O]	-  {\up}  [ {\lambda}  + \alpha_{h} + E_{h+1, h+2} |O  ]\cdot
		[{\lambda} -\alpha_{h+1} + E_{h, h+1}  | O] .\\
(2)&\quad[\lambda|E_{h,h+1}]\cdot[\lambda+E_{h+1,h+2}|O]
-\up[\lambda+\alpha_h+E_{h+1,h+2}|O]\cdot[\lambda-\alpha_{h+1}|E_{h,h+1}]=\notag\\
&\hspace{2.5cm}[  {\lambda} + E_{h, h+1} | O]\cdot
		[{\lambda}  |  E_{h+1, h+2}] -  {\up}  [ {\lambda}  + \alpha_{h} | E_{h+1, h+2}]\cdot
		[{\lambda} -\alpha_{h+1} + E_{h, h+1}  | O].\notag
\end{align}
\end{lem}
\medskip

We will prove in Corollary \ref{generators}, using a triangular relation, that the multiplication formulas in Lemmas \ref{phiupper-even-norm}--\ref{standard-phiupper1} determine the structure of queer $q$-Schur superalgebras. More precisely, the structure constants from these formulas provide the matrix representation of the regular module $_{\bs {\mathcal A}}\bs {\mathcal A}$ for  the $\QQ(\up)$-superalgebra $\bs {\mathcal A}=\qSchvsQ$.

\section{Long multiplication formulas in $\qSchvsQ$}\label{sec_spanningsets}
So far, we used certain elements and relations in the Hecke-Clifford superalgebras to determine the structure of their associated queer $q$-Schur superalgebras. In other words, we have completed Step 5 in the roadmap set in the introduction. We now show that these structure can be lifted to the quantum queer supergroup $\Uvqn$ via a formal limit process. So, we work with base changed queer $\up$-Schur superalgebras $\qSchvsQ=\qSchvsZ\otimes_{\mathcal{Z}}\QQ(\up)$ from now on.

We first observe that, for any $\Ad=(A^{\bar 0}|A^{\bar 1})=(a^{\bar 0}_{i,j}|a^{\bar 1}_{i,j})\in M_n(\NN|\Ng)$, there is an associated square matrix $A^{\star\!\!\!\!\square}={A^{\bar 0}\;A^{\bar 1}\choose A^{\bar 1}\;A^{\bar 0}}$. By mimicking the construction in \cite{BLM}, we introduce
the subset $\MNZNS(n) $ of $M_n(\NN|\Ng)$ consisting of $\Ad=(A^{\bar 0}|A^{\bar 1})$ such that the diagonal of $A^{\bar 0}$ is 0:
\begin{equation}\label{Mnpm}
\MNZNS(n) :=\{\Ad\in M_n(\NN|\Ng)\mid A^{\star\!\!\!\!\square}\text{ has a 0 diagonal}\}.
\end{equation}

For any   {$\Ad=(\SUP{A}) \in \MNZNS(n)$}, {$\bs{j} \in {\ZZ}^{n} $},
 define the following ``long elements'' in {$\qSchvsQ$}:
\begin{equation}\label{def_ajr}
\begin{aligned}
	\AJRS(\Ad, \bs{j}, r) =
\left\{
\begin{aligned}
	& \sum_{\substack{\lambda \in \CMN(n, r-\snorm{A})} }
	 {v}^{\lambda\centerdot \bs{j}} [ \SE{A} + \lambda | \SO{A} ],
		\ &\mbox{if } \snorm{A} \le r; \\
	&0, &\mbox{otherwise}.
\end{aligned}
\right.
\end{aligned}
\end{equation}
where $\lambda \centerdot \bs{j}=\sum_{i=1}^n\lambda_ij_i$ is the dot product.

The first set of the long (element) multiplication formulas are straightforward.

\begin{prop}\label{mulformzerocor}
Let 
{$h \in [1,n]$}, $\bs{j},\bs{k}\in {\ZZ}^n$, and {$O^\star:=(O|O),\Ad \in  \MNZNS(n)$}. Then,  for all  {$r\geq\snorm{A}$},
the following hold in {$\qSchvsQ$}:
\begin{itemize}
    \item[{\rm(1)}]$ \AJRS({O^\star}, \bs{k}, r) \cdot \AJRS(\Ad, \bs{j}, r)
		 = {\up}^{\ro(A)\centerdot\bs{k}}   \AJRS(\Ad, \bs{k} + \bs{j}, r);$
   \item[ {\rm(2)}] $\AJRS(\Ad, \bs{j}, r)  \cdot  \AJRS({O^\star}, \bs{k}, r)
	={\up}^{\co(A)\centerdot\bs{k} }   \AJRS(\Ad, \bs{k} + \bs{j}, r).$
\end{itemize}
In particular, we have
\begin{enumerate}
\item[{\rm{(a)}}] $\AJRS({O^\star}, \bs{0}, r) \cdot  \AJRS(\Ad, \bs{j}, r)
		=  \AJRS(\Ad, \bs{j}, r) \cdot \AJRS({O^\star}, \bs{0}, r)
		=  \AJRS(\Ad, \bs{j}, r); $
\item[{\rm{(b)}}] $ \AJRS({O^\star}, \pm \ep_h, r) \cdot \AJRS(\Ad, \bs{j}, r)
	= {v}^{ \pm  \sum_{u=1}^n a_{h,u} }   \AJRS(\Ad, \bs{j} \pm\ep_h, r); $
\item[{\rm{(c)}}] $ \AJRS(\Ad, \bs{j}, r) \cdot  \AJRS({O^\star}, \pm \ep_h, r)
	= {v}^{ \pm  \sum_{u=1}^n a_{u, h} }   \AJRS(\Ad, \bs{j} \pm\ep_h, r).$
\end{enumerate}
\end{prop}

 We need the following short form of  the notation in \eqref{f_h/g_h}:
\begin{equation}\label{f_hk}
\aligned
f_{h,k}^{\ol{0}}&=f_h(A,k) +\SOE{a}_{h+1,k},\quad
f_{h,k}^{\ol{1}}=\begin{cases}f_h(A,k)  - \SEE{a}_{h+1,k},&\text{ if }k\neq h+1;\\ f_h(A,h+1),&\text{ if }k=h+1,
\end{cases}\\
g_{h,k}^{\bar 0}&=g_h(A, k)-a^{\bar 1}_{h,k},\qquad
g_{h,k}^{\ol{1}}=\begin{cases}g_h(A,k) + \SEE{a}_{h,k},&\text{ if }k\neq h;\\ g_h(A,h),&\text{ if }k=h.
\end{cases}
\endaligned
\end{equation}

\begin{prop}\label{mulformeven-2}
	For {$h \in [1,n-1]$} and {$\Ad=(\SE{A}|\SO{A})=(a^{\bar 0}_{i,j}|a^{\bar 1}_{i,j}) \in \MNZNS(n)$} with base $A=(a_{i,j})$,
let $\Ahkomm:=\SO{A}  -E_{h,k} - E_{h+1,k}$ and $\boxed{a}_{h,k}:=(\up-\up^{-1}) \left[a_{h,k}+1\atop 2\right]$. Then,  for all  {$r\geq\snorm{A}$}, the following multiplication formulas hold in {$\qSchvsQ$}:
\begin{equation*}\label{EA-up}
\begin{aligned}
(1)&\quad \ABJRS( E_{h, h+1}, {O}, \bs{ 0 },  r ) \cdot\Ad( \bs{j},  r )\\
&=\sum_{k<h}
	{\up}^{f_{h,k}^{\ol{0}}}   [ \SEE{a}_{h,k} + 1] (\Ahkzp | \SO{A})(\bs{j}+\alpha_h,r)
+\sum_{h+1<k}
	{\up}^{f_{h,k}^{\ol{0}}}   [ \SEE{a}_{h,k} + 1] (\Ahkzp | \SO{A})(\bs{j},r)	\\
&\hspace{1.5cm}+ \frac{{\up}^{f_{h,h}^{\ol{0}}-j_h}}{\up-\up^{-1}}\left\{(\SE{A} - E_{h+1, h} | \SO{A}) (\bs{j}+\alpha_h,r)
-(\SE{A} - E_{h+1, h} | \SO{A}) (\bs{j}-\alpha^+_h,r)\right\}\\
&\hspace{1.5cm}+{\up}^{f_{h,h+1}^{\ol{0}}+j_{h+1}}   [ \SEE{a}_{h,h+1} + 1]  (\SE{A} + E_{h,h+1} | \SO{A})(\bs{j},r)\\
&+\sum_{k<h}{\up}^{f_{h,k}^{\ol{1}}}(\SE{A} |\Ahkop)(\bs{j}+\alpha_h,r)+\sum_{h+1<k}{\up}^{f_{h,k}^{\ol{1}}}(\SE{A}  |\Ahkop)(\bs{j},r)\\
&\hspace{1.5cm}+{\up}^{f_{h,h}^{\ol{1}}}(\SE{A} | \Ahhop)(\bs{j}-\ep_{h+1},r)+{\up}^{f_{h,h+1}^{\ol{1}}}(\SE{A}  | \Ahhiop)(\bs{j}-\ep_{h+1},r)\\
&-\sum_{k<h}{\up}^{f_{h,k}^{\ol{1}}}\boxed{a}_{h,k}(\SE{A}+ 2E_{h,k} | \Ahkomm)(\bs{j}+\alpha_h,r)-\sum_{h+1<k}{\up}^{f_{h,k}^{\ol{1}}}\boxed{a}_{h,k}(\SE{A}+ 2E_{h,k} |\Ahkomm)(\bs{j},r)\\
&\hspace{1.5cm}-\frac{{\up}^{f_{h,h}^{\ol{1}}-2j_h}}{\up-\up^{-1}}\bigg\{\frac{\up^{-1}}{[2]}(\SE{A}|\Ahhomm)(\bs{j}+2\ep_h-\ep_{h+1},r)+\frac{\up}{[2]}(\SE{A}|\Ahhomm)(\bs{j}-2\ep_h-\ep_{h+1},r)\\
&\hspace{1.5cm}-(\SE{A}|\Ahhomm)(\bs{j}-\ep_{h+1},r)\bigg\}-{\up}^{f_{h,h+1}^{\ol{1}} }\boxed{a}_{h,h+1}(\SE{A}+ 2E_{h,h+1} | \Ahhiomm)(\bs{j}-\ep_{h+1},r),
\end{aligned}
\end{equation*}
\begin{equation*}\label{EA-Low}
\begin{aligned}
(2)&\quad \ABJRS( E_{h+1, h}, O, \bs{ 0 },  r ) \cdot \Ad(\bs{j},  r ) \\
&=\sum_{k < h}\up^{g_{h,k}^{\ol{0}}}
	[ \SEE{a}_{h+1, k} +1]
	\ABJRS(\Ahkzm, \SO{A}, \bs{j},  r )  +
	\sum_{h+1<k}
	{\up}^{g_{h,k}^{\ol{0}}}
	[\SEE{a}_{h+1, k} +1]
	\ABJRS(\Ahkzm, \SO{A}, \bs{j}-\alpha_h, r )\\
	&\hspace{1.5cm}+
	{\up}^{g_{h,h}^{\ol{0}}+j_h}
	[\SEE{a}_{h+1, h} +1]
	\ABJRS( \SE{A} + E_{h+1, h}, \SO{A}, \bs{j},  r ) \\
	&\hspace{1.5cm} +
	\frac{{\up}^{g_{h,h+1}^{\ol{0}}-j_{h+1}}} {\up-\up^{-1} }\bigg\{
	\ABJRS( \SE{A} - E_{h,h+1}, \SO{A}, \bs{j}-\alpha_h,  r )
	-\ABJRS( \SE{A} - E_{h,h+1}, \SO{A}, \bs{j}-\alpha^+_h,  r )\bigg  \}\\
	&+\sum_{k<h}
	{\up}^{g_{h,k}^{\ol{1}}}	\ABJRS(\SE{A}, \Ahkom, \bs{j}, r)+
	\hspace{-0.2cm}\sum_{h+1<k}
	{\up}^{g_{h,k}^{\ol{1}}}
	\ABJRS(\SE{A}, \Ahkom, {\bs{j}} -\alpha_h,  r )\\
	&\hspace{1.5cm}+
	{\up}^{g_{h,h}^{\ol{1}}}
	\ABJRS(\SE{A}, \Ahhom,  {\bs{j}} +\ep_h, r)  +
	{{\up}^{g_{h,h+1}^{\ol{1}}} }
	\ABJRS(\SE{A}, \Ahhiom, \bs{j}-\ep_h, r)  \\
-&\!\sum_{k < h }
	{\up}^{g_{h,k}^{\ol{1}}}\boxed{a}_{h+1,k}
	\ABJRS( \SE{A} + 2E_{h+1,k},\Ahkomm, \bs{j}, r )-
	\!\sum_{h+1<k}
	{\up}^{g_{h,k}^{\ol{1}}}\boxed{a}_{h+1,k} \ABJRS( \SE{A} + 2E_{h+1,k}, \Ahkomm, \bs{j} -\alpha_h, r ) \\
	& \hspace{-0.4cm} -
	{\up}^{g_{h,h}^{\ol{1}}}\boxed{a}_{h+1,h}
	\ABJRS( \SE{A}  + 2E_{h+1,h}, \Ahhomm, \bs{j} + \ep_{h},  r  )  -
	\frac{{\up}^{g_{h,h+1}^{\ol{1}}-2j_{h+1}}}{\up-\up^{-1}}\bigg\{\frac{\up^{-1}}{[2]}
	\ABJRS( \SE{A}, \Ahhiomm, \bs{j} -\ep_{h}+ 2 \ep_{h+1},  r ) \\
	&\hspace{1.5cm}+\frac{\up}{[2]}
	\ABJRS( \SE{A},\Ahhiomm, \bs{j}-\ep_{h}-2 \ep_{h+1},  r ) -\ABJRS( \SE{A}, \Ahhiomm, \bs{j}-\ep_{h},  r )\bigg\}
	 .
	\end{aligned}
\end{equation*}

\end{prop}

\begin{proof}
We only prove (1).  The proof of
{\rm(2)} is similar.
By Definition \eqref{def_ajr} and Lemma \ref{phiupper-even-norm}(2), we may expand the product on the LHS as a sum of three parts
each of which is a sum on the $h$-th row of $A^{\bar0}$ or $A^{\bar1}$:
\begin{align*}
& \ABJRS( E_{h, h+1}, O, \bs{ 0 },  r ) \cdot \ABJRS( \SE{A}, \SO{A}, \bs{j},  r ) \\
& =
	\sum_{\substack{
		\lambda \in \CMN(n,r-\snorm{A}) \\
	} }
	{v}^{\lambda \centerdot {\bs{j}}}
	[ \ro(A+\lambda)  - E_{h+1, h+1} + E_{h, h+1} |O ] [\SE{A} + \lambda | \SO{A}]= {\fcY}_1 + {\fcY}_2 - {\fcY}_3,
\end{align*}
where, denoting the $(h,k)$ entry of a matrix $X$ by $X_{h,k}$,
\begin{align*}
{\fcY}_1 &=
	\sum_{\substack{
		\lambda \in \CMN(n,r-\snorm{A}) \\
	} }
	{v}^{\lambda \centerdot {\bs{j}}}
	\sum_{k=1}^n
	{\up}^{f_h(A+\lambda,k)  + \SOE{a}_{h+1,k}}  [ {(\SEE{A}+\lambda)}_{h,k} + 1]
		[\SE{A} + \lambda  - E_{h+1, k} + E_{h,k} | \SO{A}] , \\
{\fcY}_2 &=
	\sum_{\substack{
		\lambda \in \CMN(n,r-\snorm{A}) \\
	} }
	{v}^{\lambda \centerdot {\bs{j}}}
	\sum_{k=1}^n
	{\up}^{f_h(A+\lambda,k)  - (\SE{A} + \lambda )_{h+1,k}}  [\SE{A} + \lambda   | \SO{A} - E_{h+1, k}  + E_{h,k}] , \\
{\fcY}_3 &=\!\!
	\sum_{\substack{
		\lambda \in \CMN(n,r-\snorm{A}) \\
	} }
	{v}^{\lambda \centerdot {\bs{j}}}
	\sum_{k=1}^n
	{\up}^{f_h(A+\lambda,k)  -(\SE{A} + \lambda )_{h+1,k}} (\up-\up^{-1})\left[{(A+\lambda)_{h,k}}+1\atop 2\right][\SE{A} + \lambda  + 2E_{h,k} | \Ahkomm].
\end{align*}
Clearly, $(\SEE{A}+\lambda)_{i,j}=\SEE{a}_{i,j}$ for $i\neq j$ and $(\SEE{A}+\lambda)_{i,i}=\lambda_{i}$.
 By  \eqref{f_h/g_h} and \eqref{partialrow},
\begin{equation}\label{f_h(A,k)}
f_h(A+\lambda,k)
=
\left\{
\begin{aligned}
& {\BK(h,k)}  -\BK(h+1,k)+\lambda_h-\lambda_{h+1} , & \mbox{ if } k< h, \\
& {\BK(h,k)}  -\BK(h+1,k)-\lambda_{h+1} , & \mbox{ if } k = h,\\
&{\BK(h,k)}  -\BK(h+1,k),&\mbox{ if } k > h.
\end{aligned}
\right.
\end{equation}
 Using the notations in \eqref{f_hk} and \eqref{A01hk},
we have after swapping two $\Sigma$ sums
\begin{equation}\label{Comput-Y1}
\begin{aligned}
{\fcY}_1
=&\sum_{k<h}
	{\up}^{f_{h,k}^{\ol{0}}}   [ \SEE{a}_{h,k} + 1]  \sum_{\substack{
		\lambda \in \CMN(n,r-\snorm{A}) \\
	} }
	{v}^{\lambda \centerdot {\bs{j}}+\lambda_h-\lambda_{h+1}}
		[\Ahkzp + \lambda  | \SO{A}] \\
&+\sum_{\substack{
		\lambda \in \CMN(n,r-\snorm{A}) \\
	} }
	{\up}^{f_{h,h}^{\ol{0}}}   [ \lambda_h + 1]
	{v}^{\lambda \centerdot {\bs{j}}-\lambda_{h+1}}
		[A^{\bar0,+}_{h,h} + \lambda  | \SO{A}] \\&
+{\up}^{f_{h,h+1}^{\ol{0}}}   [ \SEE{a}_{h,h+1} + 1]  \sum_{\substack{
		\lambda \in \CMN(n,r-\snorm{A}) \\
	} }
	{v}^{\lambda \centerdot {\bs{j}}}
		[A^{\bar0,+}_{h,h+1} + \lambda| \SO{A}] \\
&+\sum_{k>h+1}
	{\up}^{f_{h,k}^{\ol{0}}}   [ \SEE{a}_{h,k} + 1]  \sum_{\substack{
		\lambda \in \CMN(n,r-\snorm{A}) \\
	} }
	{v}^{\lambda \centerdot {\bs{j}}}
		[\Ahkzp + \lambda  | \SO{A}]
\end{aligned}
\end{equation}
The first and last summations have summands $(\Ahkzp | \SO{A})(\bs{j}+\bsal_h,r)$ and $(\Ahkzp | \SO{A})(\bs{j},r)$, respectively. The middle two summations correspond to the $k=h$ and $k=h+1$ cases.

For $k=h$, a direct calculation shows
$$\begin{aligned}
&\sum_{\substack{
		\lambda \in \CMN(n,r-\snorm{A}) \\
	} }
	{\up}^{f_{h,h}^{\ol{0}}}   [ \lambda_h + 1]
	{v}^{\lambda \centerdot {\bs{j}}-\lambda_{h+1}}
		[\SE{A} + \lambda  - E_{h+1, h} + E_{h,h} | \SO{A}]\\
=&\sum_{\substack{
		\mu \in \CMN(n,r-\snorm{A}+1) \\
	} }
	{\up}^{f_{h,h}^{\ol{0}}}   [ \mu_h]
	{v}^{(\mu-\ep_h) \centerdot {\bs{j}}-\mu_{h+1}}
		[\SE{A} + \mu - E_{h+1, h}| \SO{A}]\;\;(\mbox{set } \;\mu=\lambda+\ep_h)\\
=&{\up}^{f_{h,h}^{\ol{0}}-\bs{j}_h}   \sum_{\substack{
		\mu \in \CMN(n,r-\snorm{A}+1) \\
	} }
\frac{\up^{\mu \centerdot {(\bs{j}+\ep_h-\ep_{h+1})}}-\up^{\mu \centerdot {(\bs{j}-\ep_h-\ep_{h+1})}}}{\up-\up^{-1}}
		[\SE{A} + \mu - E_{h+1, h}| \SO{A}]\\
=&\frac{{\up}^{f_{h,h}^{\ol{0}}-\bs{j}_h}}{\up-\up^{-1}} \left\{(\SE{A} - E_{h+1, h}| \SO{A})(\bs{j}+\bsal_h,r)-
(\SE{A} - E_{h+1, h}| \SO{A})(\bs{j}-\bsal^+_h,r)\right\},
\end{aligned}
$$
while, for $k=h+1$,
$$
\begin{aligned}
&{\up}^{f_{h,h+1}^{\ol{0}}}   [ \SEE{a}_{h,h+1} + 1]  \sum_{\substack{
		\lambda \in \CMN(n,r-\snorm{A}) \\
	} }
	{v}^{\lambda \centerdot {\bs{j}}}
		[\SE{A} + \lambda  - E_{h+1, h+1} + E_{h,h+1} | \SO{A}]\\
=&{\up}^{f_{h,h+1}^{\ol{0}}}   [ \SEE{a}_{h,h+1} + 1]  \sum_{\substack{
		\mu \in \CMN(n,r-\snorm{A}-1) \\
	} }
	{v}^{(\mu+\ep_{h+1}) \centerdot {\bs{j}}}
		[\SE{A} +\mu+ E_{h,h+1} | \SO{A}](\mbox{set }  \mu=\lambda-\ep_{h+1})\\
=&{\up}^{f_{h,h+1}^{\ol{0}}+\bs{j}_{h+1}}[\SEE{a}_{h,h+1} + 1](\SE{A}+ E_{h,h+1}| \SO{A})(\bs{j},r).
\end{aligned}
$$
Altogether shows that $\fcY_1$ is the first part of the multiplication formula in (1).

We now compute $\fcY_2$.
Observe $(\SEE{A}+\lambda)_{h+1,k}=\SEE{a}_{h+1,k}$ for $k\neq h+1$ and $(\SEE{A}+\lambda)_{h+1,h+1}=\lambda_{h+1}$.
Similarly, by the notation in \eqref{A01hk},  we have after swapping two $\Sigma$ sums
\begin{align*}
{\fcY}_2
=&\sum_{k<h}{\up}^{f_h(A,k)  - \SEE{a}_{h+1,k}}\sum_{\substack{
		\lambda \in \CMN(n,r-\snorm{A}) \\
	} }
	{v}^{\lambda \centerdot {\bs{j}}+\lambda_h-\lambda_{h+1}} [\SE{A} + \la   | \Ahkop]\\
&+{\up}^{f_h(A,h)  - \SEE{a}_{h+1,h}}\sum_{\substack{
		\lambda \in \CMN(n,r-\snorm{A}) \\
	} }
	{v}^{\lambda \centerdot {\bs{j}}-\lambda_{h+1}} [\SE{A} + \lambda   | A^{\bar1,+}_{h,h}]\\
&+{\up}^{f_h(A,h+1)}\sum_{\substack{
		\lambda \in \CMN(n,r-\snorm{A}) \\
	} }
	{v}^{\lambda \centerdot {\bs{j}} -\lambda_{h+1}} [\SE{A} + \lambda   | A^{\bar1,+}_{h,h+1}]\\
&+\sum_{k>h+1}{\up}^{f_h(A,k)  - \SEE{a}_{h+1,k}}\sum_{\substack{
		\lambda \in \CMN(n,r-\snorm{A}) \\
	} }
	{v}^{\lambda \centerdot {\bs{j}}} [\SE{A} + \lambda   | \Ahkop],
\end{align*}
which implies the middle part of the formula in (1) after applying \eqref{f_h(A,k)} and definition \eqref{def_ajr}.

Finally, we compute ${\fcY}_3=\sum_{k=1}^n\fcY_{3,k}$, where
\begin{align}\label{Y3}
{\fcY}_{3,k}= \sum_{\substack{
		\lambda \in \CMN(n,r-\snorm{A}) \\
	} }
	{\up}^{\lambda \centerdot {\bs{j}}}
	&{\up}^{f_h(A+\lambda,k)  - (\SEE{A}+\lambda)_{h+1,k}}(\up-\up^{-1}) \left[{({A}+\lambda)_{h,k}}+1\atop 2\right]
	[\SE{A} + \lambda  + 2E_{h,k} | \Ahkomm].\end{align}

Since $(\SE{A} + \lambda  + 2E_{h,k} | \SO{A}  -E_{h,k} - E_{h+1,k})\in \MNZ(n,r)$ implies  $\SOE{a}_{h,k}=\SOE{a}_{h+1,k}=1$ and  $(A+\lambda)_{i,j}=a_{i,j}$ for $i\neq j$ and $(A+\lambda)_{i,i}=\lambda_i+1$, the computation of $\fcY_{3,k}$ depends on $k$.

If $k<h$, applying \eqref{f_h(A,k)}, definition \eqref{def_ajr} and $\boxed{a}_{h,k}=(\up-\up^{-1}) \left[a_{h,k}+1\atop 2\right]$, \eqref{Y3} becomes
\begin{align*}
{\fcY}_{3,k}&={\up}^{f_{h,k}^{\ol{1}}}\boxed{a}_{h,k} \sum_{\substack{
		\lambda \in \CMN(n,r-\snorm{A}) \\
	} }
	{v}^{\lambda \centerdot {\bs{j}}+\lambda_h-\lambda_{h+1}}[\SE{A} + \lambda  + 2E_{h,k} | \Ahkomm]\\
&={\up}^{f_{h,k}^{\ol{1}}}\boxed{a}_{h,k}(\SE{A}+ 2E_{h,k} | \Ahkomm)(\bs{j}+\ep_h-\ep_{h+1}).
\end{align*}

If $k=h$,  then $f_h(A+\lambda,h)=f_h(A,h)-\lambda_{h+1}$.
Let $\mu=\lambda+2\ep_h$. Then $\lambda_h+2=\mu_h $ and
$$\left[{({A}+\lambda)_{h,h}}+1\atop 2\right]=\left[\lambda_h+2\atop 2\right]=\left[\mu_h\atop 2\right]=\frac{1}{[2](\up-\up^{-1})^2}(\up^{2\mu_h-1}+\up^{-2\mu_h+1})-\frac{1}{(\up-\up^{-1})^2},$$
which is 0 for those $\mu$ with $\mu_h=0$ or 1. Hence, by  \eqref{def_ajr},
\begin{align*}
{\fcY}_{3,h}&={\up}^{f_{h,h}^{\ol{1}}}\hspace{-0.5cm}\sum_{\substack{
		\mu \in \CMN(n,r-\snorm{A}+2) \\
	} }\hspace{-0.5cm}{v}^{(\mu-2\ep_h) \centerdot {\bs{j}}-\mu_{h+1}}\Big(\frac{\up^{2\mu_h-1}+\up^{-2\mu_h+1}}{[2](\up-\up^{-1})}-\frac{1}{(\up-\up^{-1})}\Big)
[\SE{A} + \mu|A^{\bar1,\approx}_{h,h}]\\
&=\frac{{\up}^{f_{h,h}^{\ol{1}}-2j_h}}{\up-\up^{-1}}\Big\{ \frac{\up^{-1}}{[2]}(\SE{A}| A^{\bar1,\approx}_{h,h})(\bs{j}+2\ep_h-\ep_{h+1},r)+\frac{\up}{[2]}(\SE{A}| A^{\bar1,\approx}_{h,h})(\bs{j}-2\ep_h-\ep_{h+1},r)\\
&\qquad\qquad\qquad-(\SE{A}| A^{\bar1,\approx}_{h,h})(\bs{j}-\ep_{h+1},r)\Big\}.
\end{align*}

If  $k=h+1$,  then $f_h(A+\lambda,h+1)=f_h(A,h+1)=f^{\bar1}_{h,h+1}$. Thus,
\begin{align*}
{\fcY}_{3,h+1}&={\up}^{f_{h,h+1}^{\ol{1}}}\boxed{a}_{h,h+1}\sum_{\substack{
		\lambda \in \CMN(n,r-\snorm{A}) \\
	} }{v}^{\lambda \centerdot {\bs{j}} -\lambda_{h+1}}[\SE{A} + \lambda  + 2E_{h,h+1} | A^{\bar1,\approx}_{h,h+1}]\\
&={\up}^{f_{h,h+1}^{\ol{1}}}\boxed{a}_{h,h+1}(\SE{A}+ 2E_{h,h+1} | A^{\bar1,\approx}_{h,h+1})(\bs{j}-\ep_{h+1},r).
\end{align*}

Finally, if $k>h+1$, then \eqref{f_h(A,k)}, $f_h(A+\lambda,k)=f_h(A,k)$ and $(\SEE{A}+\lambda)_{h+1,k}=a^{\bar0}_{h+1,k}$. Due to \eqref{def_ajr},
\begin{align*}
{\fcY}_{3,k}&={\up}^{f_{h,k}^{\ol{1}} }\boxed{a}_{h,k}\sum_{\substack{
		\lambda \in \CMN(n,r-\snorm{A}) \\
	} }{v}^{\lambda \centerdot {\bs{j}}}[\SE{A} + \lambda  + 2E_{h,k} |\Ahkomm]={\up}^{f_{h,k}^{\ol{1}} }\boxed{a}_{h,k}(\SE{A}+ 2E_{h,k} |\Ahkomm)(\bs{j},r).
\end{align*}
Combining the four cases gives $\fcY_3$, the last part of the formula in (1). This proves (1).
\end{proof}

We now derive long multiplication formulas in the odd case. If we regard Proposition \ref{mulformzerocor} as the even Cartan case, then the odd Cartan case below, building on Lemma \ref{normlized-phidiag1} is a bit more complicated. Moreover, we only compute the formula under the SDP condition. See Remark \ref{longtail} below for the general case.
\begin{prop}\label{mulformdiag}
Let
{$h \in [1,n]$} and {$\Ad=(\SE{A}|\SO{A})=(a^{\bar 0}_{i,j}|a^{\bar 1}_{i,j}) \in \MNZNS(n)$} with base $A=(a_{i,j})$.
Assume, for any  {$\lambda \in \CMN(n, r-\snorm{A})$},
	 {$A+\lambda$}  satisfies the SDP condition on the $h$-th row if $h<n$.
Then,  for all  {$r\geq\snorm{A} $}, the following multiplication formula holds in {$\qSchvsQ$}:
\begin{align*}
&\ABJRS(O, E_{h,h}, \bs{ 0 },  r ) \cdot \Ad(\bs{j},  r )  \\
=&\sum_{k<h}{(-1)}^{ {{\SOE{\widetilde{a}}}_{h,k}}+\parity{\Ad}}{\up}^{d_h(A,k)}(\SE{A} -E_{h,k}|\SO{A}+ E_{h,k})(\bs{j}+\ep_h,r)\\
&\qquad+{(-1)}^{ {{\SOE{\widetilde{a}}}_{h,h}}+\parity{\Ad} }{\up}^{d_h(A,h)+j_h}(\SE{A}|\SO{A}+ E_{h,h})(\bs{j},r)\\
&\qquad\qquad+\sum_{h<k}{(-1)}^{ {{\SOE{\widetilde{a}}}_{h,k}}+\parity{\Ad} }{\up}^{d_h(A,k)}(\SE{A}-E_{h,k}|\SO{A}+ E_{h,k})(\bs{j}-\ep_h,r)\\
&+\sum_{k<h}{(-1)}^{ {{\SOE{\widetilde{a}}}_{h,k}}+\parity{\Ad}}{\up}^{d_h(A,k)}[a_{h,k}]_{\up^2}(\SE{A} + E_{h,k}| \SO{A} - E_{h,k})(\bs{j}+\ep_h,r)\\
&\qquad+{(-1)}^{ {{\SOE{\widetilde{a}}}_{h,h}}+\parity{\Ad} }\frac{{\up}^{d_h(A,h)-j_h}}{\up^2-\up^{-2}}\{(\SE{A} | \SO{A} - E_{h,h})(\bs{j}+2\ep_h,r)-(\SE{A} | \SO{A} - E_{h,h})(\bs{j}-2\ep_h,r)\}\\
&\qquad\qquad+\sum_{h<k}{(-1)}^{ {{\SOE{\widetilde{a}}}_{h,k}}+\parity{\Ad} }{\up}^{d_h(A,k)}[a_{h,k}]_{\up^2}(\SE{A} + E_{h,k}| \SO{A} - E_{h,k})(\bs{j}-\ep_h,r)
\\&=:\sdpCHk(h,\Ad(\bs{j},r)).
\end{align*}
\end{prop}
\begin{proof}
Observe that for any {$\mu \in \CMN(n, r-1)$},
{${(-1)}^{\parity{\mu|E_{h,h}} \cdot \parity{\Ad}} = {(-1)}^{\parity{\Ad}}$}.
Recalling the notation $d_h(A,k)$ in \eqref{d_hA},  we have
\begin{equation}
d_h(A+\lambda,k)=\left\{
\begin{aligned}
&d_h(A,k)+\lambda_h, &\mbox{ if }k<h,\\
&d_h(A,k), &\mbox{ if }k=h,\\
&d_h(A,k)-\lambda_h, &\mbox{ if }k>h.
\end{aligned}
\right.
\end{equation}

If {$A+\lambda$}  satisfies the SDP condition on the $h$-th row  for any  {$\lambda \in \CMN(n, r-\snorm{A})$}. By  \eqref{def_ajr} and Lemma \ref{normlized-phidiag1}(1), one can obtain
\begin{align*}
	&\ABJRS( O, E_{h,h}, \bs{ 0 },  r ) \cdot \ABJRS( \SE{A}, \SO{A}, \bs{j},  r ) \\
	& = \sum_{\substack{
			\mu \in \CMN(n, r-1) \\
			\lambda \in \CMN(n, r-\snorm{A}) \\
			\co(\mu + E_{h,h}) = \ro(A+\lambda)
			}
		}
		{v}^{{\mu} \centerdot {\bs{0}}}
		{v}^{\lambda \centerdot {\bs{j}}}
		\cdot
		[{\mu}|E_{h,h}] [ \SE{A} + \lambda | \SO{A} ] \\
	& = \sum_{\substack{ \lambda \in \CMN(n,r-\snorm{A}) } }
		{v}^{ \lambda \centerdot {\bs{j}} }
		{(-1)}^{\parity{\Ad}}
		\sum_{ k=1}^n
			{(-1)}^{ {{\SOE{\widetilde{a}}}_{h,k}} }{\up}^{d_h(A+\lambda,k)}\{
			[\SE{A} + \lambda -E_{h,k}|\SO{A}+ E_{h,k}]\\
&\quad\quad+[{(A+\lambda)}_{h,k}]_{\up^2}[\SE{A} + \lambda+  E_{h,k}| \SO{A} - E_{h,k}]\}\\
&= {(-1)}^{\parity{\Ad}}( {\fcY}_1 + {\fcY}_2) ,
\end{align*}
where
\begin{align*}
{\fcY}_1
=&\sum_{\substack{ \lambda \in \CMN(n,r-\snorm{A}) } }
		{v}^{ \lambda \centerdot {\bs{j}} }
		\sum_{ k=1}^n
			{(-1)}^{ {{\SO{\widetilde{a}}}_{h,k}} }{\up}^{d_h(A+\lambda,k)}
			[\SE{A} + \lambda -E_{h,k}|\SO{A}+ E_{h,k}]\\
=&\sum_{k<h}{(-1)}^{ {{\SO{\widetilde{a}}}_{h,k}} }{\up}^{d_h(A,k)}\sum_{\substack{ \lambda \in \CMN(n,r-\snorm{A}) } }
		{v}^{ \lambda \centerdot {\bs{j}}+\lambda_h }[\SE{A}+\lambda -E_{h,k}|\SO{A}+ E_{h,k}]\\
&+{(-1)}^{ {{\SOE{\widetilde{a}}}_{h,h}} }{\up}^{d_h(A,h)}\sum_{\substack{ \lambda \in \CMN(n,r-\snorm{A}) } }
		{v}^{ \lambda \centerdot {\bs{j}}}[\SE{A}+\lambda -E_{h,h}|\SO{A}+ E_{h,h}]\\
&+\sum_{k>h}{(-1)}^{ {{\SOE{\widetilde{a}}}_{h,k}} }{\up}^{d_h(A,k)}\sum_{\substack{ \lambda \in \CMN(n,r-\snorm{A}) } }
		{v}^{ \lambda \centerdot {\bs{j}}-\lambda_h }[\SE{A}+\lambda -E_{h,k}|\SO{A}+ E_{h,k}]\\
=&\sum_{k<h}{(-1)}^{ {{\SOE{\widetilde{a}}}_{h,k}} }{\up}^{d_h(A,k)}(\SE{A} -E_{h,k}|\SO{A}+ E_{h,k})(\bs{j}+\ep_h,r)\\
&+{(-1)}^{ {{\SOE{\widetilde{a}}}_{h,h}} }{\up}^{d_h(A,h)+j_h}(\SE{A}|\SO{A}+ E_{h,h})(\bs{j},r)\\
&+\sum_{k>h}{(-1)}^{ {{\SOE{\widetilde{a}}}_{h,k}} }{\up}^{d_h(A,k)}(\SE{A}-E_{h,k}|\SO{A}+ E_{h,k})(\bs{j}-\ep_h,r),
\end{align*}
{  and }
\begin{align*}
{\fcY}_2 & = \sum_{\substack{ \lambda \in \CMN(n,r-\snorm{A}) } }
		{v}^{ \lambda \centerdot {\bs{j}} }
		\sum_{ k=1}^n
			{(-1)}^{ {{\SO{\widetilde{a}}}_{h,k}} }{\up}^{d_h(A+\lambda,k)}
[{(A+\lambda)}_{h,k}]_{\up^2}[\SE{A} + \lambda+  E_{h,k}| \SO{A} - E_{h,k}]\\
&=\sum_{k<h}{(-1)}^{ {{\SO{\widetilde{a}}}_{h,k}} }{\up}^{d_h(A,k)}[a_{h,k}]_{\up^2}\sum_{\substack{ \lambda \in \CMN(n,r-\snorm{A}) } }
		{v}^{ \lambda \centerdot {\bs{j}}+\lambda_h }[\SE{A} + \lambda+  E_{h,k}| \SO{A} - E_{h,k}]\\
&+{(-1)}^{ {{\SO{\widetilde{a}}}_{h,h}} }{\up}^{d_h(A,h)}\sum_{\substack{ \lambda \in \CMN(n,r-\snorm{A}) } }
		{\up}^{ \lambda \centerdot {\bs{j}} }\frac{\up^{2(\lambda_h+1)}-\up^{-2(\lambda_h+1)}}{\up^2-\up^{-2}}[\SE{A}+\lambda + E_{h,h}| \SO{A} - E_{h,h}](\mbox{set $\lambda+\ep_h=\mu$})\\
&+\sum_{k>h}{(-1)}^{ {{\SO{\widetilde{a}}}_{h,k}} }{\up}^{d_h(A,k)}[a_{h,k}]_{\up^2}\sum_{\substack{ \lambda \in \CMN(n,r-\snorm{A}) } }
		{v}^{ \lambda \centerdot {\bs{j}}-\lambda_h }[\SE{A} + \lambda+  E_{h,k}| \SO{A} - E_{h,k}]\\
&=\sum_{k<h}{(-1)}^{ {{\SO{\widetilde{a}}}_{h,k}} }{\up}^{d_h(A,k)}[a_{h,k}]_{\up^2}(\SE{A} + E_{h,k}| \SO{A} - E_{h,k})(\bs{j}+\ep_h,r)\\
&+{(-1)}^{ {{\SO{\widetilde{a}}}_{h,h}} }\frac{{\up}^{d_h(A,h)-j_h}}{\up^2-\up^{-2}}\{(\SE{A} | \SO{A} - E_{h,h})(\bs{j}+2\ep_h,r)-(\SE{A} | \SO{A} - E_{h,h})(\bs{j}-2\ep_h,r)\}\\
&+\sum_{k>h}{(-1)}^{ {{\SOE{\widetilde{a}}}_{h,k}} }{\up}^{d_h(A,k)}[a_{h,k}]_{\up^2}(\SE{A} + E_{h,k}| \SO{A} - E_{h,k})(\bs{j}-\ep_h,r).
\end{align*}
In the above computation for $\fcY_1$ and $\fcY_2$, we observe that if $(\SE{A} + \lambda -E_{h,k}|\SO{A}+ E_{h,k})\in\MNZN(n)$ for any $1\leq k\leq n$, then $\SO{a}_{h,k}=0$  and $\SO{\tilde{a}}_{h-1,k}=\SO{\tilde{a}}_{h,k}$
and if $(\SE{A} + \lambda+  E_{h,k}| \SO{A} - E_{h,k})\in\MNZN(n)$ for any $1\leq k\leq n$, then $\SOE{a}_{h,k}=1$.
Putting {${\fcY}_1$} and {${\fcY}_2$} together,  the proposition is proved.
\end{proof}

Applying Lemma \ref{standard-phiupper1}  
together with the notations in \eqref{A01hk} and \eqref{f_hk}, we obain
the following multiplication formulas in Propositions \ref{mulformodd1} and \ref{mulformodd2} under the SDP condition. Their detailed proofs are omitted.
See Remark \ref{longtail} for the general case.
\begin{prop}\label{mulformodd1}
	Let
	{$h \in [1,n)$} and  {$\Ad=(\SE{A}|\SO{A})=(a^{\bar 0}_{i,j}|a^{\bar 1}_{i,j}) \in \MNZNS(n)$} with base $A=(a_{i,j})$.
Assume, for each $1\leq k\leq n$ and $\lambda \in \CMN(n, r-\snorm{A})$,
{$({A+\lambda})^+_{h,k}$} is in $M_n(\NN)$ and satisfies SDP condition on the  $h$-th row.
Then,  for all  {$r\geq \snorm{A}$}, the following multiplication formula  holds in {$\qSchvsQ$}:
\begin{align*}
(-1)^{\wp(\Ad)}& \ABJRS(O, E_{h, h+1}, \bs{ 0 },  r  ) \cdot \Ad(\bs{j},r)=\\
&\sum_{k<h}
	{(-1)}^{{ {\SOE{\widetilde{a}}}_{h-1,k}}}
	{\up}^{f^{\ol{0}}_{h,k}}
	\ABJRS( \SE{A} - E_{h+1, k}, \SO{A}  + E_{h,k}, \bs{j} +\bsal_h,  r ) \\
	& \hspace{1cm} +
	{(-1)}^{{ {\SOE{\widetilde{a}}}_{h-1,h}}}
	{\up}^{f^{\ol{0}}_{h,h}}
	\ABJRS( \SE{A} - E_{h+1, h}, \SO{A}  + E_{h,h}, \bs{j}- \ep_{h+1}, r )\\
	& \hspace{1cm}+
	{(-1)}^{{ {\SOE{\widetilde{a}}}_{h-1,h+1}}}
	{\up}^{f^{\ol{0}}_{h,h+1}+j_{h+1}}
	\ABJRS( \SE{A}, \SO{A}  + E_{h,h+1}, \bs{j},  r ) \\
	& \hspace{1cm} +
	\sum_{ \substack{ k > h+1}  }
	{(-1)}^{{ {\SOE{\widetilde{a}}}_{h-1,k}}}
	{\up}^{f^{\ol{0}}_{h,k}}
	\ABJRS( \SE{A} - E_{h+1, k}, \SO{A}  + E_{h,k}, \bs{j}, r )\\
+&\sum_{k < h}
	{(-1)}^{{ {\SOE{\widetilde{a}}}_{h+1,k}}}
	{\up}^{f^{\ol{1}}_{h,k} }
	[ \SEE{a}_{h,k} +1]
	\ABJRS( \SE{A}+ E_{h,k}, \SO{A} - E_{h+1, k}, \bs{j} + \bsal_h,  r )  \\	
	& \hspace{1cm} +{(-1)}^{{ {\SOE{\widetilde{a}}}_{h+1,h}}}
	\frac{\up^{f^{\ol{1}}_{h,h}-j_h}}{\up-\up^{-1}}
		\Big\{\ABJRS( \SE{A}, \SO{A}- E_{h+1, h} , {\bs{j}} + \bsal_h, r )- \ABJRS(  \SE{A}, \SO{A}- E_{h+1, h} , {\bs{j}}-\bsal_h^+ , r )\Big\} \\
	& \hspace{1cm} +
	{(-1)}^{{ {\SOE{\widetilde{a}}}_{h+1,h+1}}}
	{\up}^{ f^{\ol{1}}_{h,h+1} }
	[\SEE{a}_{h,h+1} +1]
	\ABJRS( \SE{A}  + E_{h,h+1}, \SO{A} - E_{h+1, h+1}, \bs{j}-\ep_{h+1},  r )\\
	& \hspace{1cm}+
	\sum_{k > h+1}
	{(-1)}^{{ {\SOE{\widetilde{a}}}_{h+1,k}}}
	{\up}^{f^{\ol{1}}_{h,k}}
	[\SEE{a}_{h,k} +1]
	\ABJRS( \SE{A}  + E_{h,k}, \SO{A} - E_{h+1, k}, \bs{j},  r )\\
+&\sum_{k < h}
	{(-1)}^{{ {\SOE{\widetilde{a}}}_{h-1,k}} }
	{\up}^{f^{\ol{0}}_{h,k}}\boxed{a}_{h,k}
	\ABJRS( \SE{A} -E_{h+1, k} + 2 E_{h,k}, \SO{A} - E_{h,k}, {\bs{j}} + \bsal_h ,  r )  \\
	& \hspace{1cm} +
	{(-1)}^{{ {\SOE{\widetilde{a}}}_{h-1,h}} }\frac{{\up}^{f^{\ol{0}}_{h,h}-2j_h}}{(\up-\up^{-1})}
	\Big\{\frac{\up^{-1}}{[2]}
	\ABJRS( \SE{A}  -E_{h+1, h}, \SO{A}  -E_{h,h}, \bs{j} +2\ep_{h}-\ep_{h+1},  r ) \\
	\hspace{0.2cm}+&\frac{\up}{[2]}
	\ABJRS( \SE{A} -E_{h+1, h}, \SO{A}  -E_{h,h}, \bs{j} - 2 \ep_{h}-\ep_{h+1},  r )  -
	\ABJRS( \SE{A} -E_{h+1, h}, \SO{A}  -E_{h,h}, \bs{j}-\ep_{h+1},  r )\Big\}
	\\
	& \hspace{1cm} +
	{(-1)}^{{ {\SOE{\widetilde{a}}}_{h-1,h+1}} }
	{\up}^{f^{\ol{0}}_{h,h+1}+j_{h+1}}(\up-\up^{-1})\left[{{a}_{h,h+1}+1}\atop 2\right]
	\ABJRS( \SE{A}  + 2 E_{h,h+1}, \SO{A} - E_{h,h+1}, \bs{j},  r  )  \\
	& \hspace{1cm}+
	\sum_{k>h+1}
	{(-1)}^{{ {\SOE{\widetilde{a}}}_{h-1,k}} }
	{\up}^{f^{\ol{0}}_{h,k}}(\up-\up^{-1})\left[{{a}_{h,k}+1}\atop 2\right]
	\ABJRS( \Ahkzp +  E_{h,k}, \SO{A} - E_{h,k}, \bs{j},  r )\\
=:& \sdpCHe(h,\Ad(\bs{j},r)).
\end{align*}
\end{prop}

Recall $\OG_h(A,k)=-\AK(h+1,k)-\AK(h,k)$ from \eqref{g-hk}. Similar to \eqref{f_hk}, 
we introduce the notations
\begin{equation}\label{OG_h,k}
\Og^{-,\ol{0}}_{h,k}=\OG_h(A,k)-\SO{a}_{h,k},\quad\quad
\Og^{-,\ol{1}}_{h,k}=\left\{
\begin{aligned}
&\OG_h(A,k)+\SE{a}_{h,k}, &\mbox{if } h\neq k;\\
&\OG_h(A,k), &\mbox{if } h=k.
\end{aligned}
\right.
\end{equation}

\begin{prop}\label{mulformodd2}
	Let
	{$h \in [1,n)$},  {$\Ad=(\SE{A}|\SO{A})=(a^{\bar 0}_{i,j}|a^{\bar 1}_{i,j}) \in \MNZNS(n)$} with base $A=(a_{i,j})$, and $\boxed{a}_{\,i,j}:=(\up-\up^{-1})\left[{{a}_{i,j}+1}\atop 2\right]$.
Assume, for any   $\lambda \in \CMN(n, r-\snorm{A})$,
{${A+\lambda}$}  satisfies the SDP condition on the $h$-th row,
then the following multiplication formulas hold in {$\qSchvsQ$} for all  {$r\geq \snorm{A} $}:
\begin{align*}
(-1)^{\wp(\Ad)} & \ABJRS(O, E_{h+1, h}, \bs{ 0 },  r  ) \cdot \Ad(\bs{j},r)=\\
&
\sum_{k<h}
	{(-1)}^{{\SOE{\widetilde{a}}}_{h,k}}
 {\up}^{\Og^{-,\ol{0}}_{h,k}}
	\ABJRS( \SE{A} - E_{h,k}, \SO{A}+ E_{h+1, k}, \bs{j},  r ) \\
	&\hspace{1.2cm} +
	{(-1)}^{{\SOE{\widetilde{a}}}_{h,h}}
 {\up}^{\Og^{-,\ol{0}}_{h,h}+j_h}
	\ABJRS( \SE{A} , \SO{A}+ E_{h+1, h}, \bs{j},  r ) \\
	&\hspace{1.2cm} +
	{(-1)}^{{\SOE{\widetilde{a}}}_{h,h+1}}
	 {\up}^{\Og^{-,\ol{0}}_{h,h+1}}
	\ABJRS( \SE{A} - E_{h,h+1}, \SO{A}+ E_{h+1, h+1}, \bs{j}-\ep_h,  r ) \\
	&\hspace{1.2cm} +
 	\sum_{k>h+1}
	{(-1)}^{{\SOE{\widetilde{a}}}_{h,k}}
 {\up}^{\Og^{-,\ol{0}}_{h,k}}
	\ABJRS( \SE{A} - E_{h,k}, \SO{A}+ E_{h+1, k}, \bs{j} - \bsal_h^+,  r )\\
+&\sum_{k<h}
	{(-1)}^{{\SOE{\widetilde{a}}}_{h,k}+1}
	 {\up}^{\Og^{-,\ol{0}}_{h,k}}\boxed{a}_{h+1,k}
	\ABJRS( \Ahkzm + E_{h+1, k}, \SO{A}  -E_{h+1, k}, \bs{j},  r ) \\
	& \hspace{1.2cm}  +
	{(-1)}^{{\SOE{\widetilde{a}}}_{h,h}+1}
	{\up}^{\Og^{-,\ol{0}}_{h,h}+j_h}\boxed{a}_{h+1,h}
	\ABJRS( \SE{A} + 2E_{h+1, h}, \SO{A}  -E_{h+1, h}, \bs{j},  r ) \\
	 \hspace{1.2cm} +
	{(-1)}^{{\SOE{\widetilde{a}}}_{h,h+1}+1}&
\frac{ {\up}^{\Og^{-,\ol{0}}_{h,h+1}-2j_{h+1}}}{\up-\up^{-1}} \Big \{\frac{\up^{-1}}{[2]}
	\ABJRS( \SE{A}  - E_{h,h+1}, \SO{A}  -E_{h+1, h+1}, \bs{j} -\ep_{h}+ 2\ep_{h+1},  r ) \\
	&\hspace{1.2cm}  +
	 \frac{\up}{[2]} \ABJRS( \SE{A}  - E_{h,h+1}, \SO{A}  -E_{h+1, h+1}, \bs{j}-\ep_{h}-2 \ep_{h+1},  r ) \\
	&\hspace{1.2cm}  -
	 \ABJRS( \SE{A}  - E_{h,h+1}, \SO{A}  -E_{h+1, h+1}, \bs{j}-\ep_h,  r )
	\Big \}
	 \\
	& \hspace{1.2cm}  +
 	\sum_{k>h+1}
	{(-1)}^{{\SOE{\widetilde{a}}}_{h,k}+1}
	{\up}^{\Og^{-,\ol{0}}_{h,k}}\boxed{a}_{h+1,k}
	 \ABJRS( \Ahkzm + E_{h+1, k}, \SO{A}  -E_{h+1, k}, \bs{j}-\bsal_h^+,  r ) \\
+&\sum_{k < h}
 	{(-1)}^{{\SOE{\widetilde{a}}}_{h,k}} {\up}^{\Og^{-,\ol{1}}_{h,k}}
	[ \SEE{a}_{h+1,k}+1]
	\ABJRS( \SE{A} +  E_{h+1,k}, \SO{A}  - E_{h,k}, \bs{j},  r ) \\
	&  \hspace{1.2cm}  +
 	{(-1)}^{{\SOE{\widetilde{a}}}_{h,h}} {\up}^{\Og^{-,\ol{1}}_{h,h}}
	[\SEE{a}_{h+1,h}+1]
	\ABJRS( \SE{A} +  E_{h+1,h}, \SO{A}  - E_{h,h}, \bs{j} +  \ep_{h},  r ) \\
	 \hspace{1.2cm} +
 	{(-1)}^{{\SOE{\widetilde{a}}}_{h,h+1}}&\frac{{\up}^{\Og^{-,\ol{1}}_{h,h+1}-j_{h+1}}}{\up-\up^{-1}}\Big\{
	\ABJRS( \SE{A}, \SO{A}  - E_{h,h+1}, \bs{j} - \bsal_h,  r )-
	\ABJRS( \SE{A}, \SO{A}  - E_{h,h+1}, \bs{j}-\bsal_h^+,  r )
	\Big\}
	 \\
	&  \hspace{1.2cm}  +
	\sum_{k>h+1}
 		{(-1)}^{{\SOE{\widetilde{a}}}_{h,k}} {\up}^{\Og^{-,\ol{1}}_{h,k}}
	[{ \SEE{a}_{h+1,k}+1}]
	\ABJRS( \SE{A}  + E_{h+1,k}, \SO{A}  - E_{h,k}, \bs{j} -\bsal_h^+,  r )\\
=:&\sdpCHf(h,\Ad(\bs{j},r)).
\end{align*}
\end{prop}
\begin{rems}\label{longtail}  (1)
Since the general standard multiplication formulas for the odd case in Lemmas \ref{normlized-phidiag1} and \ref{standard-phiupper1} contain tail parts which are not explicitly known, we are not able to show in Propositions \ref{mulformdiag}, \ref{mulformodd1}, and \ref{mulformodd2} the general case as a sum of a ``head part" plus a ``tail part". We even do not know if such a product is a linear combination of some $\Bd(\bs{j},r)$!
We will prove this in the next section.

(2)
If $\Ad=(A|O)$, the multiplication formulas in Propositions \ref{mulformzerocor} and \ref{mulformeven-2} are identical to the type $A$ formulas given in \cite[Lem.~5.3]{BLM} or \cite[Th.~14.8]{DDPW}.

(3) In every multiplication formula in Propositions \ref{mulformzerocor} --\ref{mulformodd2}, all coefficients depend on the entries of $\Ad$ and $\bs{j}$ and are independent of $r\geq|A|.$
\end{rems}

\section{$\Uvqn$-type generators and relations in $\qSchvsQ$}\label{Quantum relations}
By the Schur--Weyl--Olshanski duality in Proposition \ref{DW2}, the $\mathbb Q(\up)$-superalgebra $\bs{\mathcal Q}_\up(n,r)$ as a homomorphic image of $\Uvqn$ possesses a set of generators which satisfy the defining relations (QQ1)--(QQ6) in Definition \ref{defqn}. We now use the long multiplication formulas in previous section to directly construct these generators in $\qSchvsQ$ and show that they satisfy (QQ1)--(QQ6). This approach in fact provides more. First, it builds the $\Uvqn$ structure from Hecke--Clifford and queer $\up$-Schur superalgebras. Second, this construction provides a new basis for $\Uvqn$ and all structure constants of generators relative to the new basis.

For any $r\geq1$, define the following elements in $\qSchvsQ$: for all $i\in[1,n]$ and $h\in[1,n)$,
\begin{equation}\label{degree-r generator}
\begin{aligned}
&G^\pm_{i,r}=\ABJS(O, O, \pm \ep_{i},r), \qquad
X_{h,r}=\ABJS(E_{h, h+1}, O, \bs{0},r) , \qquad
Y_{h,r}=\ABJS(E_{h+1, h}, O, \bs{0},r) , \\
&G_{\bar i,r}=\ABJS(O, E_{i,i}, \bs{0},r), \qquad
 X_{\bar h,r}=\ABJS( O, E_{h,h+1}, \bs{0} ,r),  \qquad Y_{\bar h,r}=\ABJS(O, E_{h+1,h},  \bs{0},r).
 \end{aligned}
\end{equation}
Clearly, the nonzero matrices involved in these elements define certain Chevalley generators for the queer Lie superalgebra $\mathfrak q_n(\mathbb Q)$.

\begin{lem}\label{mulform_ef-r}
For any integer {$h$} with {$ 1 \le h \le n-1$}, the following equations hold in {$\qSchvsQ$}:
{
\begin{itemize}
\item[{\rm(1)}]$
X_{\ol{h},r} Y_{\ol{h},r} + Y_{\ol{h},r} X_{\ol{h},r}
=  \frac{G_{h,r} G_{h+1,r} - G_{h,r}^{-1} G_{h+1,r}^{-1}}{{v} - {v}^{-1}}
	 + ({v} - {v}^{-1}) G_{\ol{h},r} G_{\ol{h+1},r}.$
\item[{\rm (2)}] $X_{h,r} Y_{\ol{h},r}    - Y_{\ol{h},r} X_{h,r}
= G_{\ol{h},r}  G_{h+1,r}^{-1}   -  G_{h,r}^{-1} G_{\ol{h+1},r}.$
\item[{\rm (3)}] $X_{\ol{h},r} Y_{h,r} - Y_{h,r} X_{\ol{h},r}
=  G_{h+1,r} G_{\ol{h},r}  - G_{\ol{h+1},r} G_{h,r} .$
\end{itemize}}
\end{lem}
\begin{proof}
{\rm(1)} \
By definition,
$$Y_{\ol{h},r} X_{\ol{h},r}=\ABJRS(O, E_{h+1, h}, \bs{0}, r) \cdot \ABJRS(O, E_{h, h+1}, \bs{0}, r)  =
		\sum_{
			\la,\mu \in \CMN(n,r) }
	[\mu | E_{h+1, h}]
		\cdot
		 [\lambda | E_{h, h+1}] $$
Since $[\mu | E_{h+1, h}]
		\cdot
		 [\lambda | E_{h, h+1}] \neq0$ implies $\co(\mu | E_{h+1, h}  ) = \ro(\lambda | E_{h, h+1})$ or { equivalently} $\la=\mu$.
Applying Lemma \ref{phiupper2}{\rm (4)} (and \eqref{def_ajr}) gives
\begin{align*}
Y_{\ol{h},r} X_{\ol{h},r}& =
		\sum_{		\mu \in  \CMN(n,r-1)  } (
	-[{\mu} - E_{h,h} |   E_{h, h+1} +  E_{h+1, h}]
	+\up^{-\mu_h}[ {\mu}_{h+1} + 1 ] [ {\mu} + E_{h+1, h+1} | O] ) \\
&= -\sum_{	 \substack{ \mu \in  \CMN(n,r-1) \\ \mu_h > 0} }
	[{\mu} - E_{h,h} |   E_{h, h+1} +  E_{h+1, h}]
+\sum_{ \mu \in  \CMN(n,r-1)  }
		\up^{-\mu_h}[ {\mu}_{h+1} + 1 ] [ {\mu} + E_{h+1, h+1} | O]\\
&=	-\ABJRS(O, E_{h, h+1} + E_{h+1, h}, \bs{0}, r )+\frac{1}{\up-\up^{-1}}  (\AJRS(\Od,-\ep_{h}+\ep_{h+1}, r)
	-    \AJRS(\Od, -\ep_{h}-\ep_{h+1}, r)).
\end{align*}
Here the second term is obtained by setting $\la=\mu+\ep_{h+1}$ and noting
 $$ \sum_{		\mu \in  \CMN(n,r-1)  }
	 	{  {\up}^{- \mu_{h}  }}[{\mu}_{h+1} + 1]
		[ {\mu} +  E_{h+1,h+1}| O]= \sum_{		\la \in  \CMN(n,r)  }
	 	 {\up}^{ {-\la}_{h}  }\Big(\frac{\up^{\la_{h+1}}}{\up-\up^{-1}}-\frac{\up^{-\la_{h+1}}}{\up-\up^{-1}}\Big)
		[ \la| O].$$

Similarly,  by Lemma \ref{phiupper2} {\rm(2)} and
  \eqref{def_ajr},
we have
\begin{align*}
X_{\ol{h},r}Y_{\ol{h},r} & =(O|E_{h, h+1})( \bs{0}, r) \cdot \ABJRS(O, E_{h+1, h}, \bs{0}, r)
 =
		\sum_{
			\mu \in \CMN(n,r-1)}
	 [\mu | E_{h, h+1} ]
		\cdot
		 [\mu | E_{h+1, h}]\\
&= \sum_{		\mu \in  \CMN(n,r-1)  }
	 	 {\up}^{ {\mu}_{h+1}  }[{\mu}_{h} + 1]
		[ {\mu} +  E_{h,h}| O]+\sum_{		\mu \in  \CMN(n,r-1)  }
	[\mu -  E_{h+1, h+1}   | E_{h, h+1} + E_{h+1, h} ] \\
&\qquad\quad-\sum_{		\mu \in  \CMN(n,r-1)  }
	(\up-\up^{-1}) [ \mu-  E_{h+1, h+1}    | E_{h,h} + E_{h+1,h+1}]\\
&= \frac{ 1}{ {\up} - \up^{-1} }(
		\AJRS({O},  \ep_{h} +   \ep_{h+1}, r)
		-
		\AJRS({O},  -\ep_{h}+ \ep_{h+1}, r))
		+\ABJRS( O ,  E_{h, h+1} + E_{h+1, h}, \bs{0}, r)\\
&\qquad\qquad-(\up-\up^{-1}) 	\ABJRS( O , E_{h,h} + E_{h+1,h+1} , \bs{0}, r).\end{align*}
On the other hand,  observe that for any $\lambda\in \Lambda(n,r-1)$ the matrix $\lambda+E_{h+1,h+1}$ satisfies the SDP condition on the $h$th row by Theorem \ref{CdA1} and then, by
applying Proposition \ref{mulformdiag},  we have
$$ G_{\ol{h},r} G_{\ol{h+1},r }=\ABJS(O, E_{h,h}, \bs{0},r)\ABJS(O, E_{h+1,h+1}, \bs{0},r)=-\ABJRS( O , E_{h,h} + E_{h+1,h+1} , \bs{0}, r).$$
Putting all together, we obtain (1).

{\rm(2)} Similar to the proof of (1), expanding $Y_{\ol{h},r} X_{h,r}$ and applying
 Lemma \ref{phiupper2}{\rm (3)} yields
\begin{align*}
Y_{\ol{h},r} X_{h,r}
= 	\ABJRS(E_{h, h+1}, E_{h+1, h}, \bs{0}, r)
 + 	\ABJRS(O, E_{h+1, h+1}, -\ep_{h}, r).
\end{align*}
Meanwhile, directly applying Proposition \ref{mulformeven-2} (1) gives
\begin{align*}
	X_{h,r} Y_{\ol{h},r} =\ABJRS(E_{h, h+1}, O, \bs{0}, r) \ABJRS(O, E_{h+1, h}, \bs{0},r  )
&=
\ABJRS(  E_{h, h+1}, E_{h+1, h}, \bs{ 0 }, r )
	 +
	\ABJRS( O,  E_{h, h},  -\ep_{h+1}, r ).
\end{align*}
Thus, (2) follows from Proposition \ref{mulformzerocor}.

 Finally, by Lemma \ref{phiupper2} {\rm {(1)}} and Proposition \ref{mulformeven-2},  part {\rm (3)} can be proved similarly.
\end{proof}

\begin{lem}\label{relation_eei-r}
For any integer {$h$} with {$ 1 \le h \le n-1$}, the following equations hold in {$\qSchvsQ$}:
\begin{align*}
{\rm (1)} \qquad
& X_{h,r} X_{h+1,r} - {v} X_{h+1,r} X_{h,r}
	= X_{\ol{h},r} X_{\ol{h+1},r} + {v} X_{\ol{h+1},r} X_{\ol{h},r}, \\
{\rm (2)} \qquad
& Y_{h,r} Y_{h+1,r} - {v} Y_{h+1,r} Y_{h,r}
	= - (Y_{\ol{h},r} Y_{\ol{h+1},r} + {v} Y_{\ol{h+1},r} Y_{\ol{h},r}),\\
{\rm (3)} \qquad
& X_{h,r} X_{\ol{h+1},r} - {v} X_{\ol{h+1},r} X_{h,r}
	= X_{\ol{h},r} X_{{h+1},r} - {v} X_{{h+1},r} X_{\ol{h},r}, \\
{\rm (4)} \qquad
& Y_{h,r} Y_{\ol{h+1},r} - {v} Y_{\ol{h+1},r} Y_{h,r}
	= Y_{\ol{h},r} Y_{{h+1},r} -{v} Y_{{h+1},r} Y_{\ol{h},r}.
\end{align*}
\end{lem}

\begin{proof}
We first prove (1).
By definition and noting \eqref{co=ro},
\begin{equation*}\label{relation_eei-r-pf-1}
\begin{aligned}
&X_{\ol{h},r} X_{\ol{h+1},r} + {v} X_{\ol{h+1},r} X_{\ol{h},r}\\
=&\sum_{\lambda,\mu\in\Lambda(n,r-1)}[\lambda|E_{h,h+1}][\mu|E_{h+1,h+2}]
+{ \up}\sum_{\xi,\gamma\in\Lambda(n,r-1)}[\xi|E_{h+1,h+2}][\gamma|E_{h,h+1}]\\
=&\sum_{\substack{\lambda\in\Lambda(n,r-1)}}[\lambda|E_{h,h+1}][\lambda|E_{h+1,h+2}]
+\up\sum_{\lambda\in\Lambda(n,r-1)}[\lambda+\alpha_h|E_{h+1,h+2}][\lambda-\alpha_{h+1}|E_{h,h+1}].
\end{aligned}
\end{equation*}
With our convention in \eqref{SMF}, when $\lambda_{h+1}=0$, the $(h+1)$-th entries of $\lambda+\alpha_h$ and $\lambda-\alpha_{h+1}$ are less than $0$, then
\begin{equation*}
\begin{aligned}
&[\lambda+\alpha_h|E_{h+1,h+2}]=0,[\lambda-\alpha_{h+1}|E_{h,h+1}]=0\\
&[ {\lambda}  + \alpha_{h} + E_{h+1, h+2} |O  ]=0,[{\lambda} -\alpha_{h+1} + E_{h, h+1}  | O]=0.
\end{aligned}
\end{equation*}
Thus, by Lemma \ref{relation_eei} (1),
we have
 \begin{equation*}
\begin{aligned}
X_{\ol{h},r} X_{\ol{h+1},r} + {v} X_{\ol{h+1},r} X_{\ol{h},r}&=\sum_{\substack{\lambda\in\Lambda(n,r-1)}}\big([  {\lambda} + E_{h, h+1} | O]
		[{\lambda}  +  E_{h+1, h+2} | O] \\
		&\quad\;-  {\up}  [ {\lambda}  + \alpha_{h} + E_{h+1, h+2} |O  ]
		[{\lambda} -\alpha_{h+1} + E_{h, h+1}  | O]\big).
\end{aligned}
\end{equation*}
On the other hand, using \eqref{def_ajr} and Lemma \ref{phiupper-even-norm}(2),  we obtain
\begin{equation*}
\begin{aligned}
&X_{h,r} X_{h+1,r} - {v} X_{h+1,r} X_{h,r}\\
=&(E_{h,h+1}|O)(\bs{0},r)(E_{h+1,h+2}|O)(\bs{0},r)-\up(E_{h+1,h+2}|O)(\bs{0},r)(E_{h,h+1}|O)(\bs{0},r)\\
=&{ \sum_{\lambda,\mu\in\Lambda(n,r-1)}[\lambda+E_{h,h+1}|O][\mu+E_{h+1,h+2}|O]}
{ -\up}\sum_{\xi,\gamma\in\Lambda(n,r-1)}[\xi+E_{h+1,h+2}|O][\gamma+E_{h,h+1}|O]\\
=&{ \sum_{\substack{\lambda\in\Lambda(n,r-1)}}([  {\lambda} + E_{h, h+1} | O]
		[{\lambda} +  E_{h+1, h+2} | O] -  {\up}  [ {\lambda}  + \alpha_{h} + E_{h+1, h+2} |O  ]
		[{\lambda} -\alpha_{h+1} + E_{h, h+1}  | O])}\\
=&X_{\ol{h},r} X_{\ol{h+1},r} + {v} X_{\ol{h+1},r} X_{\ol{h},r},
\end{aligned}
\end{equation*}
as desired.

We now prove (2).
First, applying Proposition \ref{mulformodd2} to the products $Y_{\ol{h},r} Y_{\ol{h+1},r}$ and $Y_{\ol{h+1},r}Y_{\ol{h},r}$ (as
$(\lambda+E_{h+2,h+1})$ and $(\lambda+E_{h+1,h})$ satisfy the SDP condition on the $h$-th row  and $(h+1)$-th row, respectively) yields
\begin{equation*}
\begin{aligned}
Y_{\ol{h},r} Y_{\ol{h+1},r}&=(O|E_{h+1,h})(\bs{0},r)(O|E_{h+2,h+1})(\bs{0},r)=-(O|E_{h+1,h}+E_{h+2,h+1})(\bs{0},r),\\
Y_{\ol{h+1},r}Y_{\ol{h},r}&=
(E_{h+2,h}|O)(\bs{0},r)+\up^{-1}(O|E_{h+1,h}+E_{h+2,h+1})(\bs{0},r).
\end{aligned}
\end{equation*}
Thus, $- (Y_{\ol{h},r} Y_{\ol{h+1},r} + {v} Y_{\ol{h+1},r} Y_{\ol{h},r})=-\up(E_{h+2,h}|O)(\bs{0},r).$
Second, applying Proposition \ref{mulformeven-2}(2) gives
\begin{equation*}
\begin{aligned}
Y_{h,r} Y_{h+1,r}&=(E_{h+1,h}|O)(\bs{0},r)(E_{h+2,h+1}|O)(\bs{0},r)=(E_{h+1,h}+E_{h+2,h+1})(\bs{0},r),\\
Y_{h+1,r} Y_{h,r}&=
(E_{h+2,h}|O)(\bs{0},r)+\up^{-1}(E_{h+1,h}+E_{h+2,h+1})(\bs{0},r).
\end{aligned}
\end{equation*}
Thus, $Y_{h,r} Y_{h+1,r} - {v} Y_{h+1,r} Y_{h,r}=-\up(E_{h+2,h}|O)(\bs{0},r)=- (Y_{\ol{h},r} Y_{\ol{h+1},r} + {v} Y_{\ol{h+1},r} Y_{\ol{h},r})$, as desired.

Finally, similar to the proof of formula (1), using \eqref{def_ajr} and Lemma \ref{relation_eei} (2), we can prove (3). Meanwhile, similar to the proof of (2),  $(\lambda+E_{h+2,h+1})$ satisfies SDP condition on the $h$-th row and $(\lambda+E_{h+1,h})$ satisfies SDP condition on the $(h+1)$-th row, then applying Proposition \ref{mulformeven-2}(2) and Proposition \ref{mulformodd2} proves (4).
\end{proof}

\begin{thm}\label{qqschur_reltion-r}
For each $r\geq 1$, there is  a superalgebra  homomorphism $\bs{\xi}_{n,r}: \Uvqn \to \qSchvsQ $ defined by
$${\genE}_{j} \mapsto X_{j,r}, \;
{\genE}_{\ol{j}} \mapsto X_{\ol{j},r}, \;
{\genF}_{j} \mapsto Y_{j,r}, \;
{\genF}_{\ol{j}} \mapsto Y_{\ol{j},r},  \; {\genK}_{i}^{\pm 1} \mapsto G_{i,r}^{\pm 1}, \;
{\genK}_{\ol{i}} \mapsto G_{\ol{i},r},$$
for all $1 \le i \le n, 1 \le j \le n-1$.
\end{thm}
\begin{proof}
We need to verify all the relations (QQ1)--(QQ6) in Definition \ref{defqn} with $\sfK_i^{\pm1}, \sfE_j,\sfF_j, \sfK_{\bar i},\sfE_{\bar j},\sfF_{\bar j}$ replaced by $G_{i,r}^{\pm 1},
X_{j,r},
Y_{j,r},
G_{\ol{i},r},
 X_{\ol{j},r},
 Y_{\ol{j},r},
 $
respectively. By Remark \ref{longtail}(2), relations involving only even generators (called even relations below) can be checked as in the $q$-Schur algebra case.
It suffices to check all the relations involving some odd generators.

{\bf For the non-even relations in (QQ1)},  $G_{i,r}G_{\bar j,r}=G_{\bar j,r}G_{i,r}$ follows from Proposition \ref{mulformzerocor}(a).

By Theorem \ref{CdA1}, any diagonal matrix statisfies SDP condition at any non-zero diagonal entry. Thus, by Proposition \ref{mulformdiag} with $A=(O|E_{j,j})$ so $\tilde{a}_{i,i}^{\bar1}=\begin{cases} 0,&\text{ if }i<j;\\
1,&\text{ if }i>j,\end{cases}$ we have, for $i,j\in[1,n]$,
\begin{align*}
\ABJS( O, E_{i,i}, \bs{ 0 },r)  \ABJS( O, E_{j,j}, \bs{ 0 } ,r)
	 =\begin{cases} -\ABJS( O, E_{j,j}+ E_{i,i}, \bs{ 0 },r ),&\text{ if }i<j; \\
 \ABJS( O, E_{j,j}+ E_{i,i}, \bs{ 0 } ,r),&\text{ if }i>j;\\
\frac{1}{(\up^2 - \up^{-2}) }
		\Big( \ABJS( O, O, 2 \ep_{j},r )  - \ABJS( O, O, -2 \ep_{j},r )
		  \Big),&\text{ if }i=j.\end{cases}.
\end{align*}
Hence,
$
G_{\ol{i},r} G_{\ol{j},r} + G_{\ol{j},r} G_{\ol{i},r}=0,~
G_{\ol{k},r}^2= \frac{  G_{k,r}^{2} - G_{k,r}^{-2}  }{{v}^2 - {v}^{-2}  },
$
proving (QQ1).

{\bf For the non-even relations in (QQ2)}, all follow from Proposition \ref{mulformzerocor}.

{\bf For the non-even relations in (QQ3)}, we divide the 12 relations into four subsets $\mathcal R_i$ ($1\leq i\leq 4$), where subset $\mathcal R_i$ consists of the two relations in the $i$-th line and the $i$-th relation in the last line of (QQ3).

By Theorem \ref{CdA1},
for any {$\lambda \in \CMN(n, r-1)$} and $i\in[1,n]$, $j\in[1,n)$, {$E_{i, i} + \lambda$}, $E_{j, j+1} + \lambda$
{$E_{j+1, j} + \lambda$} satisfy the SDP condition at every non-zero entry on the $i$-th row. So, in the following computations, Propositions \ref{mulformdiag}, \ref{mulformodd1} and \ref{mulformodd2} are all applicable.

{\bf Case $\mathcal R_1$.} We apply Proposition \ref{mulformdiag} to compute the following product:
$$G_{\ol{i},r} X_{j,r}=\ABJS( O, E_{i,i}, \bs{ 0 },r) \ABJS(E_{j, j+1}, O,  \bs{0}, r)=
\begin{cases} \up \ABJS( E_{i, i+1}, E_{i,i}, \bs{0},r )
		+    \ABJS( O,  E_{i,i+1}, - \ep_{i},r ),&\text{ if }j=i;\\
		 \ABJS(E_{j,j+1}, E_{i,i}, \bs{0} ,r),&\text{ if }j\neq i.
		\end{cases}$$
Also, by Proposition \ref{mulformeven-2}, $X_{j,r} G_{\ol{i},r}=\begin{cases}  {v} 	\ABJS( E_{i-1,i}, E_{i,i}, \bs{0},r )
		+  	\ABJS( O, E_{i-1,i}, -\ep_{i},r ),&\text{ if }j=i-1;\\
\ABJS(E_{j,j+1}, E_{i,i}, \bs{0} ,r),&\text{ if }j\neq i-1.
\end{cases}
$	
Thus, the three relations in $\mathcal R_1$ follow.		

{\bf Case $\mathcal R_2$.}  Symmetrically, the three relations in $\mathcal R_2$ follow from the multiplication formulas:
$$G_{\ol{i},r} Y_{j,r}
 = \begin{cases}    {\up}^{-1} \ABJS( E_{i,i-1}, E_{i,i}, \bs{ 0 },r )+\ABJS( O, E_{i,i-1}, \ep_{i},r ),&\text{if }j=i-1;\\
\ABJS( E_{j+1, j},  E_{j,j},\bs{0},r ),& \text{if }j\neq i-1.
\end{cases}$$
$$
Y_{j,r} G_{\ol{i},r}=\ABJS(E_{j+1, j}, O, \bs{0},r)  \ABJS( O, E_{i,i}, \bs{ 0 },r) =\begin{cases}   \up^{-1}\ABJS( E_{i+1, i}, E_{i,i}, \bs{ 0 } ,r)
	 + \ABJS(O, E_{i+1,i},  \ep_{i},r ),&\text{if }j=i;\\
\ABJS(E_{j+1,j}, E_{i,i}, \bs{0} ,r),&\text{ if }j\neq i.\\
\end{cases}
$$

{\bf Case $\mathcal R_3$.}  By
Proposition \ref{mulformdiag}, we have
$$G_{\ol{i},r}X_{\ol{j},r}=\begin{cases}
\ABJS(O,E_{j,j+1}+E_{i,i}, \bs{0},r ), &\mbox{if } j<i;\\
-\up\ABJS(O, E_{i,i}+E_{i,i+1},  \bs{0},r )+\ABJS(E_{i,i+1},O,-\ep_{i},r),&\text{if }j=i;\\
-\ABJS(O,E_{j,j+1}+E_{i,i}, \bs{0},r ),&\mbox{if } j> i;
\end{cases}
$$
while, by Proposition \ref{mulformodd1},
$$
X_{\ol{j},r}G_{\ol{i},r}=\begin{cases}-\ABJS(O, E_{i,i}+E_{j,j+1},  \bs{0},r ).&\text{if }j<i-1;\\
-\up\ABJS(O, E_{i,i}+E_{i-1,i},  \bs{0},r )+\ABJS(E_{i-1,i},O,-\ep_{i},r),&\text{if }j=i-1;\\
\ABJS(O, E_{i,i}+{ E_{j,j+1}},  \bs{0},r ).&\text{if }j\geq i.\\
\end{cases}
$$
Thus, the three relations in $\mathcal R_3$ follow.

{\bf Case $\mathcal R_4$.}  The remaining  three relations in $\mathcal R_4$ can be proved similarly, noting
$$G_{\ol{i},r}Y_{\ol{j},r}=\begin{cases}
\ABJS(O,E_{j+1,j}+E_{i,i}, \bs{0},r ), &\mbox{if } j<i-1;\\
\up^{-1}\ABJS(O, E_{i,i-1}+E_{i,i},  \bs{0},r )+\ABJS(E_{i,i-1},O,\ep_{i},r),&\text{if }j=i-1;\\
-\ABJS(O,E_{j+1,j}+E_{i,i}, \bs{0} ,r), &\mbox{if } j\geq i;
\end{cases}
$$
$$Y_{\ol{j},r}G_{\ol{i},r}=\begin{cases}
-\ABJS(O,E_{j+1,j}+E_{i,i}, \bs{0},r ), &\mbox{if } j<i;\\
 -\up^{-1}\ABJS(O, E_{i,i}+E_{i+1,i},  \bs{0} ,r)+\ABJS(E_{i+1,i},O,\ep_{i},r),&\text{if }j=i;\\
\ABJS(O,E_{j+1,j}+E_{i,i}, \bs{0} ,r), &\mbox{if } j> i.
\end{cases}
$$
This completes checking (QQ3).

{\bf For the non-even relations in (QQ4)}, the $i=j$ case is done on Lemma \ref{mulform_ef-r}. We now prove the $i\neq j$ case.
By Proposition \ref{mulformodd1},
$$X_{\bar i,r} Y_{\ol{j},r}=  \ABJS(O,E_{i, i+1}, \bs{0},r)  \ABJS(O, E_{j+1,j}, \bs{0},r )=\begin{cases}
-(O|E_{i,i+1}+E_{j+1,j})(\bs{0},r),&\text{if }i<j;\\
\;\;(O|E_{i,i+1}+E_{j+1,j})(\bs{0},r),&\text{if }i>j,\end{cases}$$
which is the same as $-Y_{\ol{j},r}X_{\bar i,r} $, by Proposition \ref{mulformodd2}. Here the required SDP condition can be checked easily.
Similarly,
$$X_{i,r} Y_{\ol{j},r}=  \ABJS(E_{i, i+1}, O,\bs{0},r)  \ABJS(O, E_{j+1,j}, \bs{0},r )=(E_{i,i+1}|E_{j+1,j})(\bs{j},r)=Y_{\ol{j},r}X_{i,r}$$ and
$X_{\bar i,r} Y_{{j},r}= (E_{j+1,j}|E_{i,i+1})(\bs{0},r)=Y_{{j},r}X_{\bar i,r} $, proving (QQ4).

{\bf For the non-even relations (QQ5)}, we first prove the two relations with square terms. Recall the notation for $\pm$-matrices defined in \eqref{Ahk}.
Since, by Theorem \ref{CdA1},
for any {$\lambda \in \CMN(n, r-1)$} and $k\in[1,n]$,
{${(E_{i, i+1} + \lambda)}^{+}_{i,k}\in M_n(\NN)$} forces $k=i+1$ and ${(E_{i, i+1} + \lambda)}^{+}_{i,i+1}$
satisfies the SDP condition on the $i$-th row. Thus, applying Propositions \ref{mulformodd1} and \ref{mulformeven-2}(1) yields
\begin{align*}
X_{\ol{i},r} ^2
&=    \ABJS( O, E_{i,i+1}, \bs{0},r ) \cdot \ABJS( O, E_{i,i+1}, \bs{0} ,r) = -(\up - \up^{-1}) \ABJS( 2 E_{i,i+1}, O,  \bs{0},r  );\\
X_{i,r} ^2&=[2]	\ABJS( 2 E_{i,i+1}, O, \bs{0},r)=(\up+\up^{-1})\ABJS( 2 E_{i,i+1}, O, \bs{0},r).
\end{align*}
Hence, $
X_{\ol{i},r} ^2=-\frac{\up - \up^{-1}}{\up + \up^{-1}}X_{i,r} ^2.
$

Similarly, since {$E_{i+1, i} + \lambda$} ($\lambda \in \CMN(n, r-1)$) satisfies the SDP condition on the $i$-th row, by Propositions  \ref{mulformeven-2}(2) and  \ref{mulformodd2}, we have
\begin{align*}
  \frac{{v} - {v}^{-1}}{{v} + {v}^{-1}} Y_{i,r}^2
&=  ({v}- {v}^{-1} )
	\ABJS( 2E_{i+1, i}, O,   \bs{0} ,r) =Y_{\ol{i},r}^2 .
\end{align*}

For the relations in the second line of (QQ5), 
applying Propositions \ref{mulformeven-2} and \ref{mulformodd1}, we obtain, for all $i,j\in[1,n]$,
(recall $ X_{i,r}=  \ABJS(E_{i, i+1}, O, \bs{0},r)$ and $  X_{\ol{j},r}=\ABJS(O, E_{j,j+1}, \bs{0},r )$.)
\begin{align}\label{XiXjbar}
X_{i,r} X_{\ol{j},r}
&=\begin{cases}  \ABJS( E_{i,i+1}, E_{j, j+1}, \bs{ 0 },r ),&\text{if }j\neq i+1;\\
\up^{-1}\ABJS( E_{i,i+1}, E_{j, j+1}, \bs{ 0 },r )+{(O|E_{i,i+2})({\bf0},r)},&\text{if }j=i+1,
\end{cases}\\\label{XjbarXi}
X_{\ol{j},r} X_{i,r}&=  \ABJS( E_{i,i+1}, E_{j, j+1}, \bs{ 0 },r ),\;\;\text{if }j\neq i-1.
\end{align}

Hence, $X_{i,r} X_{\ol{j},r} - X_{\ol{j},r} X_{i,r}= 0$ for all $i,j$ with $|i-j|\neq1$. The proof for the $Y$ case is similar.

Also, for $|i-j|>1$ and $\la$, the +-matrix  $(E_{j,j+1}+\la)^+_{i,k}${  (resp., $(E_{i,i+1}+\la)^+_{j,k}$) satisfies SDP condition on $i$-th (resp., $j$-th row)}. By Proposition \ref{mulformodd1},
$$X_{\bar i,r} X_{\ol{j},r}=  \ABJS(O,E_{i, i+1}, \bs{0},r)  \ABJS(O, E_{j,j+1}, \bs{0},r )=\begin{cases}
\;\;(O|E_{i,i+1}+E_{j,j+1})(\bs{0},r),&\text{if }i<j;\\
-(O|E_{i,i+1}+E_{j,j+1})(\bs{0},r),&\text{if }i>j.\end{cases}$$
Hence, $X_{\bar i,r} X_{\ol{j},r}= -X_{\ol{j},r}X_{\bar i,r}.$ A similar argument shows $Y_{\bar i,r} Y_{\ol{j},r}= -Y_{\ol{j},r}Y_{\bar i,r}$.

The last four relations of (QQ5) follows from Lemma \ref{relation_eei-r}, completing (QQ5) checking.

{\bf Finally, we prove the non-even relations in (QQ6),} we only prove the {$j=i-1$} case.
The $j=i+1$ case can be proved similarly. After shiftting indices, we first prove
\begin{align} \label{serre-X}
 X_{i+1,r}^2 X_{\ol{i},r} - ( {v} + {v}^{-1} ) X_{i+1,r} \boxed{X_{\ol{i},r} X_{i+1,r}} + X_{\ol{i},r}  X_{i+1,r}^2 = 0,
\end{align}
Since the SDP condition does not hold for the boxed product, we modify the left hand side slightly. Recall the second relation in line four of (QQ5):
\begin{align*}
  X_{i,r} X_{\ol{i+1},r} - {v} X_{\ol{i+1},r} X_{i,r}
	= X_{\ol{i},r} X_{i+1,r} - {v} X_{i+1,r} X_{\ol{i},r}.
\end{align*}
Thus, the left hand side of \eqref{serre-X} becomes
\begin{align*}
\text{LHS}&= X_{i+1,r}^2 X_{\ol{i},r} - \up^{-1} X_{i+1,r} X_{\ol{i},r} X_{i+1,r} + X_{\ol{i},r}  X_{i+1,r}^2-\up  X_{i+1,r} X_{\ol{i},r} X_{i+1,r}\\
&=	- {v}^{-1}X_{i+1,r} (X_{\ol{i},r} X_{i+1,r} -  {v}X_{i+1,r} X_{\ol{i},r} )
	+ (X_{\ol{i},r}  X_{i+1,r}  - {v} X_{i+1,r} X_{\ol{i},r} )X_{i+1,r}\\
&= - {v}^{-1}X_{i+1,r}  X_{i,r} X_{\ol{i+1},r} +  X_{i+1,r}  X_{\ol{i+1},r} X_{i,r}
	+ X_{i,r} X_{\ol{i+1},r}X_{i+1,r} - {v} X_{\ol{i+1},r} X_{i,r} X_{i+1,r}.
\end{align*}
Applying \eqref{XiXjbar}, \eqref{XjbarXi}, and Proposition \ref{mulformeven-2}{\rm(1)},
 we obtain
\begin{align*}
X_{i+1,r} & (X_{i,r} X_{\ol{i+1},r})\overset{\eqref{XiXjbar}}=\ABJS(E_{i+1, i+2}, O, \bs{0},r)  \cdot\Big(\up^{-1}\ABJS( E_{i,i+1}, E_{i+1, i+2}, \bs{ 0 },r )+{ (O|E_{i,i+2})({\bf0},r)}\Big)\\
&=\up^{-1} \ABJS( E_{i,i+1}+ E_{i+1,i+2}, E_{i+1, i+2}, \bs{0},r ) + \ABJS( E_{i+1,i+2}, E_{i,i+2}, \bs{0} ,r) , \\
X_{i+1,r} & (X_{\ol{i+1},r} X_{i,r})\overset{\eqref{XjbarXi}}=\ABJS(E_{i+1, i+2}, O, \bs{0},r)  \cdot
				\ABJS(E_{i, i+1}, E_{i+1, i+2}, \bs{0},r) \\
&=\ABJS( E_{i, i+1} + E_{i+1,i+2}, E_{i+1,i+2}, \bs{0} ,r) , \\
X_{i,r}& (X_{\ol{i+1},r}X_{i+1,r})\overset{\eqref{XjbarXi}} = \ABJS(E_{i, i+1}, O, \bs{0},r) \cdot
		 \ABJS(E_{i+1, i+2}, E_{i+1, i+2}, \bs{0},r)  \\
&=	\up^{-2}\ABJS( E_{i+1, i+2} + E_{i,i+1}, E_{i+1, i+2}, \bs{0},r )
	+ {v} \ABJS( E_{i,i+2}, E_{i+1, i+2}, \bs{0},r )
	+ \up^{-1}\ABJS( E_{i+1, i+2}, E_{i, i+2}, \bs{0},r ),
\end{align*}
and, by Propositions  \ref{mulformeven-2}{\rm(1)} and \ref{mulformodd1},
$$\aligned X_{\ol{i+1},r}& (X_{i,r} X_{i+1,r})=\ABJS(O, E_{i+1, i+2}, \bs{0},r)  \cdot
\Big(\up^{-1}
		\ABJS(E_{i, i+1}+E_{i+1,i+2}, O, \bs{0},r)+
		\ABJS(E_{i, i+2}, O, \bs{0},r) \Big)\\
&=\up^{-1}\ABJS(E_{i+1, i+2} + E_{i,i+1}, E_{i+1,i+2}, \bs{0} ,r)
	+ \ABJS(  E_{i,i+2}, E_{i+1,i+2}, \bs{0} ,r).\endaligned
	$$
Substituting gives LHS$=0$, proving \eqref{serre-X}.

The $Y$ case can be proved symmetrically by using  Proposition \ref{mulformodd2}  because for any {$\lambda \in \CMN(n, r)$} with {$r>0$},
we have {$d_A = 1$} when {$A=\lambda + E_{i+1, i}$} or {$\lambda + 2E_{i+1, i}$} and then the SDP condition applies to all cases due to Theorem \ref{CdA1}.

This completes the proof of the theorem.\end{proof}

The first application of the theorem is to answer the question raised in Remark \ref{longtail}(2).
Let
\begin{equation}\label{setGr}
 \fsG_{n,\bullet} = \{ G_{i,\bullet}, G_{i,\bullet}^{-1}, G_{\ol{i},\bullet},
	X_{j,\bullet}, X_{\ol{j},\bullet},
	Y_{j,\bullet}, Y_{\ol{j},\bullet}
	\where  1 \le i \le n,\  1 \le j \le n-1
	\}.
\end{equation}
Then, for every $r\geq1$,  $\fsG_{n,r}=\{ G_{i,r}, G_{i,r}^{-1}, G_{\ol{i},r},\ldots\} $ is a subset of $\qSchvsQ$.

\begin{cor}\label{common_form}
For any {$\Ad \in \MNZNS(n)$} with base $A$
and any $Z_\bullet  \in\fsG_{n,\bullet} $,
there exist ${ \scp_{\Bd,\bs{j}}(Z_\bullet,\Ad)} \in \Qv $,
for some $\Bd \in \MNZNS(n)$ and $\bs{j}\in {\ZZ}^n$,
such that, for all $r\geq|A|$, the following holds in $\qSchvsQ$
\begin{equation}\label{general-ZA}
\begin{aligned}
Z_r  \cdot \AJS(\Ad, \bs{0},r)
&=
	\sum_{\Bd, \bs{j}}{ \scp_{\Bd,\bs{j}}(Z_\bullet,\Ad)}  \AJS(\Bd, \bs{j},r ).
\end{aligned}
\end{equation}
\end{cor}
\begin{proof} If $Z_r$ is even (i.e., $Z_r\in\{G_{i,r}, G_{i,r}^{-1},X_{j,r},Y_{j,r}\}$) or $Z_r=G_{\bar n,r}$, the assertion follows from Propositions \ref{mulformzerocor}, \ref{mulformeven-2}, and \ref{mulformdiag} (the $h=n$ case).
If $Z_r$ is an odd element in
$$\{X_{\bar j,r},Y_{\bar j,r}, G_{\bar j,r}\mid 1\leq j<n\},$$ then, by the theorem and Remark \ref{induction_N}, the existence of the expression can be proved by an induction on $n$ down from $n-1$ to 1.
\end{proof}

The next result continues to address Remark \ref{longtail}(2), writing the general case in Propositions \ref{mulformdiag}, \ref{mulformodd1}, and \ref{mulformodd2} as a sum of a head part and a tail part.
Recall the head parts $\sdpCHk(h,\Ad(\bs{j},r))$, $\sdpCHe(h,\Ad(\bs{j},r))$, and $\sdpCHf(h,\Ad(\bs{j},r))$ introduced in Propositions \ref{mulformdiag}, \ref{mulformodd1}, and \ref{mulformodd2}, respectively.
\begin{cor}\label{diag-com}
Let
$h\in[1,n)$,	{$i \in [1,n]$},  and {$\Ad \in \MNZNS(n)$}  with base matrix $A$.
Then, for all $r\geq|A|$,  the following multiplication formulas hold in {$\qSchvsQ$}:
\begin{align}\label{diag-com-pf}
&G_{\ol{i},r}\cdot\AJS(\Ad,\bs{0},r)=\sdpCHk(i,\Ad(\bs{0},r))+\sum_{\substack{\Bd,\bs{j},\\B \prec A}} { \scp_{\Bd,\bs{j}}(G_{\ol{i},\bullet},\Ad)}\AJS(\Bd,\bs{j},r),\\
\label{upper-com-formula}
&X_{\ol{h},r}\cdot\AJS(\Ad,\bs{0},r)=\sdpCHe(h,\Ad(\bs{0},r))+\sum_{ \substack{\Bd,\bs{j}, \\  \exists k, B\prec{A}^+_{h,k}} }{ \scp_{\Bd,\bs{j}}(X_{\ol{h},\bullet},\Ad)}\AJS(\Bd,\bs{j},r),\\
\label{lower-com-formula}
&Y_{\ol{h},r}\cdot\AJS(\Ad,\bs{0},r)=\sdpCHf(h,\Ad(\bs{0},r))+(\up-\up^{-1})\sdpCHHf(h,\Ad(\bs{0},r))\\
&\hspace{3in}+\sum_{\substack{\Bd,\bs{j},\\  \exists k, B\prec{A}^-_{h,k}}} { \scp_{\Bd,\bs{j}}(Y_{\ol{h},\bullet},\Ad)}\AJS(\Bd,\bs{j},r),\notag
\end{align}
where $\Bd \in \MNZNS(n)$  with base $B$, $\bs{j}\in\ZZ^n$ and $\sdpCHHf(h,\Ad(\bs{0},r))$ is a $\mathbb{Q}(\up)$-linear combination of $\Md(\bs{j}',r)$ with coefficients independent of $r$ with $\Md\in \MNZNS(n)$ being one of the following matrices  for $1\leq l<k\leq n$ ($\ndelta_{i,j}=1-\delta_{i,j}$):
\begin{equation}\label{SFM}
\begin{aligned}
&(\SE{A} -\ndelta_{h,k}E_{h,k} +\ndelta_{h+1,k}E_{h+1, k} \mp \ndelta_{h+1,l}E_{h+1,l}| \SO{A}\pm E_{h+1,l}),\\
&(\SE{A} \mp \ndelta_{h+1,l}E_{h+1,l} |\SO{A} - E_{h,k} + E_{h+1,k}\pm E_{h+1,l}),\\
&(\SE{A} +\ndelta_{h+1,k}2E_{h+1,k}\mp \ndelta_{h+1,l}E_{h+1,l} | \SO{A}  - E_{h,k} - E_{h+1,k}\pm E_{h+1,l}).\\
\end{aligned}
\end{equation}
\end{cor}

\begin{proof}We only prove \eqref{lower-com-formula}. The proofs for other two are similar (and simpler).

We first expand the left hand side of \eqref{lower-com-formula} and apply Proposition \ref{standard-phiupper1}(2):
\begin{equation}
\begin{aligned}
&Y_{\ol{h},r}\cdot\AJS(\Ad,\bs{j},r)=
	\sum_{\substack{ \\ \lambda \in  \CMN(n,r-\snorm{A}) }
	}[\la-\ep_h+\ro(A) |E_{h+1, h} ]
	\cdot[\SE{A} + \lambda  | \SO{A}] \\
&=
	\sum_{\substack{ \\ \lambda \in  \CMN(n,r-\snorm{A}) }
	} \Big\{\sdpNHf(h,(\SE{A} + \lambda  | \SO{A}))+(\up-\up^{-1})\text{\rm HH}\overline{\textsc f}(h,(\SE{A}+ \lambda  | \SO{A}))\\
	&\hspace{4.5cm} +\sum_{\Bd\in \MNZ(n,r)\atop \exists k, B\prec{(\lambda+A)}^-_{h,k}\preeq A^-_{h,k}}{ \scp_{F^\star_{\bar{h},\la+\ro(A)},A^\star_\la,\Bd}}[{\Bd}]\Big\},
\end{aligned}
\end{equation}
where  $A^\star_\la=(\SE{A} + \lambda  | \SO{A})$.

By Proposition \ref{mulformodd2}, we have
\begin{equation}\label{lower-com-formula-pf-1}
	\sum_{\substack{ \\ \lambda \in  \CMN(n,r-\snorm{A}) }
	} \sdpNHf(h,(\SE{A} + \lambda  | \SO{A}))\\
=\sdpCHf(h,\Ad({\bf0},r)).
\end{equation}

On the other hand, we claim that
\begin{equation}\label{lower-com-formula-pf-3}
\begin{aligned}
\sum_{\substack{ \lambda \in  \CMN(n,r-\snorm{A}) }}
	 (\up-\up^{-1})\text{\rm HH}\overline{\textsc f}(h,(\SE{A} + \lambda  | \SO{A}))=:(\up-\up^{-1})\sdpCHHf(h,\Ad(\bs{0},r))
\end{aligned}
\end{equation}
is a linear combination of $\Md(\bs{j}',r)$ with coefficients independent of $r$ with $\Md\in \MNZNS(n)$ being one of the matrices in \eqref{SFM}.

Indeed, by  Lemma \ref{standard-phiupper1}(2), the LHS of \eqref{lower-com-formula-pf-3} has the form $$(\up-\up^{-1})(-1)^{\wp(\Ad)}
 \sum_{k=1}^n\sum_{l<k}(\clubsuit)_{h,k,l},$$
where, putting
$\tau_{h,k,l}(\SEE{A}+\lambda|\SOE{A}):=2(\overrightarrow{\bf r}^{l}_{h+1}(A+\lambda){-\overrightarrow{\bf r}_{h+1}^{k-1}(A+\lambda)})+\OG_h((A+\lambda),k)+(A+\lambda)_{h+1,l}$,
\begin{equation}\label{club}
\begin{aligned}
&(\clubsuit)_{h,k,l}=\sum_{\substack{ \lambda \in  \CMN(n,r-\snorm{A}) }}\bigg\{(-1)^{\widetilde{a}^{\bar1}_{h,l}} \up^{\tau_{h,k,l}(\SEE{A}+\lambda|\SOE{A})+(A^{\bar0}+\lambda)_{h,k}}  \Big(\up^{-(A+\lambda)_{h,k}}[{(A^{\bar0}+\lambda)}_{h+1, k} +1]\cdot\\
&([{\Ahkzm -E_{h+1,l}+\lambda| \SO{A}+E_{h+1,l}}]-{[{(A+\lambda)_{h+1,l}}]_{\up^2}}[{\Ahkzm +E_{h+1,l}+\lambda| \SO{A}-E_{h+1,l}}])\\
&+([{{\SE{A} -E_{h+1,l}+\lambda |\Ahkom+E_{h+1,l}}}]-{[{(A+\lambda)_{h+1,l}}]_{\up^2}}[{{\SE{A} +E_{h+1,l}+\lambda |\Ahkom-E_{h+1,l}}}])\\	
&-(\up-\up^{-1})\left[{(A+\lambda)_{h+1, k} +1}\atop 2\right]\big(
	[{\SE{A} + 2E_{h+1,k}-E_{h+1,l}+\lambda |\Ahkomm+E_{h+1,l}}]\\
&\hspace{5cm}-{[{(A+\lambda)_{h+1,l}}]_{\up^2}}[{\SE{A} + 2E_{h+1,k}+E_{h+1,l}+\lambda |\Ahkomm-E_{h+1,l}}]\big)
\Big)\bigg\}.
\end{aligned}
\end{equation}
Since
$\tau_{h,k,l}(\Ad)=2(\overrightarrow{\bf r}^{l}_{h+1}-\overrightarrow{\bf r}_{h+1}^{k-1})+\OG_h(A,k)+{a}_{h+1,l},$
where  $\overrightarrow{\bf r}^{j}_{h+1}:=\overrightarrow{\bf r}^{j}_{h+1}(A)$, it follows that
\begin{equation}
\tau_{h,k,l}(\SEE{A}+\lambda|\SOE{A})=\tau_{h,k,l}(\Ad)+\varepsilon\lambda_h+\varepsilon'\lambda_{h+1},
\end{equation}
with $\varepsilon,\varepsilon'\in\{-1,0,1\}$, depending on $h,k,l$.

Now we compute $(\clubsuit)_{h,k,l}$ in two cases.

{\bf Case 1. $l\neq h+1$. } If $\SO{a}_{h+1,l}=0$, then $ \SO{A}-E_{h+1,l}$, $\Ahkom-E_{h+1,l} $ and $\Ahkomm-E_{h+1,l}$ all have a negative entry. Thus, noting $(A^{\bar0}+\lambda)_{h,k}{ -}(A+\lambda)_{h,k}=-a^{\bar1}_{h,k}$,
\eqref{club} becomes
\begin{equation}
\begin{aligned}
&(\clubsuit)_{h,k,l}
=(-1)^{\widetilde{a}^{\bar1}_{h,l}}\up^{\tau_{h,k,l}(\Ad)}\sum_{\substack{ \lambda \in  \CMN(n,r-\snorm{A}) }}\up^{\lambda\centerdot(\varepsilon\ep_h+\varepsilon' \ep_{h+1})}\big\{\up^{-\SO{a}_{h,k}}[{(A^{\bar0}+\lambda)}_{h+1, k} +1]\cdot\\
& [{\Ahkzm -E_{h+1,l}+\lambda| \SO{A}+E_{h+1,l}}]+\up^{{\color{black}(A^{\bar0}+\lambda)}_{h,k}}[{{\SE{A} -E_{h+1,l}+\lambda |\Ahkom+E_{h+1,l}}}]\\	
&-(\up-\up^{-1})\up^{{\color{black}(A^{\bar0}+\lambda)}_{h,k}}\left[{(A+\lambda)_{h+1, k} +1}\atop 2\right][{\SE{A} + 2E_{h+1,k}-E_{h+1,l}+\lambda |\Ahkomm+E_{h+1,l}}]
\big\}.
\end{aligned}
\end{equation}
{\color{black}Replacing $\Ad$ by $(\SE{A}-E_{h+1,l}|\SO{A}+E_{h+1,l})$ and $(E_{h,h+1}|O)$ by $(E_{h+1,h}|O)$  when we write $\mathcal{Y}_1,\mathcal{Y}_2,\mathcal{Y}_3$ as linear combinations of $\Bd(\bs{j},r)$ with $\Bd \in \MNZNS(n)$ and $\bs{j}\in\ZZ^n$ in the proof of Proposition \ref{mulformeven-2}(1), one proves that $(\clubsuit)_{h,k,l}$ is a linear combination of $\Md(\bs{j}',r)$ with $\Md$ of the form \eqref{SFM}}
for $\ndelta_{h+1,l}=1$.

Similarly,
if $\SO{a}_{h+1,l}=1$, then \eqref{club} becomes
\begin{equation}
\begin{aligned}
&(\clubsuit)_{h,k,l}=(-1)^{\widetilde{a}^{\bar1}_{h,l}}\up^{\tau_{h,k,l}(\Ad)}\sum_{\substack{ \lambda \in  \CMN(n,r-\snorm{A}) }}\up^{\lambda\centerdot(\varepsilon\ep_h+\varepsilon' \ep_{h+1})}[a_{h+1.l}]_{\up^2}\cdot\\
&\Big\{-\up^{-\SO{a}_{h,k}}[{(A^{\bar0}+\lambda)}_{h+1, k} +1]\cdot [{\Ahkzm +E_{h+1,l}+\lambda| \SO{A}-E_{h+1,l}}]\\
&-\up^{(A^{\bar0}+\lambda)_{h,k}}[{{\SE{A} +E_{h+1,l}+\lambda |\Ahkom-E_{h+1,l}}}]\\	
&+(\up-\up^{-1})\up^{(A^{\bar0}+\lambda)_{h,k}}\left[{(A+\lambda)_{h+1, k} +1}\atop 2\right][{\SE{A} + 2E_{h+1,k}+E_{h+1,l}+\lambda |\Ahkomm-E_{h+1,l}}]
\Big\}.
\end{aligned}
\end{equation}
Like the $\SO{a}_{h+1,l}=0$ case, $(\clubsuit)_{h,k,l}=\text{a lin. comb. of }\Md(\bs{j}',r)$,
where $\Md$ is of the form \eqref{SFM} for $\ndelta_{h+1,l}={ 1}$.

{\bf Case 2. $l=h+1$. }
Since $l<k$, $(A+\lambda)_{h,k}=a_{h,k}$, $(\SE{A}+\lambda)_{h,k}=\SE{a}_{h,k}$, $(A+\lambda)_{h+1,k}=a_{h+1,k}$, and $(A+\lambda)_{h+1,l}=\lambda_{l}+\SO{a}_{h+1,l}$. Thus, in this case, { \eqref{club}} takes the form
\begin{equation*}
\begin{aligned}
(\clubsuit)_{h,k,l}&=(-1)^{\widetilde{a}^{\bar1}_{h,l}}~\cdot \up^{\tau_{h,k,l}(\Ad)+a^{\bar0}_{h,k}}\cdot~\Big\{
\up^{-a_{h,k}}[\SEE{a}_{h+1, k} +1]\\
&\big\{\sum_{\lambda}\up^{\lambda\centerdot(\varepsilon\ep_h+\varepsilon'\ep_{h+1})}([{\Ahkzm -E_{h+1,h+1}+\lambda| \SO{A}+E_{h+1,h+1}}]\\
&\hspace{6cm}-{[{\lambda}_{h+1}+1]_{\up^2}}[{\Ahkzm +E_{h+1,h+1}+\lambda| \SO{A}-E_{h+1,h+1}}])\big\}\\
+&\big\{\sum_{ \lambda }\up^{\lambda\centerdot(\varepsilon\ep_h+\varepsilon'\ep_{h+1})}([{{\SE{A} -E_{h+1,h+1}+\lambda |\Ahkom+E_{h+1,h+1}}}]\\
&\hspace{2cm}-{[{\lambda}_{h+1}+1]_{\up^2}}[{{\SE{A} +E_{h+1,h+1}+\lambda |\Ahkom-E_{h+1,h+1}}}])\big\}-(\up-\up^{-1})\left[{a_{h+1, k} +1}\atop 2\right]\\
	&\big\{\sum_{\lambda}\up^{\lambda\centerdot(\varepsilon\ep_h+\varepsilon'\ep_{h+1})}(
	[{\SE{A} + 2E_{h+1,k}-E_{h+1,h+1}+\lambda |\Ahkomm+E_{h+1,h+1}}]\\
&\hspace{4cm}-{[{\lambda}_{h+1}+1]_{\up^2}}[{\SE{A} + 2E_{h+1,k}+E_{h+1,h+1}+\lambda |\Ahkomm-E_{h+1,h+1}}])\big\}
\Big\},
\end{aligned}
\end{equation*}
where $\la$ runs over $\CMN(n,r-\snorm{A})$. {\color{black}Notice that all matrices appearing in above expression for $(\clubsuit)_{h,k,l}$ have the form of $(\lambda+M^{\bar{0}}|M^{\bar{1}})$ with $\Md=(M^{\bar{0}}|M^{\bar{1}})$ being the form of the matrices in \eqref{SFM} with $\ndelta_{h+1,l}=0$}.
Now, a similar argument to the proof of Proposition \ref{mulformdiag} (the $h=k$ case) shows that $(\clubsuit)_{h,k,l}$ has the required form.
This completes the proof of the claim and hence of \eqref{lower-com-formula-pf-3}.

On the other hand, by Corollary \ref{common_form}, we see that $Y_{\ol{h},r}\cdot\AJS(\Ad,{\color{black}\bs{0}},r)$ is a linear combination of $\Md(\bs{j},r)$ with $\Md\in\MNZ(n,r)$ and $\bs{j}\in\ZZ^n$. Further, by Proposition \ref{standard-phiupper1}(2), we may write $Y_{\ol{h},r}\cdot\AJS(\Ad,{\color{black}\bs{0}},r)=\Sigma_1+\Sigma_2$,
where $\Sigma_1$ (resp., $\Sigma_2$) is a linear combination of $\Md(\bs{j},r)$ with $M={A}^-_{h,k}$ (resp., $M\prec{A}^-_{h,k}$ for some $k$). {\color{black}Meanwhile,} since every term $\Bd(\bs{j},r)$ appearing in $\sdpCHf(h,\Ad({\bf0},r))$ and $\sdpCHHf(h,\Ad(\bs{0},r)$  satisfies $B={A}^-_{h,k}$ for some $k$ with $1\leq k\leq n$ {\color{black} by Proposition \ref{mulformodd2} and \eqref{lower-com-formula-pf-3}}, equating the coefficients with respect to the standard basis $\{[\Md]\}_\Md$ gives
$$\aligned
\Sigma_1&=\sdpCHf(h,\Ad({\bf0},r))+(\up-\up^{-1})\sdpCHHf(h,\Ad(\bs{0},r)),\\
\Sigma_2&=\sum_{\substack{ \\ \lambda \in  \CMN(n,r-\snorm{A}) }}\sum_{\Bd\in \MNZ(n,r)\atop \exists k, B\prec{A}^-_{h,k}}{ \scp_{{ F^\star_{\bar{h},\la+\ro(A)}},A^\star_\la,\Bd}}[{\Bd}].
\endaligned$$
This completes the proof of \eqref{lower-com-formula}.
\end{proof}

\section{A triangular relation and its associated monomial basis}\label{TriRelationQS}
 In this section, we shall establish a triangular relation (Theorem \ref{triang-QqS}) in the queer $\up$-Schur superalgebra (cf. \cite[Lem.~3.8]{BLM} for $\up$-Schur algebras and \cite[Th.~7.4]{DG} for $\up$-Schur superalgebras). This is equivalent to
 show that the transition matrix from a certain monomial basis to the standard basis is a triangular matrix.
  To achieve this, we need to determine the leading term in each step of the multiplications in a monomial basis element ${\bf m}^{(\Ad)}$ in Theorem \ref{triang-QqS-F}. There are three pairs of preparatory results: Propositions--Corollaries \ref{triang_upper-QqS}--\ref{InsertL}, \ref{middleTR}--\ref{triang_diag-QqS-DF}, and \ref{triang_lower-QqS}--\ref{InsertL2}, serving this purpose.

For non-negative integers {$k \le l$} and elements $\ft_i,i\in[k,l]$ in a ring, we make the following convention for the orders of products:
\begin{equation}\label{eq_product}
\begin{aligned}
 \prod_{i\in[k,l]}^<\ft_i=		\prod_{k \le i \le l}{\ft}_i:={\ft}_{k} {\ft}_{k+1} \cdots {\ft}_{l} , \quad
 \prod_{i\in[k,l]}^>\ft_i=		\prod_{l \ge  i \ge k} {\ft}_i := {\ft}_{l} {\ft}_{l-1} \cdots {\ft}_{k}.
\end{aligned}
\end{equation}

For $1\leq k\leq n, 1\leq h\leq n-1$, $p>0$, $\lambda\in\Lambda(n,r)$ and $A\in M_n(\mathbb{N})$, let

\begin{equation}\label{Ehpla}
\aligned
\Ed_{h,\lambda,p}&:=(\lambda + pE_{h, h+1}-pE_{h+1,h+1}|O),\\ \Fd_{h,\lambda,p}&:=(\lambda+ pE_{h+1,h}-pE_{h,h} |O),
 \endaligned
 \qquad\;
 \aligned A^{+,p}_{h,k}&:=A+pE_{h,k}-pE_{h+1,k},\\ A^{-,p}_{h,k}&:=A-pE_{h,k}+pE_{h+1,k}.\endaligned
 \end{equation}
 Then, in the notations in \eqref{EVEN} and \eqref{Ahk}, we actually have
$$
\begin{aligned}
&\Ed_{h,\lambda}= \Ed_{h,\lambda,1},\quad\; \Fd_{h,\lambda}=\Fd_{h,\lambda,1}, \quad\;
A^+_{h,k}= A^{+,1}_{h,k},\quad\;A^-_{h,k}= A^{-,1}_{h,k}.
\end{aligned}
$$ Thus, $ A^{+,p}_{h,k}$ (resp., $A^{-,p}_{h,k}$) is a {\it $p$-up}  (resp., {\it $p$-down}) matrix.

The multiplication formulas given in {  Lemma \ref{phiupper-even-norm}} are ``order 1'' standard multiplication formulas. The following lemma can be used to derive the higher order ones.

\begin{lem}\label{pEh}
 Fix $1\leq h\leq n-1$ and $p>0$. For any $\lambda\in\Lambda(n,r)$, we have
\begin{align}\label{divpEh}
[\Ed_{h,\lambda,p}]&=\frac1{[p]^!}\prod_{p>i\geq0}[\Ed_{h,\la+i\bsal_h}]=\frac1{[p]}[\Ed_{h,\la+(p-1)\bsal_h,1}][\Ed_{h,\la,p-1}]\qquad{\color{black}(\text{in the case }\lambda_{h+1}\geq p)},\\
[\Fd_{h,\lambda,p}]&=\frac1{[p]^!}\prod_{p>i\geq0}[\Fd_{h,\la+i\bsal_h}]=\frac1{[p]}[\Fd_{h,\la+(p-1)\bsal_h,1}][\Fd_{h,\la,p-1}]\qquad {\color{black}(\text{in the case }\lambda_{h}\geq p)} \label{divpFh}.
\end{align}
\end{lem}
\begin{proof}By { Lemma \ref{phiupper-even-norm}(2)}, one checks easily that, for any $1\leq j\leq p$,
$$\aligned
{[\lambda+j\alpha_h-E_{h+1,h+1}+E_{h,h+1}|O]}\cdot &[j]![\lambda-jE_{h+1,h+1}+jE_{h,h+1}|O]\\
=&[j+1]![\lambda-(j+1)E_{h+1,h+1}+(j+1)E_{h,h+1}|O].\endaligned$$
Thus, \eqref{divpEh} follows from induction.  \eqref{divpFh} can be proved similarly by {  Lemma \ref{phiupper-even-norm}(3)}.
\end{proof}

The following facts can be easily checked by the definition of the preorder $\preceq$ in \eqref{preorder}.
\begin{lem}\label{preorder-rem}
For any $A=(a_{i,j}), M=(m_{i,j})\in M_n(\mathbb{N})_r$, $1\leq h\leq n-1$,
and $j,k\in[1,n]$,
\begin{enumerate}
\item if $a_{h+1,j}a_{h+1,k}\neq0$ (resp. $a_{h,j}a_{h,k}\neq0$) and $j<k$, then $A^+_{h,j}\prec A^+_{h,k}$ (resp.,
 $A^-_{h,k}\prec A^-_{h,j}$).

\item if  $M\prec A$ and $m_{h+1,k}>0$ (resp., $m_{h,k}>0$), then $M^+_{h,k}\prec A^+_{h,k}$ (resp.,$M^-_{h,k}\prec A^-_{h,k}$).
 \end{enumerate}

\end{lem}

We extend the order $\preceq$ on the matrix set $M_n(\NN)$ defined in \eqref{preorder} to $M_n(\mathbb N|\mathbb N_2)$.
For $A^\star=(A^{\ol0}|A^{\ol1})\in \MNZN(n)$, let
\begin{equation}\label{Tr}
\Tr(A^{\ol1}) = \{ i \mid \SOE{a}_{i,i} = 1, 1\leq i\leq n \}
\end{equation} denote the {\it diagonal support} of $A^{\ol1}$.  We follow the definition given in \cite[\S6, (23)]{GLL}.

\begin{defn}\label{preorder-2}
For any $A^\star, B^\star\in\MNZN(n)$ with base matrix $A,B$, respectively, define $A^\star\prec^\star B^\star$ if ${A\prec B}$ or
${A\preeq B}$ but $\Tr(A^\star)\subset \Tr(B^\star).$
\end{defn}

\medskip
We first determine the leading terms relative to $\preceq$ for the products of the form $[\Bd]\cdot[\Ad]$, where $\Bd=\Ed_{h,\lambda,p}$ or $\Ed_{\ol h,\lambda}$ or something mixed, where $\Ad$ has an upper triangular base matrix  $A=(a_{i,j})$ satisfying further  $a_{i,j}=0$ for $i<j$ and $j>k$ for some fixed $1\leq k\leq n$, i.e,  $A$ has the $k^\vartriangle$-{\it shape}
\begin{equation}\label{eq:special matrix-upper}
A=\begin{pmatrix}T_k&0\\0&D_{n-k}
\end{pmatrix} \text{ where $T_k$ is } k\times k \text{ upper triangular matrix, and }D_{n-k}\text{ diagonal.}
\end{equation}

Parallelly, the leading term for $[\Cd]\cdot[\Ad]$, with $\Cd=\Fd_{h,\lambda,p}$ or $\Fd_{\ol h,\lambda}$ or something mixed,  will be discussed in Proposition \ref{triang_lower-QqS}  with $\Ad$ being of the form \eqref{eq:special matrix-lower}.

In the following results, for $\Md=(M^{\bar0}|M^{\bar1})$, the term $f_\Md[\Md]$ in the expression of the form
$$f_\Md[\Md]+(\text{lower terms})_\prec\;\;(\text{resp}., (\text{lower terms})_{\prec^\star})$$
is called the leading term, while``lower terms" relative to $\prec$ (resp., $\prec^\star$) stands for a linear combination of $[\Bd]$ with $\Bd\in\MNZ(n,r)$ whose base $ B\prec M^{\bar0}+M^{\bar1}$ (resp., with $\Bd\prec^\star\Md$).
\begin{prop}\label{triang_upper-QqS}
For $\Ad=  ({\SEE{a}_{i,j}} | {\SOE{a}_{i,j}})\in\MNZ(n,r)$ with base matrix $A=(a_{i,j})$, assume that $A$ has the $k^\vartriangle$-shape in \eqref{eq:special matrix-upper} for some
 $1\leq k\leq n$ and
 $\ro(A)=\lambda$.
Then the following holds in $\qSchvsQ$ for all $h< k$.

(1) If
{\color{black}$\SO{a}_{h+1,k}=0$ and $\SE{a}_{h+1,k}>0$, then, for all $p\in[1,\SE{a}_{h+1,k}]$, }
\begin{equation*}\label{pEhA}
\begin{aligned}
{[\Ed_{h,\lambda,p}]\cdot[\Ad]}=\left[{\SE{a}_{h,k}+p}\atop p\right][(\SE{A})^{+,p}_{h,k}|\SO{A}]+(\text{lower terms})_\prec.
\end{aligned}
\end{equation*}

(2) If $\SO{a}_{h,k}=\SO{a}_{h+1,k}=0$ and $\SE{a}_{h+1,k}>0$, then
\begin{equation*}\label{oddEhA}
{[\Ed_{\ol{h},\lambda}]\cdot[\Ad]}=(-1)^{\wp(\Ad)+\SO{\tilde{a}}_{h-1,k}}[\SE{A}-E_{h+1,k}|\SO{A}+E_{h,k}]+(\text{lower terms})_\prec.
\end{equation*}

(3) If $\SO{a}_{k,k}=0$, $\SE{a}_{k,k}>0$ and assume, in addition, $a_{i,k}=0$, for some $1\leq l<k\leq n$ and all $l\leq i<k$. Then,
for any $p=p^{\bar{0}}+p^{\bar{1}}\in[1,\SE{a}_{k,k}]$ with $p^{\ol0}\in\NN$, $p^{\bar{1}}\in\{0,1\}$, we have
\begin{equation*}
\begin{aligned}
\Big([\Ed_{\bar l,\lambda+\SE{p}\alpha_{l,k}+\SO{p}\alpha_{l+1,k}}]^{\SO{p}}&\cdot[\Ed_{l,\lambda+p\alpha_{l+1,k},\SE{p}}]
\prod_{l<h<k}[\Ed_{h,\lambda+p\alpha_{h+1,k},p}]\Big)\cdot[\Ad]\\
&=[\SE{A}-pE_{k,k}+\SE{p}E_{l,k}|\SO{A}+\SO{p}E_{l,k}]+(\text{lower terms})_\prec,
\end{aligned}
\end{equation*}
where $\alpha_{i,j}=\epsilon_i-\epsilon_j$, for $i\neq j$, and $\alpha_{i,i}=0$.
\end{prop}

\begin{proof}
(1)  We first prove (1) by induction on $p$.
If $p=1$, by Lemma \ref{phiupper-even-norm}(2),  the shape of $A$ together with  $\SO{a}_{h+1,k}=0$ ({ $h+1\leq k$}) implies that the terms $[\Bd]$ with base $B=A^+_{h,k}$ appearing in the RHS of Lemma \ref{phiupper-even-norm}(2) is unique. Thus, there exist $\scc_\Bd^{\lambda,1}\in\sZ$ such that 
\begin{equation}\label{1EhA}
\begin{aligned}
{[\Ed_{h,\lambda,1}][A^\star]}&=[\SE{a}_{h,k}+1][(\SE{A})^{+,1}_{h,k}|\SO{A}]+\sum_{\substack{\Bd\in\MNZ(n,r), B\approx A^+_{h,j},j<k}}\scc_\Bd^{\lambda,1}[\Bd]\\
&=[\SE{a}_{h,k}+1][(\SE{A})^{+,1}_{h,k}|\SO{A}]
+\sum_{\substack{\Bd\in\MNZ(n,r), B\prec A^+_{h,k}}}\scc_\Bd^{\lambda,1}[\Bd],
\end{aligned}
\end{equation}
 where the second equality follows from Lemma \ref{preorder-rem}(1). Thus, (1) holds for $p=1$.

 Now assume $p>1$ and (1) holds for  $p-1$. Then, by induction,
\begin{equation}\label{triang_upper-QqS-pf-2-2}
\begin{aligned}
{[\Ed_{h,\lambda,p-1}][\Ad]}=\left[{\SE{a}_{h,k}+p-1}\atop p-1\right][(\SE{A})^{+,p-1}_{h,k}|\SO{A}]
+\sum_{\substack{\Md\in\MNZ(n,r)\\ M\prec A^{+,p-1}_{h,k}}}\scc_{\Md}^{\lambda,p-1}[\Md].\\
\end{aligned}
\end{equation}
The leading matrix $P^\star:=((\SE{A})^{+,p-1}_{h,k}|\SO{A})$ also satisfies $\SO{a}_{h+1,k}=0$ and its base matrix  $P=A^{+,p-1}_{h,k}$ has the same shape \eqref{eq:special matrix-upper} as $A$. Thus, for $\la^{(p-1)}:=\ro(A^{+,p-1}_{h,k})
=\lambda+(p-1)\alpha_h$, by applying \eqref{1EhA} to $P^\star$,
we have
\begin{equation*}\label{triang_upper-QqS-pf-2-3}
\begin{aligned}{[E^\star_{h,\la^{(p-1)},1}][P^\star]}=
[E^\star_{h,\la^{(p-1)},1}][(\SE{A})^{+,p-1}_{h,k}|\SO{A}]
={[\SE{a}_{h,k}+p]}[(\SE{A})^{+,p}_{h,k}|\SO{A}]
+\sum_{\substack{\Bd\in\MNZ(n,r)\\B\prec A^{+,p}_{h,k}}}\scc^{\lambda^{(p-1)},1}_{\Bd}[\Bd].
\end{aligned}
\end{equation*}
Multiplying the both sides of \eqref{triang_upper-QqS-pf-2-2} by $\frac1{[p]}[E^\star_{h,\la^{(p-1)},1}]$ and noting
$\left[{\SE{a}_{h,k}+p-1}\atop p-1\right]\frac{[\SE{a}_{h,k}+p]}{[p]}=\left[{\SE{a}_{h,k}+p}\atop p\right]$, we obtain by \eqref{divpEh}
\begin{equation}
\begin{aligned}
{[\Ed_{h,\lambda,p}][\Ad]}=\left[{\SE{a}_{h,k}+p}\atop p\right][\SE{A}&-pE_{h+1,k}+pE_{h,k}|\SO{A}]+\sum_{\substack{\Bd\in\MNZ(n,r)\\B\prec A^{+,p}_{h,k}}}{ {\scc'}^{\la^{(p-1)},1}_{\Bd}}[\Bd]\\
&+
\sum_{\substack{\Md\in\MNZ(n,r)\\ M\prec A^{+,p-1}_{h,k}}}\frac{1}{[p]}\scc^{\lambda,p-1}_{\Md}[E^\star_{h,\la^{(p-1)},1}][\Md]
\end{aligned}
\end{equation}
where  ${\scc'}^{\lambda^{(p-1)},1}_{\Bd}=\frac{1}{[p]}\left[{\SE{a}_{h,k}+p-1}\atop p-1\right]\scc^{\lambda^{(p-1)},1}_{\Bd}\in \mathbb{Q}(\up)$. Now, (1) follows from another application of Lemma \ref{phiupper-even-norm}(2) which tells that, for every $M\prec A^{+,p-1}_{h,k}$,
$[E^\star_{h,\la^{(p-1)},1}][\Md]$ is a linear combination of some $[\Bd]$ with $B\prec M^{+}_{h,k}\prec(A^{+,p-1}_{h,k})^+_{h,k}=A^{+,p}_{h,k}$ (Lemma \ref{preorder-rem}(2)).

For later use in this proof, we remark that the last order relation can be generalised to higher orders by a similar induction via Lemma \ref{pEh}.

\begin{cor}[{\bf Part 1}]\label{InsertL}
Maintain the assumptions on  $\Ad$ in Proposition \ref{triang_upper-QqS}(1) and assume $\Md\in\MNZ(n,r)$ with $M\prec A$.
Then in $\qSchvsQ$ we have
\begin{equation}\label{EhP}
{[\Ed_{h,\ro(M),p}][M^\star]}=\text{lin. comb. of }[\Bd], B\prec A^{+,p}_{h,k}.
\end{equation}
\end{cor}

(2) We now prove part (2) of the proposition. If $\Ad$ satisfies $\SO{a}_{h,k}=\SO{a}_{h+1,k}=0$ and $\SE{a}_{h+1,k}>0$ ($h+1\leq k$) and has shape
\eqref{eq:special matrix-upper} then, by Lemma \ref{standard-phiupper1} (1),
\begin{equation*}
\begin{aligned}
{[\Ed_{\bar h,\lambda}][\Ad]}&=(-1)^{p(\Ad)+\SO{\tilde{a}}_{h-1,k}}[\SE{A}-E_{h+1,k}|\SO{A}+E_{h,k}]+\sum_{\substack{\Bd\in\MNZ(n,r)\\B\approx A^{+}_{h,j},j<k}}\ol{\scc}_{\Bd}[\Bd]+\sum_{\substack{\Bd\in\MNZ(n,r)\\B\prec A^{+}_{h,j},j\leq k}}  \ol{\scc}_\Bd[\Bd]\\
&=(-1)^{p(\Ad)+\SO{\tilde{a}}_{h-1,k}}[\SE{A}-E_{h+1,k}|\SO{A}+E_{h,k}]+\sum_{\substack{\Bd\in\MNZ(n,r)\\B\prec A^{+}_{h,k}}}\ol{\scc}_{\Bd}[\Bd],\;\;(\ol{\scc}_\Bd\in\sZ)
\end{aligned}
\end{equation*}
since $B\prec A^{+}_{h,j}\prec A^{+}_{h,k} $ for $j<k$ by Lemma \ref{preorder-rem}. This proves (2).

  Here is the second part of Corollary \ref{InsertL}.

 \medskip
\noindent
{\bf Corollary \ref{InsertL} (Part 2).} {\it Maintain the assumption on  $\Ad$ in Proposition \ref{triang_upper-QqS}(2). Then, for
$\Md\in\MNZ(n,r)$ with $M\prec A$, we have in $\qSchvsQ$}
\begin{equation}\label{oddEhP}
{[\Ed_{\bar h,\ro(M)}][M^\star]}=\text{lin. comb. of }[\Bd], B\prec A^{+}_{h,k}.
\end{equation}

(3) Finally, we prove (3).
For  $1\leq l<k$, we shall firstly show by downward induction on $l$ that the following holds,  for some ${\mathfrak{f}}_{\Cd}\in\mathbb{Q}(\up)$,
\begin{equation}\label{upper-QqS-cor-pf-11}
\prod_{l<h<k}[\Ed_{h,\lambda+p\alpha_{h+1,k},p}][\Ad]=[\SE{A}+pE_{l+1,k}-pE_{k,k}|\SO{A}]+\sum_{\substack{\Cd\in\MNZ(n,r)\\C\prec A-pE_{k,k}+pE_{l+1,k}}}{\mathfrak{{f}}}_{\Cd}[\Cd].
\end{equation}
 Clearly \eqref{upper-QqS-cor-pf-11} is trivially true in the case $l=k-1$ (so, $l+1=k$), taking all ${\mathfrak{f}}_{\Cd}=0$.  In the case $l=k-2$, since $\SO{a}_{k,k}=0={a}_{k-1,k}$ (additional assumption), and $a_{k,k}=a^{\bar{0}}_{k,k}\geq p$.
Applying Proposition \ref{triang_upper-QqS}(1) yields
\begin{equation}\label{upper-QqS-cor-pf-1}
{[\Ed_{k-1,\lambda,p}][\Ad]}=[\SE{A}+pE_{k-1,k}-pE_{k,k}|\SO{A}]+\sum_{\substack{\Bd\in\MNZ(n,r)\\B\prec A+pE_{k-1,k}-pE_{k,k}}}\mathfrak{f}^{k-1,p}_{\Bd}[\Bd],
\end{equation}
proving the case for $l=k-2$.

Assume now $l<k-2$ and (3) is true for $l+1$:
\begin{equation}\label{upper-QqS-cor-pf-5}
\begin{aligned}
{\prod_{l+1<h<k}[\Ed_{h,\lambda+p\alpha_{h+1,k},p}][\Ad]}=[\SE{A}+pE_{l+2,k}-pE_{k,k}|\SO{A}]+\sum_{\substack{\Md\in\MNZ(n,r)\\M\prec A-pE_{k,k}+pE_{l+2,k}}}{\mathfrak{f}'}_{\Md}[\Md].
\end{aligned}
\end{equation}
Multiplying the both sides of \eqref{upper-QqS-cor-pf-5} by $[\Ed_{l+1,\lambda+p\alpha_{l+2,k},p}]$ gives the LHS of \eqref{upper-QqS-cor-pf-11} and the right hand side $[\Ed_{l+1,\gamma,p}][\SE{A}+pE_{l+2,k}-pE_{k,k}|\SO{A}]+$(a linear comb. of $[\Ed_{l+1,\gamma,p}][\Md]$), where $\gamma:=\ro(A-pE_{k,k}+pE_{l+2,k})=\lambda+p\alpha_{l+2,k}$.

Observe that the matrix $(\SE{A}+pE_{l+2,k}-pE_{k,k}|\SO{A})$ in \eqref{upper-QqS-cor-pf-5} satisfies { $\SE{a}_{l+1,k}=\SO{a}_{l+1,k}=\SO{a}_{l+2,k}=0, \SE{a}_{l+2,k}>0$} and moreover the base matrix
$A+pE_{l+2,k}-pE_{k,k}$ is of the form as \eqref{eq:special matrix-upper}. So, applying Proposition \ref{triang_upper-QqS}(1) gives rise to
\begin{equation*}\label{upper-QqS-cor-pf-6}
\begin{aligned}
{[\Ed_{l+1,\gamma,p}][\SE{A}+pE_{l+2,k}-pE_{k,k}|\SO{A}]}=[\SE{A}+pE_{l+1,k}-pE_{k,k}|\SO{A}]+\sum_{\substack{P^\star\in\MNZ(n,r)\\ P\prec A+pE_{l+1,k}-pE_{k,k}}}{\mathfrak{f}}^{l+1,p}_{P^\star}[P^\star],
\end{aligned}
\end{equation*}
for some ${\mathfrak{f}}^{l+1,p}_{P^\star}\in\mathbb{Q}(\up)$.  On the other hand, for $\Md$ with $M\prec {}'\!A:=A+pE_{l+2,k}-pE_{k,k}$ in \eqref{upper-QqS-cor-pf-5}, we have $\ro(M)=\gamma$ and, by
\eqref{EhP} and noting ${}'\!A^{+,p}_{l+1,k}=A+pE_{l+1,k}-pE_{k,k}$,
\begin{equation*}\label{upper-QqS-cor-pf-7}
{[\Ed_{l+1,\gamma,p}]{ [\Md]}}=\sum_{\substack{N^\star\in\MNZ(n,r)\\ N\prec A+pE_{l+1,k}-pE_{k,k}}}{\mathfrak{f}'}^{l+1,p}_{N^\star}[N^\star].
\end{equation*}
Hence, \eqref{upper-QqS-cor-pf-11} follows.

To complete the proof of (3), we now multiply the RHS of \eqref{upper-QqS-cor-pf-11} by the remaining { factor  $[\Ed_{\bar l,\lambda+\SE{p}\alpha_{l,k}+\SO{p}\alpha_{l+1,k}}]^{\SO{p}}[\Ed_{l,\lambda+p\alpha_{l+1,k},\SE{p}}]$. We shall first consider the multiplication with the leading term in RHS of \eqref{upper-QqS-cor-pf-11}. }
Observe that the leading matrix $(\SE{A}-pE_{k,k}+pE_{l+1,k}|\SO{A})$ in  \eqref{upper-QqS-cor-pf-11} is of the form \eqref{eq:special matrix-upper} with $a^{\bar 1}_{l,k}=a^{\bar 1}_{l+1,k}=0$ and
$a^{\bar 0}_{l+1,k}=p\geq p^{\bar 0}, a^{\bar 0}_{l,k}=0$. Moreover, its base matrix 
is $A-pE_{k,k}+pE_{l+1,k}$ with $\ro(A-pE_{k,k}+pE_{l+1,k})=\lambda+p\alpha_{l+1,k}$.
Then applying (1) yields
\begin{equation}\label{upper-QqS-cor-pf-1}
\begin{aligned}
{[\Ed_{l,\lambda+p\alpha_{l+1,k},{\SE{p}}}]}&[\SE{A}-pE_{k,k}+pE_{l+1,k}|\SO{A}]=[\SE{A}-pE_{k,k}+\SO{p}E_{l+1,k}+\SE{p}E_{l,k}|\SO{A}]\\
&+\sum_{\substack{\Cd\in\MNZ(n,r)\\C\prec A-pE_{k,k}+\SO{p}E_{l+1,k}+\SE{p}E_{l,k}}}\mathfrak{f}^{{l,{\SE{p}}}}_{\Cd}[\Cd].
\end{aligned}
\end{equation}

If $\SO{p}=0$, then $p=\SE{p}$ and {  the multiplication of $[\Ed_{\bar l,\lambda+\SE{p}\alpha_{l,k}+\SO{p}\alpha_{l+1,k}}]^{\SO{p}}[\Ed_{l,\lambda+p\alpha_{l+1,k},\SE{p}}]$ with the leading term in RHS of \eqref{upper-QqS-cor-pf-11} is completed. For each $[C^\star]$ appearing in the lower terms in RHS of \eqref{upper-QqS-cor-pf-11}, using \eqref{EhP} in Corollary \ref{InsertL} gives
\begin{equation}\label{upper-QqS-cor-pf-lower}
[\Ed_{l,\lambda+p\alpha_{l+1,k},\SE{p}}][\Cd]=\sum_{\substack{D^\star\in\MNZ(n,r)\\ D\prec A-pE_{k,k}+\SO{p}E_{l+1,k}+\SE{p}E_{l,k}}} {\mathfrak{f}'}^{{l,{\SE{p}}}}_{\Dd}[D^\star].
\end{equation}
 This together with  \eqref{upper-QqS-cor-pf-1} proves (3) in this situation.}

If $\SO{p}=1$, then we need further to compute the multiplication of  the {\color{black}RHS} of \eqref{upper-QqS-cor-pf-1} by $[\Ed_{\bar l,\lambda+\SE{p}\alpha_{l,k}+\SO{p}\alpha_{l+1,k}}]$. Thus, we apply (2) for the leading term and \eqref{oddEhP} for the lower terms.

By the additional assumption on matrix $A$, the parity value $\parity{\Ad}=\SO{\tilde{a}}_{l,k}$. Hence, for $P^\star:=(\SE{A}-pE_{k,k}+\SO{p}E_{l+1,k}+\SE{p}E_{l,k}|\SO{A})$ with base matrix $P$, we have
$\parity{P^\star}=\parity{\Ad}=\SO{\tilde{a}}_{l,k}$ and $\ro(P)=\lambda+\SE{p}\alpha_{l,k}+\SO{p}\alpha_{l+1,k}$. Then, by (2), we have
\begin{equation}\label{upper-QqS-cor-pf-3}
\begin{aligned}
{[\Ed_{\bar l,\ro(P)}][P^\star]}
=[\SE{A}-{p}E_{k,k}+\SE{p}E_{l,k}|\SO{A}+\SO{p}E_{l,k}]+\sum_{\substack{\Bd\in\MNZ(n,r)\\B\prec A-pE_{k,k}+pE_{l,k}}}{\mathfrak{f}'}^{\overline{l}}_{\Bd}[\Bd].
\end{aligned}
\end{equation}

For any $\Cd$ in \eqref{upper-QqS-cor-pf-1}, $C\prec P$ and
$P^+_{l,k}=A-pE_{k,k}+\SO{p}E_{l+1,k}+\SE{p}E_{l,k}+E_{l,k}-E_{l+1,k}=A-pE_{k,k}+pE_{l,k}$.
Thus, by \eqref{oddEhP},
\begin{equation}\label{upper-QqS-cor-pf-4}
\begin{aligned}
{[\Ed_{\overline{l},\ro(P)}][\Cd]}=\sum_{\substack{\Bd\in\MNZ(n,r)\\B\prec A-pE_{k,k}+pE_{l,k}}}{\mathfrak{\tilde{f}}}^{\overline{l}}_{\Bd}[\Bd].
\end{aligned}
\end{equation}
{  For each $[C^\star]$ appearing in the lower terms in RHS of \eqref{upper-QqS-cor-pf-11} in this case, using Corollary \ref{InsertL}(1)--(2) one can obtain a formula similar to \eqref{upper-QqS-cor-pf-lower}. Then}
combining \eqref{upper-QqS-cor-pf-1}, \eqref{upper-QqS-cor-pf-3} and \eqref{upper-QqS-cor-pf-4} proves (3) in this situation.
\end{proof}

We now determine the leading term of the products $[\Dd_{\bar h, \ro(A)}][A^{\bar0}|A^{\bar1}]$, where both $A^{\bar0}$ and $A^{\bar1}$ are upper triangular.

\begin{prop}\label{middleTR}
Suppose $\Ad=  ({\SEE{a}_{i,j}} | {\SOE{a}_{i,j}})\in\MNZ(n,r)$ has an upper triangular base $A$ (i.e., ${\SEE{a}_{i,j}}+{\SOE{a}_{i,j}}=0\;\forall i>j$) and assume $\SE{a}_{h,h}>0$ and $\SO{a}_{h,h}=0$
for some $h\in[1,n]$. Then the following holds in $\qSchvsQ$:
\begin{equation}\label{triang_diag-QqS-D}
\begin{aligned}
{{[{\Dd_{\bar h,\lambda}}]} [\Ad]}={(-1)}^{\wp(\Ad)+{\widetilde{a}}^{\ol{1}}_{h,h} } {\up}^{d_{h}(A,h) }&
	[{\SE{A} -E_{h,h}|\SO{A}+ E_{h,h}}]+(\text{lower terms})_{\prec^\star}.
\end{aligned}
\end{equation}
In particular, for $h=n$, we have
\begin{equation}\label{triang_diag-QqS-D-n}
{{[{\Dd_{\bar n,\lambda}}]} [\Ad]}=[{\SE{A} -E_{n,n}|\SO{A}+ E_{n,n}}].
\end{equation}
\end{prop}
\begin{proof} By the hypothesis,  $A^\star_h:=({\SE{A} -E_{h,h}|\SO{A}+ E_{h,h}})\in\MNZ(n,r)$ and its base is $A$. On the other hand, $({\SE{A} +E_{h,h}|\SO{A}- E_{h,h}})\notin\MNZ(n,r)$ which means $[{\SE{A} +E_{h,h}|\SO{A}- E_{h,h}}]=0$ by the convention right after \eqref{SMF}.
Thus, the hypothesis and Lemma \ref{normlized-phidiag1}(2) imply
\begin{equation*}\label{triang_diag-QqS-D-1}
\begin{aligned}
{[{\Dd_{\bar h,\lambda}}][\Ad]}
&={(-1)}^{\wp(\Ad)+{\widetilde{a}}^{\ol{1}}_{h,h} } {\up}^{d_{h}(A,h) }
	[{\SE{A} -E_{h,h}|\SO{A}+ E_{h,h}}]\\
&+\sum_{ k>h}
{(-1)}^{\wp(\Ad)+{\widetilde{a}}^{\ol{1}}_{h,k} } {\up}^{d_{h}(A,k) }
\Big\{
	[{\SE{A} -E_{h,k}|\SO{A}+ E_{h,k}}]
	+[a_{h,k}]_{\up^2}
	[{\SE{A} +  E_{h,k}| \SO{A} - E_{h,k}}]\Big\}\\
&+\sum_{\Bd,{B} \prec { {A} }} { \scp_{{ {\Dd_{\bar h,\lambda}}},\Ad,\Bd}}  [{\Bd}].
\end{aligned}
\end{equation*}
If $h=n$, then $A$ satisfies the SDP condition. Thus, both summations do not occur (one may apply Lemma \ref{normlized-phidiag1}(1)
directly), proving \eqref{triang_diag-QqS-D-n}.

It remains to prove that $A^\star_h$ is a leading term when $h<n$.
For those $\Bd$ with $B\prec A$, by Definition \ref{preorder-2},  we have
$\Bd\prec^\star A^\star_h$.  If (at most) one of the matrices $({\SE{A} -E_{h,k}|\SO{A}+ E_{h,k}})$ and $({\SE{A} +  E_{h,k}| \SO{A} - E_{h,k}})$, for $k>h$, is in $\MNZ(n,r)$, call it $A^\star_k$. We have, by \eqref{Tr},
\begin{equation*}
\Tr(A^\star_k)=\Tr(\Ad)\subset\Tr(\Ad)\cup\{h\}=\Tr(A^\star_h).
\end{equation*}
Hence, by Definition \ref{preorder-2}, $A^\star_k\prec^\star A^\star_h$. This proves \eqref{triang_diag-QqS-D}.
\end{proof}
With a similar argument, we also have the following.
\begin{cor}\label{triang_diag-QqS-DF}Maintain the assumption on $\Ad$ and $h$ as in Proposition \ref{middleTR}.
For any $\Md\in\MNZ(n,r)$
with $M^\star\prec^\star \Ad$, we have in $\qSchvsQ$
\begin{equation*}\label{triang_diag-QqS-F}
[\Dd_{\bar h,\ro(M)}][\Md]=\text{lin. comb. of }[\Bd], \Bd\prec^\star(\SE{A}-E_{h,h}|\SO{A}+E_{h,h}).
\end{equation*}
\end{cor}

Finally, we determine the leading term of the products $[\Cd]\cdot[\Ad]$, where $\Cd=\Fd_{h,\lambda,p}$ or $\Fd_{\ol h,\lambda}$ or something mixed and the base matrix $A$ satisfies  $a_{i,j}=0$ for $i>j$ and $j<k$ for certain $1\leq k\leq n$, i.e $A$ has the $k_\triangledown$-shape:
\begin{equation}\label{eq:special matrix-lower}
A=\begin{pmatrix}T_{k-1}&B\\0&C
\end{pmatrix} \text{ where $T_{k-1}$ is upper triangular and $B,C$ arbitrary}.
\end{equation}

Recall the $^-$-matrix (of higher order) $A^{-,p}_{h,k}=A-pE_{h,k}+pE_{h+1,k}$ defined in \eqref{Ehpla}. Note that $A^{-,p}_{h,k}$ is the base matrix of $((\SE{A})^{-,p}_{h,k}|\SO{A})$.
\begin{prop}\label{triang_lower-QqS}
Assume $\Ad=  ({\SEE{a}_{i,j}} | {\SOE{a}_{i,j}})\in\MNZ(n,r)$ and its base matrix $A=(a_{i,j})$ satisfies $a_{i,j}=0$ for $i>j$and $j<k$ (i.e., $A$ has the shape in \eqref{eq:special matrix-lower}), for some $k\in[1, n]$. Then, putting  $\ro(A)=\lambda$, the following holds in $\qSchvsQ$
for any given $h\in[1,n]$ satisfying $h\geq k$.

(1)  If $\SO{a}_{h,k}=0$ in the case $h>k$ (no requirement in the case $h=k$)  and $a_{h,k}^{\bar{0}}>0$ then
\begin{equation*}\label{pFhA}
\begin{aligned}
{[\Fd_{h,\lambda,p}][\Ad]}=\up^{-p\SO{a}_{h,k}}\left[{\SE{a}_{h+1,k}+p}\atop p\right][(\SE{A})^{-,p}_{h,k}|\SO{A}]+(\text{lower terms})_{\prec^\star}.
\end{aligned}
\end{equation*}

(2) If $\SO{a}_{h,k}=0$ in case $h>k$, $\SO{a}_{h+1,k}=0$ {\color{black}and $\SE{a}_{h,k}>0$}, then
\begin{equation*}\label{oddFhA}
{[\Fd_{\bar h,\lambda}][\Ad]}=(-1)^{p(\Ad)+{\SO{\tilde{a}}_{h,k}}}\up^{-\SO{a}_{h,k}}[\SE{A}-E_{h,k}|\SO{A}+E_{h+1,k}]+(\text{lower terms})_{\prec^\star}.
\end{equation*}

(3)  If $\SE{a}_{k,k}>0$ and, in addition, assume $a_{i,k}=0$ for all $i$ with $k<i\leq l+1$ for some $l$ with $k\leq l<n$.
Then, for any $p=\SE{p}+\SO{p}\in[1,\SE{a}_{k,k}]$ with $\SE{p}\in\NN$ and $\SO{p}\in\{0,1\}$, we have
\begin{equation*}
\begin{aligned}
{[\Fd_{\bar l,\lambda+\SE{p}\alpha_{l+1,k}+\alpha_{l,k}}]}^{\SO{p}}&[\Fd_{l,\lambda+p\alpha_{l,k},\SE{p}}]
\prod_{l>h\geq k}[\Fd_{h,\lambda+p\alpha_{h,k},p}][\Ad]=(-1)^{p^{\bar1}(\parity{\Ad}+\SO{\tilde{a}}_{k,k})}\up^{-{p}\SO{a}_{k,k}}\\
&\cdot[\SE{A}-pE_{k,k}+\SE{p}E_{l+1,k}|\SO{A}+\SO{p}E_{l+1,k}]+
(\text{lower terms})_{\prec^\star}.
\end{aligned}
\end{equation*}
\end{prop}

\begin{proof}
  (1) We first prove (1)
by induction on $p$ via \eqref{divpFh}. We need to consider two cases: (i) $h>k$ and $\SO{a}_{k,k}=0$ and (ii) $h=k$ and $\SO{a}_{k,k}\not=0$.
The proof of {  Case (i)} is entirely similar to the proof of Proposition \ref{triang_upper-QqS}(1) with the $E$-notations replaced by the $F$-notations, noting that the hypothesis implies $\SO{a}_{h,k}=0$, for all $h\geq k$.

We now assume $h=k$ and $\SO{a}_{k,k}=1$. We again apply induction on $p$. If $p=1$, applying Lemma  \ref{phiupper-even-norm}(3)
\begin{equation}\label{triang_lower-QqS-pf-11}
\begin{aligned}
{[\Fd_{k,\lambda,1}][\Ad]}&=\up^{-\SO{a}_{k,k}}[\SE{a}_{k+1,k}+1][\SE{A}-E_{k,k}+E_{k+1,k}|\SO{A}]+\up^{\SE{a}_{k,k}}[\SE{A}|\SO{A}-E_{k,k}+E_{k+1,k}]\\
\\
&-(\up-\up^{-1})\up^{\SE{a}_{k,k}}\left[a_{k+1,k}+1\atop 2\right][\SE{A}+2E_{k+1,k}|\SO{A}-E_{k,k}-E_{k+1,k}]\}\\
&+\sum_{\substack{\Bd\in\MNZ(n,r)\\B\approx A^-_{k,j},j>k}}{\mathfrak{f}^{{k,1}}_{\Bd}}[\Bd]\;\;\big(\mathfrak{f}^{{k,1}}_{\Bd}\in\mathbb{Q}(\up)\big)
\end{aligned}
\end{equation}
has three terms with the same leading base matrix $A^-_{k,k}$. Since
 \begin{equation*}
 \begin{aligned}
 &\Tr(\SE{A}|\SO{A}-E_{k,k}+E_{k+1,k})=\Tr(\SE{A}+2E_{k+1,k}|\SO{A}-E_{k,k}-E_{k+1,k})=\Tr(\Ad)\setminus \{k\}\\
 &\Tr(\SE{A}-E_{k,k}+E_{k+1,k}|\SO{A})=\Tr(\Ad),
 \end{aligned}
 \end{equation*}
 it follows from Definition \ref{preorder-2} that
$$
 \begin{aligned}
 (\SE{A}|\SO{A}-E_{k,k}+E_{k+1,k})\prec^\star(\SE{A}-E_{k,k}+E_{k+1,k}|\SO{A}),\\
 (\SE{A}+2E_{k+1,k}|\SO{A}-E_{k,k}-E_{k+1,k})\prec^\star(\SE{A}-E_{k,k}+E_{k+1,k}|\SO{A}).
 \end{aligned}
 $$
 Hence \eqref{triang_lower-QqS-pf-11} can be rewritten as
 \begin{equation*}\label{triang_lower-QqS-pf-12}
\begin{aligned}
{[\Fd_{k,\lambda,1}][\Ad]}&=\up^{-\SO{a}_{k,k}}[\SE{a}_{k+1,k}+1][\SE{A}-E_{k,k}+E_{k+1,k}|\SO{A}]+\sum_{\substack{\Bd\in\MNZ(n,r)\\\Bd\prec^\star (\SE{A}-E_{k,k}+E_{k+1,k}|\SO{A})}}\mathfrak{f}^{{k,1}}_{\Bd}[\Bd],
\end{aligned}
\end{equation*}
proving the inductive base.

Assume now $p>1$ and (1) is correct for $p-1$ (under the conditions $h=k$ and $\SO{a}_{k,k}=1$):
 \begin{equation}\label{triang_lower-QqS-pf-13}
\begin{aligned}
{[\Fd_{k,\lambda,p-1}][\Ad]}=&\up^{-(p-1)\SO{a}_{k,k}}\left[{\SE{a}_{k+1,k}+p-1}\atop {p-1}\right][\SE{A}-(p-1)E_{k,k}+(p-1)E_{k+1,k}|\SO{A}]\\
&+\sum_{\substack{\Cd\in\MNZ(n,r)\\\Cd\prec^\star (\SE{A}-(p-1)E_{k,k}+(p-1)E_{k+1,k}|\SO{A})}}{\mathfrak{f}^{{k,p-1}}_{\Cd}}[\Cd].
\end{aligned}
\end{equation}

Observe that the base matrix of $P^\star:=(\SE{A}-(p-1)E_{k,k}+(p-1)E_{k+1,k}|\SO{A})$ appearing in \eqref{triang_lower-QqS-pf-13} has the form as \eqref{eq:special matrix-lower} and its row vector is $\mu:=\lambda-(p-1)\alpha_{k}$. Then we can apply the proof for the $p=1$ case to obtain
\begin{equation}\label{triang_lower-QqS-pf-14}
\begin{aligned}
&
[\Fd_{k,\mu,1}][P^\star]=\up^{-\SO{a}_{k,k}}[\SE{a}_{k+1,k}+p][\SE{A}-pE_{k,k}+pE_{k+1,k}|\SO{A}]+\sum_{\substack{\Bd\in\MNZ(n,r)\\\Bd\prec^\star { (\SE{A}-pE_{k,k}+pE_{k+1,k}|\SO{A})}}}{ {\fkf'}^{{k,1}}_{B^\star}}[\Bd].
\end{aligned}
\end{equation}

For any lower term $[\Cd]$ appearing in {  \eqref{triang_lower-QqS-pf-13}}, since $\Cd\prec (\SE{A}-(p-1)E_{k,k}+(p-1)E_{k+1,k}|\SO{A})$,
it follows by Lemma  \ref{phiupper-even-norm}(3) that
\begin{equation}\label{triang_lower-QqS-pf-15}
\begin{aligned}{[\Fd_{k,\mu,1}][\Cd]}=\sum_{\substack{\Bd\in\MNZ(n,r)\\\Bd\prec^\star (\SE{A}-pE_{k,k}+pE_{k+1,k}|\SO{A})}}{ {\fkf''}^{{k,1}}_{B^\star}}[\Bd].
\end{aligned}
\end{equation}
Combining  \eqref{triang_lower-QqS-pf-13}, \eqref{triang_lower-QqS-pf-14} and \eqref{triang_lower-QqS-pf-15} proves (1) (in case (ii)).

(2) The proof of (2) is a bit more complicated, comparing with Proposition \ref{triang_upper-QqS} (2).
According to the assumption, the matrix $\Ad$ satisfies $a_{i,j}=0$ for $i>j$ and  $j<k$ and moreover, $\SO{a}_{h,k}=0$, for $h>k$, $\SO{a}_{h+1,k}=0$, and $\SE{a}_{h,k}>0$. Then, by Lemma \ref{standard-phiupper1}(2), there exist  some $\mathfrak{f}^{{\bar h}}_{A^{\prime\star}}\in\mathbb{Q}(\up)$,
\begin{equation}\label{triang_lower-QqS-pf-16}
\begin{aligned}
&\quad\,[\Fd_{\bar h,\lambda}][\Ad]=(-1)^{\wp(\Ad)+\SO{\tilde{a}}_{h,k}}\up^{-\SO{a}_{h,k}}[\SE{A}-E_{h,k}|\SO{A}+E_{h+1,k}]\\
&+{ \delta_{h,k}}(-1)^{\wp(\Ad)+\SO{\tilde{a}}_{h,k}} \up^{\SE{a}_{h,k}}[{\SE{a}_{h+1,k}+1}]
		[{\SE{A} + E_{h+1, k} |\SO{A} - E_{h,k}}]+\sum_{\substack{A^{\prime\star}\in\MNZ(n,r)\\ A'\approx A^-_{h,j},{\color{black}j>k}}}{\mathfrak{f}^{{\bar h}}_{A^{\prime\star}}}[A^{\prime\star}]\\
		&+(\up-\up^{-1})\text{\rm HH}\overline{\textsc f}(h,\Ad)
+\sum_{\substack{\Bd\in\MNZ(n,r)\\ B\prec A^-_{h,j}, j>k}}{ \scp_{\Fd_{\bar h,\lambda},\Ad,\Bd}}[\Bd],
\end{aligned}
\end{equation}
where $\text{\rm HH}\overline{\textsc f}(h,\Ad)$ is a linear combination of the basis vectors $[C^\star_{k',l}]=[\SE{C}_{k'}\pm E_{h+1,l}|\SO{C}_{k'}\mp E_{h+1,l}]$
which have base matrix $C_{k',l}=A^-_{h,k'}$ with $1\leq l<k'\leq n$.

 Since $a_{h,k'}=a_{h+1,k'}=0$ for $k'<k {  (\leq h)}$, these matrices $A^-_{h,k'}$ and $C^\star_{k',l}$ are only defined (i.e., belong to $M_n(\NN|\NN_2)$) for $k\leq k'$ and $k\leq l$, respectively. Hence, $\text{\rm HH}\overline{\textsc f}(h,\Ad)$ is a linear combination of the basis vectors $[C^\star_{k',l}]$, for $k\leq l<k'$, with
  base matrix  $C_{k',l}=A^-_{h,k'}\prec A^-_{h,k}$ as $k'>k$.
In other words,
   $\text{\rm HH}\overline{\textsc f}(h,\Ad)$ is a linear combination of $[\Cd]$ with $C\prec A^-_{h,k}$ and hence, by definition, $C^\star\prec^\star (A^{\bar 0}-E_{h,k}|A^{\bar 1}+E_{h+1,k})$, the leading matrix.

Meanwhile, for these $M^\star$ appearing in the two summations in \eqref{triang_lower-QqS-pf-16}, since $M\approx A^-_{h,j}$ or $M\prec A^-_{h,j}$ with $j>k$, we clearly  have $M^\star\prec^\star (\SE{A}-E_{h,k}|\SO{A}+E_{h+1,k})$ as $A^-_{h,j}\prec A^-_{h,k}$.

Finally, it remains to deal with the term in \eqref{triang_lower-QqS-pf-16} involving $\delta_{h,k}$ when $h=k$ and $\SO{a}_{k,k}=1$. Though  $(\SE{A}-E_{h,k}|\SO{A}+E_{h+1,k})$ and $({\SE{A} + E_{h+1, k} |\SO{A} - E_{h,k}})$ have the same base matrix $A^-_{h,k}$,
but we have
$$\Tr({\SE{A} + E_{k+1, k} |\SO{A} - E_{k,k}})=\Tr(\Ad)\setminus \{k\}\subset \Tr(\SE{A}-E_{k,k}|\SO{A}+E_{k+1,k}).$$
Hence,
$$
({\SE{A} + E_{h+1, k} |\SO{A} - E_{h,k}})\prec^\star (\SE{A}-E_{h,k}|\SO{A}+E_{h+1,k}).
$$
Hence, this proves that \eqref{triang_lower-QqS-pf-16} has the required form as in (2).

Like Corollary \ref{InsertL}, similar arguments for proving (1) and (3) gives the following result.
\begin{cor}\label{InsertL2}
Keep the hypotheses of Proposition \ref{triang_lower-QqS} on the matrix $A$ and assume $\Md\in M_n(\NN|\NN_2)_r$ satisfying $\Md\prec^\star\Ad$. Then  we have (with the hypothesis in Proposition \ref{triang_lower-QqS}(1))
\begin{equation}\label{FhP}
\begin{aligned}
{[\Fd_{h,\ro(M),p}][M^\star]}=\text{lin. comb. of }[\Bd], \Bd\prec^\star((\SE{A})^{-,p}_{h,k}|\SO{A}),\;\;
\forall  p\in[1,a_{h,k}^{\bar{0}}],
\end{aligned}
\end{equation}
and (with the hypothesis in Proposition \ref{triang_lower-QqS}(2))
\begin{equation}\label{oddFhP}
\begin{aligned}
{[\Fd_{\bar h,\ro(M)}][M^\star]}=\text{lin. comb. of }[\Bd], \Bd\prec^\star(\SE{A}-E_{h,k}|\SO{A}+E_{h+1,k}).
\end{aligned}
\end{equation}
\end{cor}

(3) Finally, with the preparation of (1) and (2) and Corollary \ref{InsertL2}, the proof of (3) is entirely similar to the proof of Proposition \ref{triang_upper-QqS}(3) using the extended order $\preceq^\star$.
\end{proof}

We are now ready to introduce the aforementioned monomials and associated triangular relations in $\qSchvsQ$.
We need a linear order over the following set to define certain monomials (cf. the order on \cite[p.562]{DDPW}):{
$$\mathscr{I}_n:=\{(i,j)\mid 1\leq i, j\leq n\}=\mathscr I_n^<\sqcup\mathscr I_n^=\sqcup\mathscr I_n^>.$$
where $\mathscr I_n^<=\{(i,j)\in\mathscr I_n\mid i< j\}$, $\mathscr I_n^==\{(i,j)\in\mathscr I_n\mid i=j\}$} and $\mathscr I_n^>=\{(i,j)\in\mathscr I_n\mid i> j\}$.

\begin{defn}\label{order}  We linearly order $\mathscr I_n^<$ by setting,
 for $(i,j), (i',j')\in\mathscr{I}^<$,
 $$(i,j)<(i',j')\iff\text{either } j>j' \text{ or, if }j=j', \text{ then } i>i'.$$
 Then, linearly order $\mathscr I_n^==\{(i,i)\mid i\in[1,n]\}$ by setting
 $(i,i)<(i',i')\iff  i<i' ,$
 and linearly order $\mathscr I_n^>$ by setting,
 for $(i,j), (i',j')\in\mathscr{I}^>$,
 $$(i,j)<(i',j')\iff \text{either } j<j' \text{ or, if }j=j', \text{ but } i<i'.$$
Finally, {  for $(i,j)\in\mathscr{I}^<$, $(i',i')\in\mathscr{I}_n^=$ and  $(i'',j'')\in\mathscr{I}^>$, we set $(i,j)>(i',i')>(i'',j'')$.}
\end{defn}
For example, the ordering on
$\mathscr I_4$ is given by $(2,1)<(3,1)<(4,1)<(3,2)<(4,2)<(4,3)<(1,1)<(2,2)<(3,3)<(4,4)<(3,4)<(2,4)<(1,4)<(2,3)<(1,3)<(1,2)$.

 For any given $\Ad=(\SE{a}_{i,j}|\SO{a}_{i,j})\in\MNZ(n,r)$ with base matrix $A=(a_{i,j})$,  we now introduce the elements $\bfm^{(\Ad)}_{i,j}$
 (see Definition \ref{E^A_ij} below), for all $i,j\in\mathscr I_n$, and use the linear order on $\mathscr I_n$ to define the following products:
\begin{equation}\label{m(A)}
\bfm^{(\Ad)}={\prod_{{1} \le j \le {n-1}}} \Big({\prod_{j+1\leq i\leq n}} \bfm^{(\Ad)}_{i,j}\Big)
	\cdot{ \prod_{1\leq j\leq n} \bfm^{(\Ad)}_{j,j}\cdot
	{\prod_{{n} \ge j \ge {1}}}  \big ( {\prod_{{j-1} \ge i \ge {1}}} \bfm^{(\Ad)}_{i,j}\big) },
\end{equation}
where product order follows the convention in \eqref{eq_product}. In other words, we have
$$\begin{aligned}
\mathbf{m}^{(\Ad)}&
:=(\bfm^{(\Ad)}_{2,1}\bfm^{(\Ad)}_{3,1}\cdots \bfm^{(\Ad)}_{n,1})(\bfm^{(\Ad)}_{3,2}\bfm^{(\Ad)}_{4,2}\cdots \bfm^{(\Ad)}_{n,2})\cdots(\bfm^{(\Ad)}_{n,n-1})\\
&\cdot(\bfm^{(\Ad)}_{1,1}\cdots \bfm^{(\Ad)}_{n-1,n-1}\bfm^{(\Ad)}_{n,n})\cdot (\bfm^{(\Ad)}_{n-1,n}\cdots \bfm^{(\Ad)}_{2,n} \bfm^{(\Ad)}_{1,n})\cdots( \bfm^{(\Ad)}_{2,3}\bfm^{(\Ad)}_{1,3})(\bfm^{(\Ad)}_{1,2}).
\end{aligned}$$
\begin{defn}[for $\bfm^{(\Ad)}_{i,j}$] \label{E^A_ij}
We set $\bfm^{(\Ad)}_{i,j}=1$, for all $(i,j)\in\mathscr I_n$ with $ a_{i,j}=0$, and define $\bfm^{(\Ad)}_{i,j}$ downward recursively on the well-ordered subset
$\mathscr{I}_n(\Ad):=\{(i,j)\in\mathscr{I}_n\mid a_{i,j}\neq0\}\cup{ \{(1,2)\}}:$
\begin{enumerate}
\item For the largest element ${ (1,2)}$ in $\mathscr{I}_n(\Ad)$, set {  $$\bfm^{(\Ad)}_{1,2}=\begin{cases}[\Ed_{\ol{1},\co(A)-\SE{a}_{1,2}\alpha_1}]^{\SO{a}_{1,2}}[\Ed _{1,\co(A),\SE{a}_{1,2}}],&\text{if }{a}_{1,2}>0;\\
[\Dd_{\co(A)}],&\text{if }{a}_{1,2}=0.\end{cases}$$}
\item Suppose, for $(i',j')\in \mathscr{I}_n(\Ad)$ and $(i',j')\leq{ (1,2)}$, $\bfm^{(\Ad)}_{i',j'}$ is defined with $\gamma_{i',j'}=\ro(\bfm^{(\Ad)}_{i',j'})$,
and $(i,j)$ is an immediate predecessor of $(i',j')$ in $\mathscr{I}_n(\Ad)$. Then, set 
\begin{equation*}\label{EAij}
\bfm^{(\Ad)}_{i,j}\!=\!\begin{cases}[\Ed_{\bar i,\gamma_{i',j'}+\SE{a}_{i,j}\alpha_{i,j}+\SO{a}_{i,j}\alpha_{i+1,j}}]^{\SO{a}_{i,j}}[\Ed_{i,\gamma_{i',j'}+a_{i,j}\alpha_{i+1,j},\SE{a}_{i,j}}]\prod_{i<h<j}[\Ed_{h,\gamma_{i',j'}+a_{i,j}\alpha_{h+1,j},a_{i,j}}], \!&\mbox{if }i<j;\\
[\Dd_{\bar j,\gamma_{i',j'}}]^{\SO{a}_{j,j}},\! &\text{if }i=j;\\
{[\Fd_{\overline{i-1},\gamma_{i',j'}+\SE{a}_{i,j}\alpha_{i,j}+\SO{a}_{i-1,j}\alpha_{i,j}}]^{\SO{a}_{i,j}}[\Fd_{i-1,\gamma_{i',j'}+p\alpha_{i-1,j},\SE{a}_{i,j}}]
\prod_{i> h\geq j}[\Fd_{h,\gamma_{i',j'}+a_{i,j}\alpha_{h-1,j},a_{i,j}}]}, \!&\mbox{if }i>j,
\end{cases}
\end{equation*}
where $\alpha_{i,j}=\ep_i-\ep_{i+1}$ for all $i,j\in[1,n]$.
\end{enumerate}
 \end{defn}
 Note from the definition that the right most term of $\bfm^{(\Ad)}_{i,j}$ for $i<j$ is $[\Ed_{h,\gamma_{i',j'},a_{i,j}}]$ as $\alpha_{j,j}=0$ and $\co(\Ed_{h,\gamma_{i',j'},a_{i,j}})=\gamma_{i',j'}$.
Associated with $\Ad=(\SE{a}_{i,j}|\SO{a}_{i,j})\in\MNZ(n,r)$ and $(l,k)\in\mathscr{I}_n$, we may describe the leading term of the partial product $\prod_{ (l,k)\leq (i,j)}\bfm^{(\Ad)}_{i,j}$ over the interval $[(l,k),(1,2)]:=\{(i,j)\in\mathscr{I}_n\mid (l,k)\leq(i,j)\}$.

\begin{defn}[for matrices $\Adlk(l,\star,k)$]\label{leadingM} We start with $[D^\star_\mu]=[\mu|O]$ with $\co(A)=\mu$ and set
$\Adlk(1,\star,2)$ to be the matrix defining the leading term of  $\bfm_{1,2}^{(\Ad)}[D^\star_\mu]$. For $(l,k)<(1,2)$, define
 $\Adlk(l,\star,k)=(\SE{b}_{i,j}|\SO{b}_{i,j})$ by moving certain entries from $\Ad$ to the corresponding positions in $\Adlk(l,\star,k)$ as follows:

If $l<k$ (i.e., $(l,k)\in\mathscr I_n^<$), then the entries of $\Adlk(l,\star,k)$ are defined by setting
\begin{itemize}
\item[(1)]   $\SE{b}_{i,j}=\SE{a}_{i,j}$ and $\SO{b}_{i,j}=\SO{a}_{i,j}$ for $i<j<k$ and for $j=k$ but $i\leq l$;
\item[(2)]
$\SE{b}_{j,j}=\mu_j-\sum_{i<j}a_{i,j}={ \sum_{i\geq j}a_{i,j}}$ for $j<k$, $\SE{b}_{k,k}=\sum_{i> l}a_{i,k}$, and $\SE{b}_{j,j}=\mu_j$ for $j>k$;
\item[(3)]the remaining entries are 0.
\end{itemize}

If $l=k$ (i.e., $(l,k)\in\mathscr I_n^=$), set \begin{equation}\label{spe-matrix-lower-diag}
\Adlk(k,\star,k)=(\Adlk(n-1,{\bar0},n)-\sum_{j\geq k}\SO{a}_{j,j}E_{j,j}|\Adlk(n-1,{\bar1},n)+\sum_{j\geq k}\SO{a}_{j,j}E_{j,j}).
 \end{equation}

If $l>k$ (i.e., $(l,k)\in\mathscr I_n^>$), the entries of $\Adlk(l,\star,k)=(\SE{b}_{i,j}|\SO{b}_{i,j})$ are defined by setting
\begin{itemize}
\item[(a)] $b^{\bar0}_{i,j}=0=b_{i,j}^{\bar1}$, for $i>j$ and $j<k$ or for  $j=k$ but $k<i<l$;
\item[(b)] $\SE{b}_{j,j}=\sum_{i> j}a_{i,j}+\SE{a}_{j,j}$ for $j<k$;  $\SE{b}_{k,k}=\SE{a}_{k,k}+\sum_{k<i< l}a_{i,k}\;\;(a_{i,j}=\SE{a}_{i,j}+\SO{a}_{i,j})$;
\item[(c)] all remaining entries are the same as the corresponding entries of $\Ad$.
\end{itemize}
\end{defn}
\noindent
{\bf Observe from the definition the following:}

The least $(l,k)$ satisfying  $l<k$ is $(n-1,n)$. Thus, if we use the canonical decomposition
\begin{equation}\label{pzm}
M=M_-+M_0+M_+
\end{equation} of a square matrix $M$ into a sum of lower triangular, diagonal, and upper triangular parts, then the entries in the upper triangular part
$A^{\star}_+=({A}^{\bar0}_+|A^{\bar1}_+)$ are all in the corresponding positions of $\Adlk(n-1,{\star},n)$.

The least $(l,k)$ satisfying  $l=k$ is $(1,1)$. Note that $\Adlk(1,\star,1)$ is the matrix such that the entires in $A^{\bar0}_+$ and $A^{\bar1}_0+A^{\bar1}_+$ are all in the corresponding positions of $\Adlk(1,\star,1)$.

Finally, the least $(l,k)$ satisfying  $l>k$ is $(2,1)$. By the definition, we see that $\Adlk(2,\star,1)=\Ad$.

\begin{thm}\label{triang-QqS}
For  any $\Ad=(\SE{a}_{i,j}|\SO{a}_{i,j})\in\MNZ(n,r)$ with base matrix $A=(a_{i,j})$,
we have
\begin{equation*}\label{triang-QqS-F}
{\bf m}^{(\Ad)}=(-1)^{\sum_{l>k}a^{\bar1}_{l,k}(\wp(\Ad)-\SO{\widetilde{a}}_{l,k})+\sum_{i\leq j,j>k}\SO{a}_{k,k}\SO{a}_{i,j}}\up^{\sum_{l>k}a^{\bar1}_{k,k}(a_{k,l}-a_{l,k})}[\Ad]+
(\text{lower terms})_{\prec^\star}.
\end{equation*}
In particular, the set $\{{\bf m}^{(\Ad)}\mid \Ad\in\MNZ(n,r)\}$ forms a basis for $\qSchvsQ$.
\end{thm}

\begin{proof}
 For a fixed $(l,k)\in\mathscr{I}_n$, let
 \begin{equation}\label{partial prod}
 \bfm^{(\Ad)}|_{(l,k)}:=\prod_{ (l,k)\leq (i,j)\leq(1,2)}\bfm^{(\Ad)}_{i,j}.
 \end{equation}
  be the product
  over the interval $[(l,k),(1,2)]$ and, for $(i,j)\in\mathscr{I}_n$, let
 $$\mathfrak{f}^{i,j}_{\Ad}=\begin{cases}1,&\text{if } i< j\text{ or }i=j=n;\\
 { (-1)^{\SO{a}_{i,i}{\sum_{s\leq t,t>i}\SO{a}_{s,t}}}\up^{\sum_{t>i}\SO{a}_{i,i}a_{i,t}},} &\text{if }{  i=j< n;}\\
  (-1)^{\SO{a}_{i,j}(\parity{\Ad}-\SO{\tilde{a}}_{i,j})}\up^{-{a}_{i,j}\SO{a}_{j,j}},& \text{if } i>j.\end{cases}$$
 We claim that
\begin{equation}\label{triang-QqS-pf-2}
\bfm^{(\Ad)}|_{(l,k)}=\lc^{(l,k)}[\Adlk(l,\star,k)]+(\text{lower terms})_{\prec^\star}.
\end{equation}
Here $\lc^{(l,k)}=\Big(\prod_{(i,j)\geq (l,k)}\mathfrak{f}^{i,j}_{\Ad}\Big)$ is the leading coefficient at $(l,k)$ and
``$(\text{lower terms})_{\prec^\star}$'' means a linear combination of $[\Md]$  with $\Md\prec^\star\Adlk(l,\star,k)$, for $\Md\in\MNZ(n,r)$.

To prove the claim, we apply induction on the well-ordered set $(\mathscr{I}_n,\leq)$ (or, more precisely, on $(\mathscr{I}_n(\Ad),\leq)$).

For the inductive base $(l,k)=(1,2)$, if $a_{1,2}=0$, then, by Definition \ref{E^A_ij}(1), $$\bfm^{(\Ad)}_{1,2}=[\Dd_{\co(A)}]=[\Adlk(1,\star,2)].$$ If $a_{1,2}>0$, applying Proposition \ref{triang_upper-QqS}(3) yields
$$\bfm^{(\Ad)}_{1,2}=[\Ed_{\ol{1},\co(A)-\SE{a}_{1,2}\alpha_1}]^{\SO{a}_{1,2}}[\Ed _{1,\co(A),\SE{a}_{1,2}}]
=[\Adlk(1,\star,2)]+(\text{lower terms})_{\prec^\star}.
$$

Assume now $(l',k')\leq(1,2)$ and the claim is true at $(l',k')$:
\begin{equation}\label{triang-QqS-pf-3}
\bfm^{(\Ad)}|_{(l',k')}=\lc^{(l',k')}[\Adlk({l'},\star,{k'})]+(\text{lower terms})_{\prec^\star}.
\end{equation}

Let $(l,k)$ be the immediate predecessor of $(l',k')$ in $\mathscr{I}_n(\Ad)$.
We now prove that the claim is true at $(l,k)$. In other words, $\bfm^{(\Ad)}|_{(l,k)}=\bfm^{(\Ad)}_{l,k}\cdot\bfm^{(\Ad)}|_{(l',k')}$ has the leading term
$[\Adlk(l,\star,k)]$. There are three cases to consider.

{\bf Case 1 ($l<k$)}.
In this case, for any  $(l,k)<(i,j)<(1,2)$, we have {  $i< j$} and so $\mathfrak{f}^{i,j}_{\Ad}=1$. By the definition linear ordering on $\mathscr{I}_n(\Ad)$, $(l',k')$ must be one of the following two situations:
  \begin{enumerate}
\item[(1a)]  if $l=1$, then {  $k>2$ and$(l',k')=(k-2,k-1)$};
\item[(1b)]  if $l>1$,  then $(l',k')=(l-1,k)$.
\end{enumerate}

In both cases, by Definition \ref{leadingM}(1)--(3), the matrices involved have the shape as in \eqref{eq:special matrix-upper}. Thus,
applying Proposition \ref{triang_upper-QqS} yields,
\begin{equation*}\label{them-case1-3}
\bfm^{(\Ad)}_{l,k}[\Adlk(l',\star,k')]=[\Adlk(l,\star,k)]+(\text{lower terms})_{\prec}.
\end{equation*}
For a lower term $[\Cd]$ with $C\prec \Adlk(l', ,k')$, the base of $\Adlk(l',\star ,k')$, by Definition \ref{preorder-2}, repeatedly applying  Corollary \ref{InsertL} gives
\begin{equation*}\label{them-case1-4}
\bfm^{(\Ad)}_{l,k}[\Cd]=\text{lin. comb. of }[\Md], M \prec \Adlk(l, ,k),
\end{equation*}
proving \eqref{triang-QqS-pf-2} in this case.

{\bf Case 2 ($l=k$).} By the definition of linear ordering on $\mathscr{I}_n(\Ad)$, $(l',k')$ must be one of the following two situations:
  \begin{enumerate}
\item[(2a)]  if $l=n$, then $k=n$ and $(l',k')=(n-1,n)$;
\item[(2b)]  if $l<n$,  then $(l',k')=(k+1,k+1)$.
\end{enumerate}

The proof of (2a) uses \eqref{triang_diag-QqS-D-n} with leading coefficient 1. This case can be done similarly by Proposition \ref{middleTR} and Corollary \ref{triang_diag-QqS-DF}.

If $k<n$ and $(l',k')=(k+1,k+1)$, by \eqref{spe-matrix-lower-diag}, the base matrix $B$ of $\Bd=\Adlk(k+1,\star,k+1)$ is with $b_{i,j}=0$ for $i>j$ and $\SO{b}_{k,k}=0$, moreover, $(\SE{B}-\SO{a}_{k,k}E_{k,k}|\SO{B}+\SO{a}_{k,k}E_{k,k})=\Adlk(k,\star,k)$.
By Proposition \ref{middleTR}, we have
$$\begin{aligned}\bfm^{(\Ad)}_{k,k}[\Adlk(k+1,\star,k+1)]=(-1)^{a_{k,k}^{\bar1}(\wp(\Adlk(k+1,\star,k+1))+\SO{\tilde b}_{k,k})}\up^{a_{k,k}^{\bar1}\sum_{t>k}b_{k,t}}[\Adlk(k,\star,k)]+\text{(lower terms)}_{\prec^\star}.
\end{aligned}
$$
Since $\wp(\Adlk(k+1,\star,k+1))=\sum_{i<j}\SO{a}_{i,j}+\sum_{j>k}\SO{a}_{j,j}$,  $\SO{\tilde b}_{k,k}=\sum_{i<j\leq k}\SO{a}_{i,j}$ and
$\sum_{t>k}b_{k,t}=\sum_{t>k}a_{k,t}$,
$$\begin{aligned}(-1)^{a_{k,k}^{\bar1}(\wp(\Adlk(k+1,\star,k+1))+\SO{\tilde b}_{k,k})}\up^{a_{k,k}^{\bar1}\sum_{t>k}b_{k,t}}&=(-1)^{a_{k,k}^{\bar1}(\sum_{i<j}\SO{a}_{i,j}+\sum_{j>k}\SO{a}_{j,j}+\sum_{i<j\leq k}\SO{a}_{i,j})}\up^{a_{k,k}^{\bar1}\sum_{t>k}a_{k,t}}\\
&=(-1)^{a_{k,k}^{\bar1}\sum_{i\leq j,j>k}\SO{a}_{i,j}}\up^{a_{k,k}^{\bar1}\sum_{t>k}a_{k,t}}=\mathfrak{f}^{k,k}_{\Ad}.
\end{aligned}
$$

For lower term $[\Cd]$ with $\Cd\prec \Adlk(k+1,\star,k+1)$ in  \eqref{triang-QqS-pf-3}, applying Corollary \ref{triang_diag-QqS-DF}
proves  \eqref{triang-QqS-pf-2} in this case.

{\bf Case 3 $l>k$.} By the definition linear ordering on $\mathscr{I}_n(\Ad)$, $(l',k')$ must be one of the following two situations:
  \begin{enumerate}
\item[(3a)]  if $l=n$, then $k=n-1$ and $(l',k')=(1,1)$ or $k<n-1$ and $(l',k')=(k+2,k+1)$;
\item[(3b)]  if $l<n$,  then $(l',k')=(l+1,k)$.
\end{enumerate}

If $(l',k')=(1,1)$, the base matrix $B$ of $\Bd=\Adlk(1,\star,1)$ has the shape as in \eqref{eq:special matrix-lower}, and $$(\SE{B}-a_{n,n-1}E_{n-1,n-1}+\SE{a}_{n,n-1}E_{n,n-1}|\SO{B}+\SO{a}_{n,n-1}E_{n,n-1})=\Adlk(n,\star,n-1).$$
Applying Proposition \ref{triang_lower-QqS}(3), due to $\parity{\Adlk(1,\star,1)}+\SO{\tilde{b}}_{n-1,n-1}=2\SO{\tilde{b}}_{n-1,n-1}+\parity{\Ad}-\SO{\tilde{a}}_{n,n-1}$, $\SO{b}_{n-1,n-1}=\SO{a}_{n-1,n-1}$ and definition of $\mathfrak{f}^{n,n-1}_{\Ad}$ above, we have
\begin{equation}
\begin{aligned}
\bfm^{(\Ad)}_{n,n-1}[\Adlk(1,\star,1)]
&=(-1)^{\SO{a}_{n,n-1}(\parity{\Ad}-\SO{\tilde{a}}_{n,n-1})}\up^{-{a}_{n,n-1}\SO{a}_{n-1,n-1}}[\Adlk(n,\star,n-1)]+(\text{lower terms})_{\prec^\star}.
\end{aligned}
\end{equation}
Thus, the first case in (3a) follows from Corollary \ref{InsertL2}.

If $(l',k')=(k+2,k+1)$, by Definition \ref{leadingM}(a)--(c), the base matrix $B$ of $\Bd=\Adlk(k+2,\star,k+1)$ has the shape as in \eqref{eq:special matrix-lower}, and $(\SE{B}-a_{n,k}E_{k,k}+\SE{a}_{n,k}E_{n,k}|\SO{B}+{\color{black}\SO{a}_{n,k}}E_{n,k})=\Adlk(n,\star,k)$. Thus, this subcase can be proved similarly by Proposition \ref{triang_lower-QqS}(3)
and Corollary \ref{InsertL2}.

Finally, if $(l',k')=(l+1,k)$, the base matrix $B$ of $\Bd=\Adlk(l+1,\star,k)$ still has the shape as in \eqref{eq:special matrix-lower}, and $(\SE{B}-a_{l,k}E_{k,k}+\SE{a}_{l,k}E_{n,k}|\SO{B}+\SO{a}_{l,k}E_{l,k})=\Adlk(l,\star,k)$. Applying Proposition \ref{triang_lower-QqS}(3),
\begin{equation*}
\begin{aligned}
\bfm^{(\Ad)}_{l,k}[\Adlk(l+1,\star,k)]&=(-1)^{\SO{a}_{l,k}(\parity{\Adlk(l+1,\star,k)}+\SO{\tilde{b}_{k,k}})}\up^{-a_{l,k}\SO{b}_{k,k}}[\Adlk(l,\star,k)]
+(\text{lower terms})_{\prec^\star}\\
&=\mathfrak{f}^{l,k}_{\Ad}{\color{black}[\Adlk(l,\star,k)]}+(\text{lower terms})_{\prec^\star},
\end{aligned}
\end{equation*}
noting that $\parity{\Adlk(l+1,\star,k)}+\SO{\tilde{b}_{k,k}}=2\SO{\tilde{b}}_{k,k}+\parity{\Ad}-\SO{\tilde{a}}_{l,k}$, $\SO{b}_{k,k}=\SO{a}_{k,k}$ and $\mathfrak{f}^{l,k}_{\Ad}=(-1)^{\SO{a}_{l,k}(\parity{\Ad}-\SO{\tilde{a}}_{l,k})}\up^{-a_{l,k}\SO{a}_{k,k}}.$

Since $\lc^{(l,k)}=\mathfrak{f}^{l,k}_{\Ad}\lc^{(l',k')}$, applying Corollary \ref{InsertL2} proves the claim \eqref{triang-QqS-pf-2} in this $l>k$ case. This completes the proof of the theorem.
\end{proof}

\begin{cor}\label{generators}
The $\mathbb Q(\up)$-algebra
$\qSchvsQ$ is generated by
\begin{equation}
\{[\Ed_{h,\lambda}],[\Ed_{\bar{h},\lambda}],[\Fd_{h,\lambda}],[\Fd_{\bar{h},\lambda}],[\Dd_\lambda], [\Dd_{\bar i,\lambda}]\mid \lambda\in\Lambda(n,r),1\leq h<n, i\in [1,n]\}.
\end{equation}
\end{cor}
\begin{proof}By Lemma \ref{pEh}, $[p]^![\Ed_{h,p,\lambda}]$ (resp., $[p]^![\Fd_{h,p,\lambda}]$) is a product of some $[\Ed_{h,\mu}]$
(resp., $[\Fd_{h,\mu}]$). Thus,  Theorem \ref{triang-QqS} implies that, for any $\Ad\in\MNZ(n,r)$, $[\Ad]$ is in the subalgebra generated by  $$\{[\Ed_{h,\lambda}],[\Ed_{\bar{h},\lambda}],[\Fd_{h,\lambda}],[\Fd_{\bar{h},\lambda}],  [\Dd_{\bar i,\lambda}], [\Dd_\lambda]\mid \lambda\in\Lambda(n,r),1\leq h<n, i\in [1,n]\}.$$
Now the assertion follows from Corollary \ref{sdbasis}.
\end{proof}

\section{Mapping $\Uvqn$ onto a superalgebra of formal infinite series}\label{sec_generators}\label{defining relations}
Recall from Proposition \ref{mulformzerocor} that {$\AJRS({O}, \bs{0}, r)$} is the identity element in $\qSchvsQ$.
Define the superalgebra
$$
{\bs{\mathcal  Q}_\up^s(n)}  := \prod_{r \ge 1 } {\qSchvsQ}.
$$
The elements $(x_r)_{r\geq1}$ in $\bsSQvn$
are written as {\it formal infinite series} $\sum_{r\geq1}x_r$
 which is homogenous of degree $\bar i$ if every $x_r$ is homogenous of degree $\bar i$.
In particular, the identity element
\begin{equation}\label{id elt}
1=\sum_{\la\in\NN^n\setminus\{\bs{0}\}}[\Dd_\la]=\sum_{r\ge1}1_r,\qquad [\Dd_\la][\Dd_\mu]=\delta_{\la,\mu}[\Dd_\la],\;\forall \la,\mu\in\NN^n\setminus\{\bs{0}\},
\end{equation}
 where $1_r=\sum_{\la\in\Lambda(n,r)}[\Dd_\la]$ is the identity element of $\qSchvsQ$.
For any  $\Ad \in \MNZNS(n), \bs{j} \in {\ZZ}^{n}$ (see \eqref{Mnpm}),
let
\begin{equation}\label{A(j)}
\AJS(\Ad, \bs{j}) := \sum_{r \ge 1 } \AJRS(A, \bs{j}, r)=\sum_{r\geq|A|}\sum_{\substack{\lambda \in \CMN(n, r-\snorm{A})} }
	 {v}^{\lambda* \bs{j}} [ \SE{A} + \lambda | \SO{A} ]
\in {\bsSQvn},
\end{equation}
a formal infinite series,
 and form the superspace
\begin{equation}\label{Anv}
    \USnv = \tspan_{\Qv} \{\AJS(\Ad, \bs{j})\  \where  \Ad \in \MNZNS(n), \bs{j} \in {\ZZ}^{n} \} \subset {\bsSQvn}.
\end{equation}
Note that $1=(O|O)({\bf0})\in \USnv $.

The linear independence of the elements $\Ad(\bs{j})$ can be proved by a method
similar  to the  proof of \cite[Proposition 4.1{\rm(2)}]{DF2} or \cite[Th.~8.2]{DLZ}.
\begin{lem}\label{unv_basis}
The set {$\mathfrak{L}=\{  \AJS(\Ad, \bs{j})  \where   \Ad \in \MNZNS(n), \bs{j} \in {\ZZ}^{n} \} $} is a $\Qv$-basis of {$\USnv$}.
\end{lem}

We introduce the following subset of generators in $\USnv$:
\begin{equation}\label{setG}
 \fsG_n = \{ G_{i}, G_{i}^{-1}, G_{\ol{i}},
	X_{j}, X_{\ol{j}},
	Y_{j}, Y_{\ol{j}}
	\where  1 \le i \le n,\  1 \le j \le n-1
	\},
\end{equation}
where
\begin{equation}\label{raw generator}
\aligned
&G_{i}^{\pm 1}
	= \ABJS(O, O, \pm \ep_{i}), \qquad
X_{j} = \ABJS(E_{j, j+1}, O, \bs{0}) , \qquad
Y_{j} = \ABJS(E_{j+1, j}, O, \bs{0}) , \\
& G_{\ol{i}} 	=  \ABJS(O, E_{i,i}, \bs{0}), \qquad
 X_{\ol{j}} =    \ABJS( O, E_{j,j+1}, \bs{0} ),  \qquad
Y_{\ol{j}} 	=   \ABJS(O, E_{j+1,j},  \bs{0} ).
 \endaligned
\end{equation}
For notational clarity, we also write $ \ol{G}_i=G_{\ol{i}} ,\ol{X}_i= X_{\ol{j}} $ and $\ol{Y}_i= Y_{\ol{j}} $.

The following result is immediate from Corollary \ref{common_form} together with the fact that the coefficients $\scp_{\Bd,\bs{j}}(Z_\bullet,\Ad)$ are independent of $r\geq |A|$.
\begin{lem}\label{rem:prod alg}
For any {$\Ad \in \MNZNS(n)$} and $Z  \in\fsG_n $, there exist some $\scp_{\Bd,\bs{j}}(Z,\Ad) \in \Qv $ such that in $\bsSQvn$
\begin{equation}\label{general-ZA}
\begin{aligned}Z  \cdot \AJS(\Ad, \bs{0})
&=	\sum_{\Bd, \bs{j}'} \scp_{\Bd, \bs{j}'}(Z,\Ad )  \AJS(\Bd, \bs{j}' ),
\end{aligned}
\end{equation}
where $\Bd \in \MNZNS(n)$,  $\bs{j}'\in {\ZZ}^n$. In particular, for any $\bs{j}\in\mathbb Z^n$, there exists $a_{\bs j}\in\ZZ$ such that
\begin{equation}\label{AjAl}
 \AJS(\Ad, \bs{j}) = \up^{a_{\bs j}} \cdot (\prod_{h=1}^{n}G_{h}^{{j}_h}  ) \cdot \AJS(\Ad, \bs{0}).
 \end{equation}
\end{lem}
\begin{proof} The last assertion follows from the fact
\begin{align}\label{mult-GA}
G_i^{\pm} \cdot\AJS(\Ad, \bs{j})
	=v^{\pm\ro(A)*\ep_i} \AJS(\Ad, \bs{j} \pm \ep_i),
	\quad
\AJS(\Ad, \bs{j})\cdot G_i^{\pm} =v^{\pm\co(A)*\ep_i} \AJS(\Ad, \bs{j} \pm \ep_i),
\end{align}
see Proposition \ref{mulformzerocor}.
\end{proof}

We are now ready to show the Main Theorem \ref{mthm} in the introduction by the following two steps: (1) Prove the image of $\bs{\xi}_n$ is $\USnv$, and (2) Prove $\bs{\xi}_n$ is injective. We complete (1) in this section and leave (2) to the next section.

For any {${}'\!\!\Ad=(a^{\bar 0}_{i,j}|a^{\bar 1}_{i,j}) \in \MNZNS(n)$}, define
\begin{equation}\label{def_udl}
\scm_{i,j}^{{}'\!\!\Ad}=\begin{cases}
{\ol{X}_{i}}^{\SOE{a}_{i,j}}
	X_{i}^{(\SEE{a}_{i,j})}    X_{i+1}^{({a}_{i,j})} \cdots   X_{j-2}^{({a}_{i,j})} X_{{j-1}}^{({a}_{i,j})},
	 &\mbox{for all } 1 \le i < j \le n; \\
\ol{ G}_{i}^{\SOE{a}_{i,i}} ,& \mbox{for all} 1\leq i=j\leq n;\\
 {\ol{Y}_{i-1}}^{\SOE{a}_{i,j}}
	Y_{i-1}^{(\SEE{a}_{i,j})}    Y_{i-2}^{({a}_{i,j})} \cdots   Y_{j+1}^{({a}_{i,j})} Y_{{j}}^{({a}_{i,j})} ,
	 &\mbox{for all } 1 \le j<i \le n.
	 \end{cases}
\end{equation}
Here, for any $Z\in \fsG_n $ and positive integer $a$, $Z^{(a)}=\frac{Z^{a}}{[a]!}$ denotes the divided power {\color{black} and we set $Z^{(a)}=1$ in the case $a=0$}.

We also define some monomials in the generators in $\mathscr G_n$, using the product order convention in \eqref{eq_product}.
For $'\!\!\Ad$ as above  and $\bs j=(j_1,\ldots,j_n)\in\ZZ^n$, use the well-ordered set $(\mathscr{I}_n,\leq)$ (Definition \ref{order}) to define
\begin{equation}\label{MA}
\aligned
{ {\scm}^{{}'\!\!\Ad,\bf0} } &:=\prod_{j\in[1,n)}^< \Big({\prod_{j+1\leq i\leq n}} \scm^{{}'\!\!\Ad}_{i,j}\Big)
	\cdot
	{ \prod_{1\leq j\leq n}\scm^{{}'\!\!\Ad}_{j,j} \cdot\prod_{j\in[1,n]}^>  \big ( {\prod_{{j-1} \ge i \ge {1}}} \scm^{{}'\!\!\Ad}_{i,j}\big)} =\prod_{{\color{black}(i,j)\in(\mathscr{I}_n,\leq)}}{\color{black}\scm}_{i,j}^{{}'\!\!\Ad},\\
\scm^{{}'\!\!\Ad,\bs j}&:=G_1^{j_1}\cdots G_n^{j_n}{\scm}^{{}'\!\!\Ad,\bf0}.
\endaligned
\end{equation}
Here, for $j=1$, the product $\prod_{0\geq i\geq1}$ is understood as 1.

By Proposition \ref{mulformeven-2} and Remark \ref{longtail}(2),
we have, for $a\in\NN, i\in[1,n)$,
\begin{equation}\label{divided-power}
X_i^{(a)}=(aE_{i,i+1}|O)(\bs{0}), \;\;Y_i^{(a)}=(aE_{i+1,i}|O)(\bs{0}).
\end{equation}
Thus, if $\Ad={}'\!\!\Ad+(\diag(a_{1,1}^{\bar0},\cdots,a_{n,n}^{\bar0})|O)\in\MNZ(n,r)$ has base $A$, then, by \eqref{id elt}
\begin{equation}\label{M to m}
\scm^{{}'\!\!\Ad,\bf0}\cdot[\Dd_{{ \co(A)}}]={\bf m}^{(\Ad)},\;\;\text{the monomial defined in \eqref{m(A)}},
\end{equation}
and
$\scm_r^{{}'\!\!\Ad,\bf0}:={\scm}^{{}'\!\!\Ad,\bf0}\cdot 1_r$ is the $r$-th component of $\scm^{{}'\!\!\Ad,\bf0}$ in $\qSchvsQ$.

Note that \eqref{M to m} can  be seen as follows. First,
$\scm^{{}'\!\!\Ad}_{1,2} [\Dd_{\co(A)}]=\bfm_{1,2}^{(\Ad)}$ (see Definition \ref{E^A_ij}). Assume $(i',j')\geq(1,2)$ and, for
$\scm^{{}'\!\!\Ad,\bf0}|_{(i',j')}$ as similarly defined in \eqref{triang-QqS-pf-2}, $\scm^{{}'\!\!\Ad,\bf0}|_{(i',j')}\cdot[\Dd_{{\co(A)}}]={\bf m}^{(\Ad)}|_{(i',j')}$. If $\ro({\bf m}^{(\Ad)}|_{(i',j')})=\gamma_{i',j'}$ and $(i,j)$ is
 immediate predecessor of $(i',j')$, then
$
\scm^{{}'\!\!\Ad}_{i,j}\cdot[\Dd_{\gamma_{i',j'}}]=\bfm^{(\Ad)}_{i,j}.
$ This proves  \eqref{M to m} by induction.

\begin{thm}\label{triangular_relation_q}
(1) For any {$\Ad \in \MNZNS(n)$},  we have in $\SQvnR$
\begin{align}\label{longTR}
{\scm}^{\Ad,\bf0}
	=\lc^\Ad \AJS(\Ad, \bs{0}) +
		\sum_{\!\!\Bd \prec \Ad , \  \bs{j}\in \ZZ^{n} }
		\mathfrak{f}^{\Ad}_{\Bd, \bs{j}} \AJS(\Bd, \bs{j}),
\end{align}
for some  $\mathfrak{f}^{\Ad}_{\Bd, \bs{j}} \in \Qv$, where $$\lc^\Ad={ (-1)^{\sum_{l>k}\SO{a}_{l,k}(\parity{\Ad}-\SO{\tilde{a}}_{l,k})+\sum_{i\leq j,j>k}\SO{a}_{k,k}\SO{a}_{i,j}}\up^{\sum_{l>k}\SO{a}_{k,k}(a_{k,l}-a_{l,k})}}.$$

(2) The $\Qv$-space  {$\USnv$} is a subalgebra of $\SQvnQ$ generated by the set $\fsG_n$ defined in \eqref{setG}. Moreover, the set
$\{\scm^{\Ad,\bs j}\mid \Ad\in M^\pm_n(\NN|\NN_2),\bs j\in\ZZ^n\}$ forms a basis for $\USnv$.

(3) There is  a superalgebra  epimorphism $\bs{\xi}_n: \Uvqn \to \USnv$ defined by
$${\genE}_{j} \mapsto X_{j}, \;
{\genE}_{\ol{j}} \mapsto X_{\ol{j}}, \;
{\genF}_{j} \mapsto Y_{j}, \;
{\genF}_{\ol{j}} \mapsto Y_{\ol{j}},  \; {\genK}_{i}^{\pm 1} \mapsto G_{i}^{\pm 1}, \;
{\genK}_{\ol{i}} \mapsto G_{\ol{i}},$$
for all $1 \le i \le n, 1 \le j \le n-1$.

\end{thm}
\begin{proof} (1) With the notation introduced above, we have, by Theorem \ref{triang-QqS},
$$\aligned
\scm^{\Ad,\bf0}1_r&=\sum_{\lambda\in\Lambda(n,r-|A|)}{\scm}^{\Ad,\bf0}\cdot[\Dd_{\co(A+\la)}]=\sum_{\lambda\in\Lambda(n,r-|A|)}{\bf m}^{(\Ad+\la)}\\
&=\sum_{\lambda\in\Lambda(n,r-|A|)}\mathfrak{f}_{(\SE{A}+\lambda|\SO{A})}[\SE{A}+\lambda|\SO{A}]+\sum_{\lambda\in\Lambda(n,r-|A|)}\sum_{\substack{C^\star\in\MNZ(n,r)\\ C^\star\prec^\star (\SE{A}+\lambda|\SO{A})}}{ \fkf_{A^\star+\la,C^\star}}[C^\star]\\
&=\lc^\Ad \Ad({\bf0},r)+\sum_{\lambda\in\Lambda(n,r-|A|)}\sum_{\substack{C^\star\in\MNZ(n,r)\\ C^\star\prec^\star \Ad}}{ \fkf_{A^\star+\la,C^\star}}[C^\star],
\endaligned
$$
where $\lc^\Ad=\mathfrak{f}_{(\SE{A}+\lambda|\SO{A})}={ (-1)^{\sum_{l>k}\SO{a}_{l,k}(\parity{\Ad}-\SO{\tilde{a}}_{l,k})+\sum_{i\leq j,j>k}\SO{a}_{k,k}\SO{a}_{i,j}}\up^{\sum_{l>k}\SO{a}_{k,k}(a_{k,l}-a_{l,k})}}$ is independent of $r$. Thus,
$$\scm^{\Ad,\bf0}=\lc^\Ad \Ad({\bf0})+\sum_{r\geq|A|}\sum_{\lambda\in\Lambda(n,r-|A|)}\sum_{\substack{C^\star\in\MNZ(n,r)\\ C^\star\prec^\star \Ad}}{ \fkf_{A^\star+\la,C^\star}}[C^\star].$$
On the other hand, by Lemma \ref{rem:prod alg}, $\scm^{\Ad,\bf0}$ is a linear combination of $\Bd(\bs{j})$. Consequently, the triple summation above is a linear combination of $\Bd(\bs{j})$ with $\Bd\prec^\star\Ad$, proving \eqref{longTR}.

(2) By \eqref{longTR} and \eqref{AjAl}, there exists $a_\Ad\in\ZZ$ such that
$${\scm}^{\Ad,\bs j}
	=\pm\up^{a_\Ad} \AJS(\Ad, \bs{j}) +
		\sum_{\!\!\Bd \prec^\star \Ad , \  \bs{j}'\in \ZZ^{n} }
		\mathfrak{f'}^{\Ad}_{\Bd, \bs{j}'} \AJS(\Bd, \bs{j}').$$
This relation shows that every $A(\bs j)$ is in the subalgebra generated by $\fsG_n$ and all ${\scm}^{\Ad,\bs j} $ forms a basis.

(3) Note that the map $\bs\xi_n$ is induced by the homomorphisms $\bs\xi_{n,r}$ given in Theorem \ref{qqschur_reltion-r}. So, it is well-defined. The surjectivity follows from (2).
\end{proof}

\begin{cor}\label{surj-xi-nr}
For each $r\geq 0$, the homomorphism $\bs{\xi}_{n,r}:\Uvqn \to \qSchvsQ $ in Theorem \ref{qqschur_reltion-r} is surjective.  Moreover $\qSchvsQ$ is generated by
$
\fsG_{n,r} $ (see \eqref{setGr}) and spanned by $\{\Ad(\bs{j},r)\mid \Ad\in \MNZNS(n), \bs{j}\in\mathbb{Z}^n\}$.

\end{cor}

\begin{proof}We first observe that if $\pi_r$ denotes the canonical projection from $\bsSQvn$ onto its $r$-th component, then $\bs\xi_{n,r}=\pi_r|_{\USnv}\circ\bs\xi_n$.
Note also that, for $\lambda\in\Lambda(n,r)$, each element $G_\la:=\prod_{i=1}^n\left[G_{i}\atop \lambda_i\right]$ with ${\left[G_{i}\atop \lambda_i\right]}=\prod_{t=1}^i\frac{G_{i}\up^{-t+1}-G_{i}\up^{t-1}}{\up^t-\up^{-t}}$ has an inverse image in $\Uvqn$ and
$\pi_r(G_\la)=[\lambda|O]=[\Dd_\lambda]$, by \cite[(13.9.3)]{DDPW}. Thus, the set
\begin{equation}
\{\scm^{{}'\!\!\Ad,\bf0}G_{\co(A)}\mid\Ad\in M_n(\NN|\NN_2)_r\}\subset\bsSQvn
\end{equation}
is a subset of the image of $\bs\xi_n$. By \eqref{M to m} and Theorem \ref{triang-QqS} and noting ${}'\!\!\Ad=\Ad-(\diag(A^{\bar0})|O)$,
$$\pi_r\{\scm^{{}'\!\!\Ad,\bf0}G_{\co(A)}\mid\Ad\in M_n(\NN|\NN_2)_r\}=\{\mathbf m^{(\Ad)}\mid \Ad\in M_n(\NN|\NN_2)_r\}$$
forms a basis for $\qSchvsQ $. Hence, $\bs\xi_{n,r}$ is surjective. The last two assertions follow easily.
\end{proof}
 \begin{rem}
Using the formulas in Lemma \ref{phiupper-even-norm}, Lemma \ref{normlized-phidiag1} and Proposition \ref{mulformzerocor}, it is straightforward to check that the ideal of $\Uvqn$ generated by the elements in \eqref{SvRelations}
is contained in ${\ker}(\bs{\xi}_{n,r})$. Then by the surjectivity of $\bs{\xi}_{n,r}$ and a dimension comparison, one can deduce that ${\ker}(\bs{\xi}_{n,r})$ actually coincides with $\ker(\Phi_r)$ in Proposition \ref{DW2}.  We may use the $\bs{\sH}_r^c$-module isomorphism
$\bs V(n|n)^{\otimes r}\cong \bigoplus_{\lambda\in\Lambda(n,r)}x_\la\bs\sH_r^c$ to induce an isomorphism $\bs{\omega}_r:\qSchvsQ\to \text{End}_{\bs\sH_r^c}(\bs V(n|n)^{\otimes
 r})$.  It is natural to expect that  $\Phi_r=\bs\omega_r\circ\bs\xi_{n,r}$.
\end{rem}

\section{Completing the proof of the main Theorem \ref{mthm}}\label{sec_generators}

In this section, we shall prove that {$ \bs{\xi}_n$} in Theorem \ref{triangular_relation_q}(3) is injective and, hence, we obtain an isomorphism between {$\Uvqn$} and {$\USnv$}. Thus, the main theorem in subsection 1.4 follows.

In the quantum $\mathfrak{gl}_n$ case, a monomial basis which is triangularly related to a PBW type basis is used to prove the type $A$ map similar to $\bs\xi_n$ is injective, see \cite[5.7]{BLM}. Though this approach might still work in the queer case, we provide below a different approach.

We first introduce a PBW type basis for $\Uvqn$. Following \cite[Th.~6.2]{Ol} and \cite[Rem.5.6, Lem. 5.7]{DW1},
for {$1 \le i \le n-1$}, set\footnote{Here  {${\RZ}_{i,j}$}  is in fact  the quantum root vector which is denoted by {${{X}}_{i,j}$} in \cite{DW1}.}
\begin{align*}
{\RZ}_{i, i+1} = {\genE}_{i}, \quad
{\RZ}_{i+1, i} = {\genF}_{i}, \quad
{\ol{\RZ}}_{i, i+1} = \ol{\genE}_{{i}}= {\genE}_{\ol{i}}, \quad
{\ol{\RZ}}_{i+1, i} =\ol {\genF}_{{i}}= {\genF}_{\ol{i}},\quad \ol{\sfK}_a=\sfK_{\bar a}
\end{align*}
for {$|j-i| > 1$} and then introduce the following quantum root vectors:
\begin{align*}
{\RZ}_{i, j} =
\left\{
\begin{aligned}
&{\RZ}_{i, k} {\RZ}_{k, j} - {v}  {\RZ}_{k, j} {\RZ}_{i, k},   &\mbox{ if } i<j, \\
&{\RZ}_{i, k} {\RZ}_{k, j} - {v}^{-1}  {\RZ}_{k, j} {\RZ}_{i, k},   &\mbox{ if } i>j,
\end{aligned}
\right.
\quad
{\ol{\RZ}}_{i, j} =
\left\{
\begin{aligned}
&{\RZ}_{i, k} {\ol{\RZ}}_{k, j} - {v}  {\ol{\RZ}}_{k, j} {\RZ}_{i, k},  &\mbox{ if } i<j, \\
&{\ol{\RZ}}_{i, k} {{\RZ}}_{k, j} - {v}^{-1}  {{\RZ}}_{k, j} {\ol{\RZ}}_{i, k},  &\mbox{ if } i>j,
\end{aligned}
\right.
\end{align*}
where {$k$} is any number strictly between {$i$} and {$j$} and it is known that {${\RZ}_{i, j}$} and {${\ol{\RZ}}_{i, j}$} does not depend on the choice of $k$.

Set
\begin{equation}\label{eq:set-q-roots}
\aligned
\widetilde{\fsG}_n&=\{\mathsf{K}_a,\mathsf{K}_a^{-1}, \ol{\mathsf{K}}_{a}, \mathsf{E}_{j},\mathsf{E}_{\bar{j}},\mathsf{F}_{j},\mathsf{F}_{\bar{j}}|1\leq a\leq n, 1\leq j\leq n-1\},\\
\scrR_n&=\{\mathsf{K}_a,\mathsf{K}_{a}^{-1},\ol{\mathsf{K}}_a, \mathsf{E}_{i,j},\mathsf{E}_{j,i},\ol{\mathsf{E}}_{i,j},\ol{\mathsf{E}}_{j,i}|1\leq a\leq n, 1\leq i<j\leq n\}.
\endaligned
\end{equation}
 Then $\widetilde{\fsG}_n\subseteq \scrR_n$ and $\widetilde{\fsG}_n$ is the inverse of \eqref{setG} satisfying  $\bs\xi_n(\widetilde{\fsG}_n)=\mathscr G_n$.

For any {$\Ad=(A^{\bar0}|A^{\bar1}) =(a_{i,j}^{\bar0}|a_{i,j}^{\bar1})\in \MNZNS(n)$},
let
\begin{align}\label{PBWbs}
{\ttb}^{\Ad,\bf{0}}
={\prod_{{1} \le j \le {n-1}}} \
	 \Big({\prod_{{j+1}\le i\le n}}  {\RZ}_{i,j}^{{\SEE{a}_{i,j}}} {\ol{\RZ}}_{i,j}^{\SOE{a}_{i,j}}\Big){
	\cdot \Big(\prod_{1\leq j\leq n}\ol{\genK}_{j}^{\SOE{a}_{j,j}}\Big)\cdot {\prod_{{n} \ge j \ge {1}}}
			\big({\prod_{{j-1} \ge i \ge {1}}} {\RZ}_{i,j}^{{\SEE{a}_{i,j}}} {\ol{\RZ}}_{i,j}^{\SOE{a}_{i,j}}\big)} .
\end{align}

By the matrix decomposition in \eqref{pzm}, let
$$A^\star_-=(A^{\bar0}_-|A^{\bar1}_-),\quad A^\star_0=(O|A^{\bar1}_0),\quad A^\star_+=(A^{\bar0}_+|A^{\bar1}_+),$$ then we may rewrite
${\ttb}^{\Ad,\bf{0}}={\ttb}^{\Ad,\bf{0}}_-\cdot {\ttb}^{\Ad,\bf{0}}_0 \cdot{\ttb}^{\Ad,\bf{0}}_+$, where ${\ttb}^{\Ad,\bf{0}}_\bullet$ is the product associated with $A^\star_\bullet$ for all $\bullet\in\{-,0,+\}$.

For {$ \bs{j}=(j_i) \in \ZZ^n$} and $\Ad$ as above, define ${\ttb}^{\Ad,\bs{j}}=\prod_{i=1}^n {\genK}_{i}^{j_i} \cdot{\ttb}^{\Ad,\bf{0}}$.
These elements form a PBW type basis for $\Uvqn$.

\begin{prop}[{\cite{Ol},\cite[Prop. 5.8]{DW1}}]\label{pbw_basis}
The set
$$ \mathfrak{B}= \{ \ttb^{\Ad,\bs{j}}
	\where
	\Ad \in \MNZNS(n), \ \bs{j} \in \ZZ^n
	\},
$$
containing $\scrR_n$, forms a {$\Qv$}-basis for {$\Uvqn$}.
\end{prop}

Let $\bs U_v^+$ (resp. $\bs U_v^-$) be the subalgebra of {$\Uvqn$}
 generated by $\mathsf{E}_j, \mathsf{E}_{\bar{j}}$ (resp. $\mathsf{F}_j, \mathsf{F}_{\bar{j}}$ ) with $1\leq j\leq n-1$
 and let $\bs U_v^0$ be the subalgebra of {$\Uvqn$} generated by ${\genK}_i^{\pm},  {\genK}_{\bar{i}}$ with $1\leq i\leq n$.
 Then the proposition above implies that $\Uvqn$ admits the following triangular decomposition (see also \cite[Th.~ 2.3]{GJKK}):
\begin{equation}\label{lem:tri-decom}
\Uvqn=\bs U_v^- \bs U_v^0 \bs U_v^+\cong \bs U_v^-\otimes \bs U_v^0\otimes \bs U_v^+,
\end{equation}
and $\bs U_v^-,\bs U^0_v, \bs U_v^+$ have bases, respectively,
\begin{equation}\label{pmpbw}
\{\ttb^{\Ad,\bf0}\mid \Ad\in\Xi^-_n\},\quad\{\ttb^{\Ad,\bs j}\mid \Ad\in\Xi^0_n,\bs j\in\ZZ^n\},\quad\{\ttb^{\Ad,\bf0}\mid \Ad\in\Xi_n^+\},
\end{equation}
 where $\Xi_n^-$ (resp., $\Xi_n^0$, $\Xi_n^+$) is the subset of $\MNZNS(n)$ consisting of $\Ad=(A^{\bar0}|A^{\bar1})$ such that both $A^{\bar0}$ and $A^{\bar1}$ are lower triangular (resp., diagonal, upper triangular) (cf. \eqref{pzm}).
Note that we have the decomposition
\begin{equation}\label{eq:matrix-decomp}
{ \MNZNS(n)=\Xi_n^-+\Xi_n^0+\Xi_n^+.}
\end{equation}

Let $\bs U^0_\sfK$ (resp., $\bs U^0_{\ol{\sfK}}$) be the subalgebra generated by $\sfK_j$ (resp., $\ol{\sfK}_j$) for all $j\in[1,n]$.
There is a $\bs U^0_\sfK$-module structure on $\Uvqn$ defined by conjugation action  (see \cite[p.~837]{GJKK}):
$$\sfK_i*u:=\sfK_i u\sfK_i^{-1},\;\;\text{for all }u\in\Uvqn.$$
We now describe its weight subspaces.

Recall that the queer Lie superalgebra $\mathfrak{q}(n)$ admits a Cartan subalgebra $\mathfrak h$ (see \cite[\S1.2.6]{CW}). Let $\mathfrak h_{\bar 0}$ be the even part of $\mathfrak h$ with $h_i,\;i\in[1,n],$ as the standard basis.\footnote{Using the notation in \eqref{Mnpm}, $h_i=D_i^{\star\!\!\!\!\square}$ for $D_i=(E_{i,i}|O)$.} Let $\ep_i$ be its dual basis defined by $\ep_i(h_j)=\delta_{i,j}$ and define  roots
 $\alpha_{i,j}=\ep_i-\ep_{j}$ for $1\leq i\neq j\leq n$. Then, we obtain the root lattice $\mathsf{Q}=\sum_{j=1}^{n-1}\mathbb{Z}\alpha_j$, where $\alpha_j=\alpha_{j,j+1}$.
Let $\mathsf{Q}^{+}=\sum_{j=1}^{n-1}\mathbb{Z}^+\alpha_j$ and $\mathsf{Q}^{-}=-\mathsf{Q}^{+}$.

For any $\alpha\in\mathsf Q$, define the $\alpha$-{\it weight space}
$$\Uvqn_\alpha=\{u\in\Uvqn\mid \sfK_i*u:=\sfK_i u\sfK_i^{-1}=\up^{\alpha(h_i)}u\},$$
and write $\wt(u)=\alpha$ if $0\neq u\in\Uvqn_\alpha$.
Then, by Definition \ref{defqn}, we have
\begin{equation}\label{eq:degree}
\mathsf{K}^{\pm}_i,\mathsf{K}_{\bar{i}}\in\Uvqn_0,\quad\mathsf{E}_j,\mathsf{E}_{\bar j}\in\Uvqn_{\alpha_j}, \quad \mathsf{F}_j,\mathsf{F}_{\bar j}\in\Uvqn_{-\alpha_j}\;\; ( i\in[1, n], j\in[1,n)).
\end{equation}
Thus, an induction on $|i-j|$ shows that $\wt(\sfE_{i,j})=\alpha_{i,j}=\wt(\ol{\sfE}_{i,j})$. Hence, by \eqref{PBWbs},
\begin{equation}\label{wtB}
\wt(\Ad):=\wt(\ttb^{\Ad,\bs j})=\wt(\ttb^{\Ad,\bf0})=\sum_{1\leq i, j\leq n}a_{i,j}\alpha_{i,j},
\end{equation}
where $A=(a_{i,j})$ is the base of $\Ad$.
Hence, we have $\Uvqn=\bigoplus_{\alpha\in\mathsf Q}\Uvqn_\alpha$. Similarly, we have weight space decompositions
\begin{equation}\label{+wtsp}
\bs U_\up^+=\bigoplus_{\alpha\in\mathsf Q^+}(\bs U_\up^+)_\alpha\;\text{ and }\;\bs U_\up^-=\bigoplus_{\alpha\in\mathsf Q^-}(\bs U_\up^-)_\alpha.
\end{equation}
By \eqref{pmpbw}, we have immediately the following
\begin{cor} \label{wtspbs}
For $\ep\in\{+,-\}$ and $\alpha\in\mathsf Q^\ep$, each weight subspace $(\bs U_\up^\ep)_\alpha$ has a basis
{\rm\begin{equation*}
\{\ttb^{\Ad,\bf0}\mid \Ad\in\Xi^\ep_n,\texttt{wt}(\Ad)=\alpha\}.
\end{equation*}}
\end{cor}

We are ready to introduce a monomial basis for $\Uvqn$.
For any {$\Ad=(a_{i,j}^{\bar0}|a_{i,j}^{\bar1})\in \MNZNS(n)$} with base $A=(a_{i,j})$ and  {$ \bs{j}=(j_i) \in \ZZ^n$}, we use the same product order in \eqref{MA} to define
\begin{equation}\label{eq:mono-queer}
\aligned
{\ttm}^{\Ad, \bs{j}}&=\prod_{i=1}^n {\genK}_{i}^{j_i} \cdot {\prod_{{1} \le j \le {n-1}}} \Big({\prod_{j+1 \le i \le n}} \sfM^{\Ad}_{i,j}\Big)
	\cdot\Big(\prod_{1\leq j\leq n} \sfM^{\Ad}_{j,j}\Big)\cdot
	{\prod_{{n} \ge j \ge {1}}}  \Big ( {\prod_{{j-1} \ge i \ge {1}}} \sfM^{\Ad}_{i,j} \Big ),\;\;\text{where}\\
\ttm^\Ad_{i,j}&=
\begin{cases}
{\ol{\mathsf{E}}_{i}}^{\SOE{a}_{i,j} }
	 \cdot  \mathsf{E}_{i}^{(\SEE{a}_{i,j})}\mathsf{E}_{i+1}^{(a_{i,j})}\cdots \mathsf{E}_{j-1}^{(a_{i,j})}  ,
	 \qquad &\mbox{for all } 1 \le i < j \le n, \\
  \ol{\mathsf{K}}_{i}^{\SOE{a}_{i,i}} ,
	\qquad & \mbox{for all }   1 \le i   \le n, \\
{\ol{\mathsf{F}}_{i-1}}^{\SOE{a}_{i,j} }
	 \cdot  \mathsf{F}_{i-1}^{(\SEE{a}_{i,j})}\mathsf{F}_{i-2}^{(a_{i,j})}\cdots \mathsf{F}_{j}^{(a_{i,j})}  ,
	 \qquad  & \mbox{for all } 1 \le j < i \le n.
\end{cases}
\endaligned
\end{equation}

\begin{thm}\label{injective}The set
$$\fkM=\{ {\mathsf{M}}^{\Ad, \bs{j}}
	\where
	\Ad \in \MNZNS(n), \ \bs{j} \in \ZZ^n
	\}$$
 forms a basis for $\Uvqn$. In particular,
the superalgebra epimorphism $\bs\xi_n:\Uvqn\to\USnv$ given in
Theorem \ref{triangular_relation_q}(3) is an isomorphism.
\end{thm}
\begin{proof}
By definition, we have
$\bs{\xi}_{n}({\mathsf{M}}^{\Ad, \bs{j}})= {\scm}^{\Ad, \bs{j}}$.
Thus, by Theorem \ref{triangular_relation_q}(2), the set $\{ {\mathsf{M}}^{\Ad, \bs{j}}
	\where
	\Ad \in \MNZNS(n), \ \bs{j} \in \ZZ^n
	\}
$ is linearly independent.

The monomial defined in \eqref{eq:mono-queer} are also weight vectors. More precisely, by \eqref{eq:degree}, we have the following weight formulas: for $1\leq i<j\leq n$,
$$
{\wt}(\ttm^\Ad_{i,j})=a_{i,j}^{\bar{1}}(\alpha_i+\alpha_{i+1}+\cdots+\alpha_{j-1})+a_{i,j}^{\bar{0}}(\alpha_i+\alpha_{i+1}+\cdots+\alpha_{j-1})
=a_{i,j}\alpha_{i,j}.
$$
Clearly, same formula holds for $i>j$. Thus, if $\Ad\in\Xi_n^+$, then
$$
\wt({\ttm}^{\Ad, \bs{0}})=\sum_{j=1}^n\sum_{i=1}^{j-1}\wt(\mathsf{M}^\Ad_{i,j})=\sum_{j=1}^n\sum_{i=1}^{j-1}a_{i,j}\alpha_{i,j}=\sum_{1\leq i<j\leq n}a_{i,j}\alpha_{i,j}=\wt(\Ad).
$$
Hence, $(\bs U_\up^+)_\alpha$ contains $\{\sfM^{\Ad,\bf0}\mid \Ad\in\Xi^+_n,\wt(\Ad)=\alpha\}$ which is linearly independent. Consequently, a dimensional comparison with Corollary \ref{wtspbs}, this set forms a basis for the weight space. This shows that the set
$\{\sfM^{\Ad,\bf0}\mid \Ad\in\Xi^+_n\}$ forms a second basis for $\bs U_\up^+$.

Similarly, one proves that the set
$\{\sfM^{\Ad,\bf0}\mid \Ad\in\Xi^-_n\}$ forms a second basis for $\bs U_\up^-$ and clearly,
$\bs U^0_v$ has basis
$
\{\prod_{i=1}^n\mathsf{K}_{i}^{j_i}\prod_{1\leq i\leq n}\ol{\mathsf{K}}_{i}^{t_i}\mid \bs{j}\in\mathbb{Z},\bs{t}\in\mathbb{N}_2^n \}.$
Altogether, by \eqref{lem:tri-decom}, \eqref{eq:matrix-decomp}, and \eqref{eq:mono-queer}, we conclude that $\mathfrak M$ forms a basis for $\Uvqn$ which forces the map $\bs \xi_n$ is injective.
\end{proof}

\noindent{\bf Completing the proof of the main Theorem \ref{mthm}. }
By Theorems \ref{triangular_relation_q} and \ref{injective}, the map $\bs{\xi}_n:\Uvqn\rightarrow\USnv$ is a superalgebra isomorphism. Thus, we identify the two superalgebras. In particular, under this identification, $\Uvqn$ has three bases:
\begin{enumerate}
\item the PBW type basis $\mathfrak B$ given in Proposition \ref{pbw_basis};
\item the BLM type basis $\mathfrak L$ given in Lemma \ref{unv_basis}, and
\item the monomial basis $\mathfrak M$ given in Theorem \ref{injective}.
\end{enumerate}
So, the basis assertion follows.

Finally, the multiplication formula in (1) follows from Proposition \ref{mulformzerocor},
those in (2) and (3) follow from Proposition \ref{mulformeven-2}, and the last one in (4) follows from Proposition \ref{mulformdiag}, the $h=n$ case. \qed

\vspace{0.3cm}
{\it Acknowledgement}:
The work was  partially supported by the UNSW Science FRG
and the Natural Science Foundation of China (\#12071129, \#12471121, \#11871404, \#12122101, \#12071026).
The first author acknowledges the support from UNSW Science FRG.
The second author  acknowledges the support from NSFC-12071129, NSFC-12471121, CSC and ME Key Laboratory of MFA.
The third author would like to thank the support from Professor Yanan Lin in Xiamen University,
and  NSFC-11871404.
The fourth author is supported by NSFC-12122101 and NSFC-12071026.

\end{document}